\documentclass[reqno]{amsart}
\usepackage{amsfonts}
\usepackage{amssymb}
\numberwithin{equation}{section}
\theoremstyle{plain}
 \newtheorem{thm}{Theorem}[section]
 \newtheorem{lem}[thm]{Lemma}
 \newtheorem{cor}[thm]{Corollary}
 \newtheorem{prop}[thm]{Proposition}
\theoremstyle{definition}
 \newtheorem{defn}[thm]{Definition}
\theoremstyle{remark}
 \newtheorem{rem}[thm]{Remark}
 \newtheorem{ex}[thm]{Example}

\begin{document}
\begin{flushleft}
{\Large {\bf Transformations of infinitely divisible
distributions via improper stochastic integrals}}

\bigskip
{\bf Ken-iti Sato}

\bigskip
{\small Hachiman-yama 1101-5-103, Tenpaku-ku, Nagoya, 468-0074 Japan\\
{\it E-mail address}: ken-iti.sato@nifty.ne.jp\\
{\it URL}: http://ksato.jp/}
\end{flushleft}

\bigskip
\noindent {\bf Abstract.} 
Let $X^{(\mu)}(ds)$ be an $\mathbb{R}^d$-valued homogeneous 
independently scattered
random measure over $\mathbb{R}$ having $\mu$ as the distribution of 
$X^{(\mu)}((t,t+1])$.  Let $f(s)$ be a nonrandom measurable function on 
an open interval $(a,b)$ where $-\infty\leqslant a<b\leqslant\infty$.  
The improper stochastic integral $\int_{a+}^{b-} f(s)X^{(\mu)}(ds)$ is
studied.  Its distribution $\Phi_f(\mu)$ defines a mapping from $\mu$ to an
infinitely divisible distribution on $\mathbb{R}^d$.  
Three modifications (compensated, essential, and symmetrized) and 
absolute definability are considered.
After their domains are characterized, necessary and sufficient
 conditions for the domains
to be very large (or very small) in various senses
 are given.  
The concept of the dual in the class of purely non-Gaussian infinitely 
divisible distributions on $\mathbb{R}^d$ is introduced and employed in
studying some examples.
The $\tau$-measure $\tau$
of function $f$ is introduced and whether $\tau$ determines $\Phi_f$
is discussed. Related transformations of L\'evy measures are also studied.

\medskip
\noindent {\small 2000 {\it Mathematics Subject Classification}. 60G51,
60H05, 60E07.}

\bigskip
\section{Introduction}

In Sato (2006a,b,c) the improper stochastic integral
$\int_0^{\infty-} f(s)dX_s^{(\mu)}$, or\linebreak
$\int_0^{\infty-} f(s)X^{(\mu)}(ds)$, is studied.  Here $X_s^{(\mu)}$
is a L\'evy process on $\mathbb{R}^d$ with distribution $\mu$ at time $1$
and $X^{(\mu)}(ds)$ is the homogeneous independently scattered
random measure associated with $X_s^{(\mu)}$, $f(s)$ is a nonrandom
function locally $X^{(\mu)}$-integrable on the closed half line
$[0,\infty)$, and $\int_0^{\infty-}$ 
is the limit in probability of $\int_0^t$ as $t\to\infty$.
Let $\mathcal L(Y)$ denote the distribution of a random element $Y$.
Let $ID(\mathbb{R}^d)$ denote the class of infinitely divisible distributions
on $\mathbb{R}^d$.
Given a function $f$, we define a mapping $\Phi_f$ from
a subclass of $ID(\mathbb{R}^d)$ into $ID(\mathbb{R}^d)$ by
$\Phi_f(\mu)=\mathcal L\left(\int_0^{\infty-} f(s)X^{(\mu)}(ds)\right)$.
Such a mapping appears in many papers. With $f(s)=e^{-s}$, it appears
in the representation of selfdecomposable distributions (see 
Rocha-Arteaga and Sato (2003)
for references); with $f(s)=\log(1/s)$ for $s\in(0,1)$,
it appears
in the representation of the Goldie--Steutel--Bondesson class on $\mathbb{R}^d$
(see Barndorff-Nielsen and Thorbj\o rnsen (2002, 2006b) and Barndorff-Nielsen,
Maejima, and Sato (2006)). In Sato (2006b)
a family of functions $f_{\alpha}(s)$ with $\alpha\geqslant0$ such that, as 
$s\downarrow0$,
$f_{\alpha}(s)\sim \log(1/s)$ for all $\alpha$ and, as $s\to\infty$, 
$f_{\alpha}(s)\sim ce^{-s}$ with a constant $c>0$ for $\alpha=0$ and
$f_{\alpha}(s)\sim (\alpha s)^{1/\alpha}$ for $\alpha>0$, is studied.
In Aoyama and Maejima (2007) the function $f(s)$ for $0<s<1$ defined by
$s=\int_{f(s)}^{\infty} (2\pi)^{-1/2} e^{-v^2/2} dv$ (hence 
$f(s)\to\infty$ as $s\downarrow0$ and $f(s)\to-\infty$ as $s\uparrow1$) is
utilized in the representation of the class of type $G$ distributions
on $\mathbb{R}^d$.  In the case of the last example and in the case
$f(s)=\log(1/s)$ for $s\in(0,1)$, we define $f(s)$ to be $0$ for
$s\in\{0\}\cup [1,\infty)$.  Then the functions in all these examples
are locally $X^{(\mu)}$-integrable on the closed half line
$[0,\infty)$ for all $\mu\in ID(\mathbb{R}^d)$, so that they are in the
framework of Sato (2006a,b,c).
However, it would be more natural to consider an open interval $(a,b)$
with $-\infty\leqslant a<b\leqslant\infty$ and a function $f(s)$ 
locally $X^{(\mu)}$-integrable on $(a,b)$ and study
improper stochastic integrals $\int_{a+}^{b-} f(s)X^{(\mu)}(ds)$,
the limit
in probability of $\int_p^q f(s)X^{(\mu)}(ds)$ as $p\downarrow a$ and
$q\uparrow b$.  In this paper we carry out the study of
improper stochastic integrals of this type. 
Examples such as $\int_{0+}^{\pi-} (\sin s)^{-1}X^{(\mu)}(ds)$
and $\int_{0+}^{\pi-} (\cot s)X^{(\mu)}(ds)$ are in our mind;
for some $\mu$ these integrands extended to $[0,\infty)$ 
with the value 0 outside of $(0,\pi)$ are not
locally $X^{(\mu)}$-integrable on $[0,\infty)$.
The improper integrals $\int_{-\infty}^t e^{s-t} X^{(\mu)}(ds)$, 
$t\in\mathbb{R}$, for stationary Ornstein--Uhlenbeck type processes
in Maejima and Sato (2003) and $\int_{-\infty}^{\infty}({(t-s)_+}^{\alpha}-
{(-s)_+}^{\alpha})X^{(\mu)}(ds)$ ($0<\alpha<1/2$, $u_+=u\lor0$), 
$t\in\mathbb{R}$, for fractional L\'evy processes in Marquardt (2006)
are also included in our framework.

In Sections 2--4 we study improper integrals 
$\int_{a+}^{b-} f(s)X^{(\mu)}(ds)$ in general. 
It is largely parallel to the
treatment of $\int_0^{\infty-} f(s)X^{(\mu)}(ds)$.
Three modifications, that is, compensated, essential, and
symmetrized improper integrals, are defined in addition to the
usual improper integrals.  They induce transformations of 
infinitely divisible distributions defined by
{\allowdisplaybreaks
\begin{align*}
\Phi_f(\mu)&=\mathcal L\left(\int_{a+}^{b-} f(s)X^{(\mu)}(ds)\right),\\
\Phi_{f,\,\mathrm{c}}(\mu)&=\text{the set of distributions of the 
compensated limits of $\int_p^q f(s)X^{(\mu)}(ds)$}\\
&\quad\;\text{as $p\downarrow a$ and $q\uparrow b$},\\
\Phi_{f,\,\mathrm{es}}(\mu)&=\text{the set of distributions of the 
essential limits of $\int_p^q f(s)X^{(\mu)}(ds)$}\\
&\quad\;\text{as $p\downarrow a$ and $q\uparrow b$},\\
\Phi_{f,\,\mathrm{sym}}(\mu)&=\mathcal L\left(\text{symmetrized limit of 
$\int_p^q f(s)X^{(\mu)}(ds)$ as $p\downarrow a$ and $q\uparrow b$}\right).
\end{align*}
We describe the L\'evy--Khintchine triplets of these images and the domains}
$\mathfrak D (\Phi_f)$, $\mathfrak D (\Phi_{f,\,\mathrm{c}})$, 
$\mathfrak D (\Phi_{f,\,\mathrm{es}})$, and $\mathfrak D 
(\Phi_{f,\,\mathrm{sym}})$.
Some distributions $\mu$ in $\mathfrak D (\Phi_f)$ satisfy a 
condition called absolute definability of $\int_{a+}^{b-} f(s)X^{(\mu)}(ds)$.
Let
\[
\mathfrak D^0 (\Phi_f)=\left\{\mu\colon\text{$\int_{a+}^{b-} f(s)X^{(\mu)}(ds)$ 
is absolutely definable}\right\}.
\]
In general
\[
\mathfrak D^0 (\Phi_f)\subset \mathfrak D (\Phi_f)\subset 
\mathfrak D (\Phi_{f,\,\mathrm{c}})\subset \mathfrak D (\Phi_{f,\,\mathrm{es}})
=\mathfrak D (\Phi_{f,\,\mathrm{sym}}).
\]
Among these, $\mathfrak D^0 (\Phi_f)$ and $\mathfrak D (\Phi_{f,\,\mathrm{es}})$
 play especially important roles in further study.
The relation of these definitions to the previous ones in
Sato (2006a,b,c) will be given in Remarks \ref{r4.1} and
\ref{r5.2}.

Sections 5--9 deal with special problems.
In Section 5 we introduce the concept of the dual in the class of
purely non-Gaussian infinitely divisible distributions on $\mathbb{R}^d$,
and study some improper stochastic integrals on a finite interval
by transforming them to those on an infinite interval with respect
to the L\'evy process with the dual distribution at time $1$.  
We call $\mu'\in ID(\mathbb{R}^d)$ with
L\'evy--Khintchine triplet $(0,\nu',\gamma')$ the dual of 
$\mu\in ID(\mathbb{R}^d)$ with $(0,\nu,\gamma)$ if
\[
\nu'(B)=\int_{\mathbb{R}^d\setminus\{0\}} 1_B(\iota(x))|x|^2\nu(dx)
\]
and $\gamma'=-\gamma$, where $\iota(x)=|x|^{-2}x$, the inversion of $x$.

In Section 6 we seek conditions in order that the domains are very
large. Specifically, the condition for the domains being the whole
class $ID(\mathbb{R}^d)$ is given. This amplifies some results
in Barndorff-Nielsen and P\'erez-Abreu (2005).
Other conditions such as for
$ID_{\mathrm{AB}}(\mathbb{R}^d)\subset\mathfrak D (\Phi_{f,\mathrm{es}})$ and for
$ID_{\mathrm{AB}}(\mathbb{R}^d)\subset\mathfrak D^0 (\Phi_f)$ are given 
and the relations of those conditions are discussed (see below
for the definition of $ID_{\mathrm{AB}}(\mathbb{R}^d)$). 
Conditions in order that the domains are very small are also
considered.

For a real-valued measurable function $f$ on $(a,b)$ we introduce
the measure $\tau$ given by
\[
\tau(B)=\int_a^b 1_{\{f(s)\in B\}}ds,
\]
and call it the $\tau$-measure of $f$.  In Section 7 we study whether
$\tau$ determines $\mathfrak D (\Phi_f)$ and its variants and $\Phi_f(\mu)$
and its variants.
Roughly speaking, the answer is yes for 
$\mathfrak D^0 (\Phi_f)$ and $\mathfrak D (\Phi_{f,\,\mathrm{es}})$, but no for
$\mathfrak D (\Phi_f)$.
Further, under some conditions including decrease of $f$, we address
the problem whether $\tau$ determines $f$.
The $\tau$-measure is a development of ideas in Aoyama and Maejima
(2007) and Barndorff-Nielsen and Thorbj\o rnsen (2006a,b).

In one dimension the class of infinitely divisible distributions
concentrated on $[0,\infty)$ is important in theory and applications.
Its multivariate analogue is given with $[0,\infty)$ replaced by 
a proper cone $K$ in $\mathbb{R}^d$.  Some results in Section 6 
paraphrased for such distributions are given in Section 8.

In the transformation from $\mu\in ID(\mathbb{R}^d)$ to $\widetilde\mu=
\Phi_f(\mu)\in ID(\mathbb{R}^d)$ their L\'evy measures $\nu$ and 
$\widetilde\nu$ are
related as
\begin{equation}\label{0.1}
\widetilde\nu(B)=\int_a^b ds\int_{\mathbb{R}^d} 1_B(f(s)x)\,\nu(dx)
=\int_{\mathbb{R}\setminus\{0\}} \nu\left(\frac{1}{u} B\right)\tau(du)
\end{equation}
for Borel sets $B$ in $\mathbb{R}^d\setminus\{0\}$, where $\tau$ is the
$\tau$-measure of $f$.  Defining $\Psi_f(\nu)$ by $\Psi_f(\nu)=\widetilde\nu$,
we study in Section 9 the relation between $\Psi_f$ and 
$\Phi_{f,\,\mathrm{es}}$, and give counterparts of some results in 
Sections 6 and 8.
The transformation of the form of the right extreme of 
\eqref{0.1} is introduced by Maejima and
Rosi\'nski (2002) for the standard Gaussian distribution
$\tau$ and by Barndorff-Nielsen and P\'erez-Abreu (2005, 2007)
and Barndorff-Nielsen and Thorbj\o rnsen (2006a,b)
for measures $\tau$ on $(0,\infty)$.

Let us prepare some general concepts used in this paper.
Let $\widehat\mu(z)$, $z\in\mathbb{R}^d$, be the characteristic function of $\mu$.  
Let $C_{\mu}(z)$ be the cumulant function of $\mu\in ID(\mathbb{R}^d)$. 
That is, $C_{\mu}(z)$ is
the unique
complex-valued continuous function on $\mathbb{R}^d$ satisfying $C_{\mu}(0)=0$ 
and $\widehat\mu(z)=e^{C_{\mu}(z)}$.
Sometimes we write $C_X (z)=C_{\mu}(z)$, using an $\mathbb{R}^d$-valued random 
variable $X$ with $\mathcal L (X)=\mu$. 
We use the L\'evy--Khintchine triplet $(A,\nu,\gamma)$
of $\mu$ in the form
\begin{equation}\label{1.1}
C_{\mu}(z)= -\frac12 \langle z,Az\rangle +\int_{\mathbb{R}^d}\left(
e^{i\langle z,x\rangle}-1-\frac{i\langle z,x\rangle}{1+|x|^2}\right)\nu(dx)
+i\langle\gamma,z\rangle,
\end{equation}
where $A$ is a $d\times d$ symmetric nonnegative-definite matrix, 
$\nu$ is a
measure on $\mathbb{R}^d$ satisfying $\nu(\{0\})=0$ and $\int_{\mathbb{R}^d}
(|x|^2\land1)\nu(dx)
<\infty$, and $\gamma$ is an element 
of $\mathbb{R}^d$. We have one-to-one correspondence between $\mu$ and 
$(A,\nu,\gamma)$.
Let $\mu=\mu_{(A,\nu,\gamma)}$ denote the distribution corresponding to 
$(A,\nu,\gamma)$. The measure $\nu$ is called the L\'evy measure of $\mu$.
The distribution $\mu\in ID(\mathbb{R}^d)$ with $A=0$ is called purely non-Gaussian.
Let $ID_0(\mathbb{R}^d)$ denote the class of purely non-Gaussian infinitely 
divisible distributions on $\mathbb{R}^d$. 
Following Sato (1999), we call $\mu=\mu_{(A,\nu,\gamma)}$
{\allowdisplaybreaks
\begin{align*}
&\text{of type A if $A=0$ and $\nu(\mathbb{R}^d)<\infty$,}\\
&\text{of type B if $A=0$, $\nu(\mathbb{R}^d)=\infty$, and $\int_{|x|\leqslant1}|x|
\nu(dx)<\infty$,}\\
&\text{of type C if $A\not=0$ or $\int_{|x|\leqslant1}|x|\nu(dx)=\infty$.}
\end{align*}
The class of} $\mu\in ID(\mathbb{R}^d)$ of type A or B is denoted by 
$ID_{\mathrm{AB}}(\mathbb{R}^d)$.
Sometimes we omit $\mathbb{R}^d$ in the notation $ID(\mathbb{R}^d)$, 
$ID_0(\mathbb{R}^d)$, and
$ID_{\mathrm{AB}}(\mathbb{R}^d)$.

If $\mu=\mu_{(A,\mu,\gamma)}$ in $ID(\mathbb{R}^d)$ satisfies $\int_{|x|\leqslant1}
|x|\nu(dx)<\infty$, then there is a unique $\gamma^0\in \mathbb{R}^d$ such that
\begin{equation}\label{7.1}
C_{\mu}(z)=-\frac12 \langle z,Az\rangle +\int_{\mathbb{R}^d}(e^{i\langle z,
x\rangle}-1)\,\nu
(dx)+i\langle\gamma^0,a\rangle,\qquad z\in\mathbb{R}^d.
\end{equation}
We express this fact by saying that $\mu$ has triplet $(A,\nu,\gamma^0)_0$ 
(see Remark 8.4 of Sato (1999)). This $\gamma^0$ is called the drift of $\mu$.

Let $\mathrm{Lvm}(ID(\mathbb{R}^d))$ denote the class of L\'evy measures of 
distributions in $ID(\mathbb{R}^d)$, that is, $\mathrm{Lvm}(ID(\mathbb{R}^d))$ is
the class of measures $\nu$ on $\mathbb{R}^d$
satisfying $\nu(\{0\})=0$ and $\int_{\mathbb{R}^d}(|x|^2\land1)\nu(dx)
<\infty$. Similarly, let
$\mathrm{Lvm}(ID_{\mathrm{AB}}(\mathbb{R}^d))$ denote the class of L\'evy measures of 
distributions in $ID_{\mathrm{AB}}(\mathbb{R}^d)$, that is, 
$\mathrm{Lvm}(ID_{\mathrm{AB}}(\mathbb{R}^d))$ is
the class of measures $\nu$ on $\mathbb{R}^d$
satisfying $\nu(\{0\})=0$ and $\int_{\mathbb{R}^d}(|x|\land1)\nu(dx)
<\infty$.

Equalities among random variables are always understood to be
almost surely.  
We use the words {\em decrease} and {\em increase} in the wide sense
allowing flatness.
When we say that a function is {\em real-valued} or {\em $\mathbb{R}$-valued}, the
values $\infty$ and $-\infty$ are not allowed.
The class of Borel sets in $\mathbb{R}^d$ is denoted by $\mathcal B(\mathbb{R}^d)$.
The class of bounded Borel sets in $\mathbb{R}$ is denoted by 
$\mathcal B^0_{\mathbb{R}}$.
The class of Borel sets $B$ in $(a,b)$ such that
$\displaystyle\inf_{x\in B}x>a$ and $\displaystyle\sup_{x\in B}x<b$ is denoted by
$\mathcal B^0_{(a,b)}$. We simply write (2006a), (2006b), and (2006c), 
indicating Sato's respective papers.

\medskip
Inspired by some results in Barndorff-Nielsen and P\'erez-Abreu (2005), 
the author sent three memos to a small circle in January 2005. 
Some theorems in Section 6 are 
developments of those memos.  This paper has grown up from that part.
The author thanks Makoto Maejima, V\'ictor P\'erez-Abreu, and Ole
Barndorff-Nielsen for having constant interest in this study 
and for informing him of their related results.
He also thanks Jan Rosi\'nski for giving him valuable remarks to the
manuscript.

\section{Stochastic integrals of nonrandom functions}

First we give definition and existence of homogeneous independently 
scattered random measures.

\begin{defn}\label{d2.1}
A class $X=\{ X(B)\colon B\in \mathcal B^0_{\mathbb{R}}\}$ of 
$\mathbb{R}^d$-valued random
variables is called an {\em independently scattered random measure} 
on $\mathbb{R}$ if
{\rm(1)} for any sequence $B_1, B_2, \ldots$ of disjoint sets in 
$\mathcal B_{\mathbb{R}}^0$
with $\bigcup_{n=1}^{\infty} B_n\in  \mathcal B_{\mathbb{R}}^0$, 
$\sum_{n=1}^{\infty}  X(B_n)$ converges a.\,s.\ and equals 
$ X(\bigcup_{n=1}^{\infty} B_n)$,
{\rm(2)} for any finite sequence $B_1, \ldots, B_n$ of disjoint sets in 
$\mathcal B_{\mathbb{R}}^0$, $X(B_1),\ldots,  X(B_n)$ are independent,
{\rm(3)}  $X(\{ a\})=0$ for every $a\in\mathbb{R}$.
It is called {\em homogeneous} if, in addition, $\mathcal L(X(B))=\mathcal L(X(B+t))$
 for all $B\in\mathcal B_{\mathbb{R}}^0$
and $t\in\mathbb{R}$. (It follows from (1) that $X(\emptyset)=0$.)
\end{defn}

If $X=\{ X(B)\colon B\in \mathcal B^0_{\mathbb{R}}\}$ is an $\mathbb{R}^d$-valued 
independently 
scattered random measure, then $\mathcal L(X(B))\in ID(\mathbb{R}^d)$ for all 
$B\in \mathcal B^0_{\mathbb{R}}$.

\begin{prop}\label{p2.0} 
For any $\mu\in ID(\mathbb{R}^d)$ there exists a unique (in law) $\mathbb{R}^d$-valued
homogeneous independently 
scattered random measure $X=\{ X(B)\colon B\in \mathcal B^0_{\mathbb{R}}\}$
on $\mathbb{R}$ such that $\mathcal L(X((t,t+1]))=\mu$ for all $t\in\mathbb{R}$.
\end{prop}

See Rajput and Rosinski (1989), Sato (2004),
 or Maejima and Sato (2003) for the proof.

\medskip
Fix $\mu=\mu_{(A,\nu,\gamma)}\in ID(\mathbb{R}^d)$ and let 
$X^{(\mu)}=\{X^{(\mu)}(B)\colon B\in\mathcal B^0_{\mathbb{R}}\}$ be an 
$\mathbb{R}^d$-valued
homogeneous independently 
scattered random measure such that $\mathcal L(X((t,t+1]))=\mu$.
Let $(a,b)$ be an open interval with $-\infty\leqslant a<b\leqslant\infty$.
The following definition of integrals with respect to $X^{(\mu)}$
is similar to that in Urbanik and Woyczy\'nski (1967), 
Rajput and Rosinski (1989), Kwapie\'n and Woyczy\'nski (1992), 
and Sato (2004, 2006a).

\begin{defn}\label{d3.1}
Call $f(s)$ a {\it simple function on $(a,b)$}, if $f(s)=\sum_{j=1}^n r_j
1_{B_j}(s)$ for some $n$, where $B_1,\ldots, B_n$ are disjoint Borel sets
in $(a,b)$ and $r_1,\ldots,r_n\in\mathbb{R}$. For a simple function $f(s)$ of this
form, define
\[
\int_B f(s)X^{(\mu)}(ds)=\sum_{j=1}^n r_j X^{(\mu)}(B\cap B_j)
\]
for $B\in\mathcal B^0_{(a,b)}$. 
An $\mathbb{R}$-valued measurable function $f(s)$ on $(a,b)$ is called
{\it locally $X^{(\mu)}$-integrable on $(a,b)$}, if there is a sequence of simple
functions $f_n$, $n=1,2,\ldots$, on $(a,b)$ such that $f_n(s)\to f(s)$ 
Lebesgue almost everywhere on $(a,b)$ as $n\to\infty$ and that, for
every $B\in\mathcal B^0_{(a,b)}$, the sequence $\int_B f_n(s)X^{(\mu)}(ds)$
converges in probability as $n\to\infty$. The limit is denoted by
$\int_B f(s)X^{(\mu)}(ds)$.
\end{defn}

Using the Nikod\'ym theorem, we
 can prove that if $f(s)$ is locally $X^{(\mu)}$-integrable on $(a,b)$, then, 
for every $B\in\mathcal B^0_{(a,b)}$, $\int_B f(s)X^{(\mu)}(ds)$ does not depend 
on the 
choice of the sequence of simple functions satisfying the conditions above.
If $f(s)$ is locally $X^{(\mu)}$-integrable on $(a,b)$, then, for
$p$ and $q$ satisfying $a<p<q<b$, $\int_{(p,q]}f(s)X^{(\mu)}(ds)$, 
$\int_{[p,q)}f(s)X^{(\mu)}(ds)$, 
$\int_{(p,q)}f(s)X^{(\mu)}(ds)$, and $\int_{[p,q]}f(s)X^{(\mu)}(ds)$ are identical
 almost surely;
they are denoted by $\int_p^qf(s)X^{(\mu)}(ds)$. 

\begin{defn}\label{d3.2}
Let $\mathbf L_{(a,b)}(X^{(\mu)})$ denote
 the class of locally $X^{(\mu)}$-integrable functions 
on $(a,b)$.
\end{defn}

\begin{prop}\label{p3.1}
If $f\in\mathbf L_{(a,b)}(X^{(\mu)})$, then
\begin{equation}\label{3.-1}
\int_q^p |C_{\mu}(f(s)z)|ds<\infty,\quad z\in\mathbb{R}^d
\qquad\text{for all $p,q$ with $a<p<q<b$}
\end{equation}
and
\begin{equation}\label{3.0}
C_{\int_B f(s)X^{(\mu)}(ds)} (z)=\int_B C_{\mu}(f(s)z)ds,
\quad z\in\mathbb{R}^d\qquad\text{for
$B\in\mathcal B_{(a,b)}^0$}.
\end{equation}
\end{prop}

See Proposition 2.17 of Sato (2004).

\begin{thm}\label{t3.1}
Let $f(s)$ be an $\mathbb{R}$-valued measurable function on $(a,b)$. 

{\rm(i)}  Suppose that $A\ne0$. Then $f\in\mathbf L_{(a,b)}(X^{(\mu)})$ if 
and only if
\begin{equation}
\int_p^q f(s)^2 ds<\infty\qquad\text{for all $p,q$ with $a<p<q<b$}.\label{3.1}
\end{equation}

{\rm(ii)}  Suppose that $A=0$. Then $f\in\mathbf L_{(a,b)}(X^{(\mu)})$ if and 
only if
{\allowdisplaybreaks
\begin{gather}
\int_p^q ds \int_{\mathbb{R}^d}(|f(s)x|^2\land 1)\,\nu(dx)<\infty\quad
\text{for all $p,q$ with $a<p<q<b$},\label{3.2}\\
\begin{split}
&\int_p^q \left|f(s)\gamma+\int_{\mathbb{R}^d} f(s)x\left(\frac{1}{1+|f(s)x|^2}
-\frac{1}{1+|x|^2}\right)\nu(dx)\right| ds<\infty\\
&\qquad\text{for all $p,q$ with $a<p<q<b$}.
\end{split}\label{3.3}
\end{gather}
}
\end{thm}

\begin{proof}
Let
\begin{equation}\label{3.4}
\begin{split}
\varphi(s,u)&=u^2\mathrm{tr}\, A+\int_{\mathbb{R}^d}(|ux|^2\land1)\nu(dx)\\
&\quad+\left|u\gamma+u\int_{\mathbb{R}^d}x\left(\frac{1}{1+|ux|^2}-\frac{1}{1+|x|^2}
\right)\nu(dx)\right|.
\end{split}
\end{equation}
Then $f\in\mathbf L_{(a,b)}(X^{(\mu)})$ if and only if
\begin{equation}\label{3.5}
\int_p^q \varphi(s,f(s))ds<\infty\qquad\text{for all $p,q$ with $a<p<q<b$}.
\end{equation}
See (2006a) for the proof. 
Property \eqref{3.5} is equivalent to saying that 
\begin{equation}\label{3.6}
\int_p^q f(s)^2\,\mathrm{tr}\, A\,ds<\infty\qquad\text{for all $p,q$ with $a<p<q<b$}
\end{equation}
together with \eqref{3.2} and \eqref{3.3}.

(i) Suppose that $A\ne0$. Since $\mathrm{tr}\, A>0$, \eqref{3.6} is equivalent to 
\eqref{3.1}. Since
\[
\int_p^q ds \int_{\mathbb{R}^d}(|f(s)x|^2\land 1)\,\nu(dx)\leqslant
\int_p^q (f(s)^2+ 1)ds \int_{\mathbb{R}^d}(|x|^2\land 1)\,\nu(dx),
\]
\eqref{3.2} follows from \eqref{3.1}. Further,
\[
\int_p^q |f(s)|ds\leqslant \left( (q-p)\int_p^q |f(s)|^2 ds\right)^{1/2}<\infty.
\]
Thus, \eqref{3.3} also follows from \eqref{3.1}, since we have
{\allowdisplaybreaks
\begin{align*}
&\int_p^q ds\left|\int_{\mathbb{R}^d} f(s)x\left(\frac{1}{1+|f(s)x|^2}
-\frac{1}{1+|x|^2}\right)\nu(dx)\right|\\
&\qquad\leqslant\int_p^q ds\left|\int_{\mathbb{R}^d} \frac{f(s)x(|x|^2+|f(s)x|^2)
\nu(dx)}
{(1+|f(s)x|^2)(1+|x|^2)} \right|
\leqslant\int_p^q\frac{1+f(s)^2}{2} ds \int_{\mathbb{R}^d} 
\frac{|x|^2\nu(dx)}{1+|x|^2}\\
&\qquad<\infty\quad\text{for all $p,q$ with $a<p<q<b$},
\end{align*}
using $r/(1+r^2)\leqslant 1/2$ for all $r\geqslant0$.} 

(ii) Suppose that $A=0$. If $f\in\mathbf L_{(a,b)}(X^{(\mu)})$, then
\eqref{3.2} and \eqref{3.3} hold, as we have seen above. 
Conversely, assume that \eqref{3.2} and \eqref{3.3} hold. Then
\eqref{3.6} is satisfied since $\mathrm{tr}\, A=0$. Hence $f\in\mathbf L_{(a,b)}
(X^{(\mu)})$.
\end{proof}

\begin{rem}\label{r3.1}
Suppose that $A=0$. Then $f\in\mathbf L_{(a,b)}(X^{(\mu)})$ 
if and only if \eqref{3.2} holds
and
\begin{equation}\label{3.7}
\begin{split}
&\int_p^q |f(s)|\left|\gamma-\int_{\mathbb{R}^d} \frac{x|f(s)x|^2\nu(dx)}
{(1+|f(s)x|^2)(1+|x|^2)} \right|ds<\infty\qquad\text{for all $p,q$}\\
&\qquad\text{with $a<p<q<b$}.
\end{split}
\end{equation}
Indeed,  since
\[
\int_p^q ds\left|\int_{\mathbb{R}^d} \frac{f(s)x\,|x|^2}{(1+|f(s)x|^2)(1+|x|^2)}
\nu(dx)\right|
\leqslant\frac12 \int_p^q ds\int_{\mathbb{R}^d} \frac{|x|^2}{1+|x|^2}\nu(dx)<\infty,
\]
\eqref{3.3} is rewritten into \eqref{3.7}.
\end{rem}

\begin{rem}\label{r3.2}
In all cases (irrespective of whether $A=0$ or not) condition \eqref{3.1},
that is, local square-integrability on $(a,b)$,
is sufficient for $f\in\mathbf L_{(a,b)}(X^{(\mu)})$.
See the proof of (i) of Theorem \ref{t3.1}.
\end{rem}

\begin{prop}\label{p3.2}
Suppose that $f\in\mathbf L_{(a,b)}(X^{(\mu)})$. 
Let $(A_p^q,\nu_p^q, \gamma_p^q)$ be the
triplet of $\int_p^q f(s)X^{(\mu)}(ds)$ for $a<p<q<b$. Then
{\allowdisplaybreaks
\begin{gather}
A_p^q=\int_p^q f(s)^2 Ads,\label{3.8}\\
\nu_p^q(B)=\int_p^q ds \int_{\mathbb{R}^d} 1_B(f(s)x)\nu(dx)\quad
\text{for $B\in\mathcal B
(\mathbb{R}^d)$ with $0\not\in B$}\label{3.9}\\
\gamma_p^q=\int_p^q ds\left(f(s)\gamma+\int_{\mathbb{R}^d} f(s)x\left(
\frac{1}{1+|f(s)x|^2}
-\frac{1}{1+|x|^2}\right)\nu(dx)\right).\label{3.10}
\end{gather}
}
\end{prop}

See Corollary 2.19 of (2006a).

Let us consider the case of $ID_{\mathrm{AB}}$.

\begin{thm}\label{p7.2}
Suppose that $\mu\in ID_{\mathrm{AB}}(\mathbb{R}^d)$.
Let $f(s)$ be an $\mathbb{R}$-valued measurable function on $(a,b)$. Then the
following statements are equivalent.

{\rm(a)} $f\in\mathbf L_{(a,b)}(X^{(\mu)})$ and
\begin{equation}\label{7.1a}
\mathcal L\left(\int_p^q f(s)X^{(\mu)}(ds)\right) \in ID_{\mathrm{AB}}
(\mathbb{R}^d)\quad
\text{for all $p,q$ with $a<p<q<b$}.
\end{equation}

{\rm(b)} The L\'evy measure $\nu$ and the drift $\gamma^0$ of $\mu$ satisfy
{\allowdisplaybreaks
\begin{gather}
\int_p^q ds\int_{\mathbb{R}^d} (|f(s)x|\land1)\,\nu(dx)<\infty
\quad\text{for all $p,q$ with $a<p<q<b$},\label{7.1b}\\
\int_p^q |f(s)\gamma^0|ds<\infty\quad\text{for all $p,q$ with $a<p<q<b$}
.\label{7.1c}
\end{gather}}
\end{thm}

\begin{proof}
Assume (a).  We use Propositions \ref{p3.1} and \ref{p3.2}.
The triplet $(A_p^q,\nu_p^q, \gamma_p^q)$ of $\int_p^q f(s)X^{(\mu)}(ds)$
satisfies $A_p^q=0$ and $\int_{\mathbb{R}^d}(|x|\land1)\nu_p^q(dx)<\infty$.
We have \eqref{7.1b} since
\[
\int_p^q ds\int_{\mathbb{R}^d} (|f(s)x|\land1)\,\nu(dx)
=\int_{\mathbb{R}^d}(|x|\land1)\,\nu_p^q(dx)
\]
from \eqref{3.9}. We have
{\allowdisplaybreaks
\begin{equation*}
\infty>\int_p^q |C_{\mu}(f(s)z)|ds
=\int_p^q ds\left|\int_{\mathbb{R}^d}(e^{i\langle f(s)z,x\rangle}-1)\,\nu(dx)
+\langle f(s)z,\gamma^0\rangle\right|
\end{equation*}
and
\[
\int_p^q ds\int_{\mathbb{R}^d}|e^{i\langle f(s)z,x\rangle}-1|\,\nu(dx)\leqslant 
I_1+I_2,
\]
where
\begin{align*}
I_1&=\int_p^q ds\int_{\mathbb{R}^d}|\langle f(s)x,z\rangle| 1_{\{|f(s)x|\leqslant1\}}
\nu(dx)\\
&\leqslant |z|\int_p^q ds\int_{\mathbb{R}^d}|f(s)x|1_{\{|f(s)x|\leqslant1\}}\nu(dx)
\leqslant |z|\int_p^q ds\int_{\mathbb{R}^d}(|f(s)x|\land1)\,\nu(dx)<\infty
\end{align*}
and
\[
I_2=2\int_p^q ds\int_{\mathbb{R}^d}1_{\{|f(s)x|>1\}}\nu(dx)
\leqslant2\int_p^q ds\int_{\mathbb{R}^d}(|f(s)x|\land1)\,\nu(dx)<\infty.
\]
Therefore
\[
\int_p^q |\langle f(s)\gamma^0,z\rangle|ds<\infty\quad\text{for all }z.
\]
Choosing} $z=(\delta_{jk})_{1\leqslant k\leqslant d}$, we see that
$\gamma^0=(\gamma_j^0)_{1\leqslant j\leqslant d}$ satisfies $\int_p^q
|f(s)\gamma^0_j|ds
<\infty$. Hence \eqref{7.1c} is satisfied. Thus (b) is obtained.
Note that $\int_p^q f(s)X^{(\mu)}(ds)$ has drift 
\begin{equation}\label{7.1d}
(\gamma^0)_p^q=\int_p^q f(s)\gamma^0 ds.
\end{equation}

Conversely assume (b). We have
\begin{equation*}
\int_p^q ds\int_{\mathbb{R}^d} (|f(s)x|^2\land1)\,\nu(dx)\leqslant
\int_p^q ds\int_{\mathbb{R}^d} (|f(s)x|\land1)\,\nu(dx)<\infty.
\end{equation*}
Since $\gamma^0=\gamma-\int_{\mathbb{R}^d}x(1+|x|^2)^{-1}\nu(dx)$, we have
{\allowdisplaybreaks
\begin{align*}
&\int_p^q ds\left|f(s)\gamma+\int_{\mathbb{R}^d} f(s)x\left(\frac{1}{1+|f(s)x|^2}
-\frac{1}{1+|x|^2}\right)\nu(dx)\right|\\
&\qquad=\int_p^q ds\left|f(s)\gamma^0+\int_{\mathbb{R}^d} \frac{f(s)x}{1+|f(s)x|^2}
\nu(dx)\right|,
\end{align*}
which is finite. Indeed, we have \eqref{7.1c} and,
using $(1+r^2)^{-1}\leqslant 2(1+r)^{-1}$ for 
$r\geqslant0$, we have
\begin{align*}
&\int_p^q ds\int_{\mathbb{R}^d} \frac{|f(s)x|}{1+|f(s)x|^2}
\nu(dx)\leqslant
2\int_p^q ds\int_{\mathbb{R}^d} \frac{|f(s)x|}{1+|f(s)x|}
\nu(dx)\\
&\qquad\leqslant
2\int_p^q ds\int_{\mathbb{R}^d} (|f(s)x|\land1)\nu(dx)<\infty.
\end{align*}
It follows from Theorem \ref{t3.1} that} $f\in\mathbf L_{(a,b)}(X^{(\mu)})$.
The triplet $(A_p^q,\nu_p^q, \gamma_p^q)$ of\linebreak
$\int_p^q f(s)X^{(\mu)}(ds)$
satisfies $A_p^q=0$ and \eqref{3.9}. Thus
\[
\int_{\mathbb{R}^d}(|x|\land1)\,\nu_p^q(dx)=
\int_p^q ds\int_{\mathbb{R}^d} (|f(s)x|\land1)\,\nu(dx)
<\infty.
\]
Hence we obtain \eqref{7.1a}.
\end{proof}

\section{Improper stochastic integrals on $(a,b)$}

Fix $-\infty\leqslant a<b\leqslant\infty$. 
Let us define improper stochastic integrals on $(a,b)$ with nonrandom integrands
and their modifications. For $(a,b)=(0,\infty)$ 
Cherny and Shiryaev (2005) study stochastic integrals
up to infinity (with random integrands in general) in semi-martingale
approach, but we are treating a simpler situation without using semi-martingales.
 Let $\mu=\mu_{(A,\nu,\gamma)}
\in ID(\mathbb{R}^d)$ and let $X^{(\mu)}$ be as in 
Section 2. 
{\em In this section throughout, we assume that $f\in\mathbf L_{(a,b)}(X^{(\mu)})$.}

\begin{defn}\label{d4.1}
We say that {\it the improper stochastic integral on $(a,b)$ 
of $f$ with respect to $X^{(\mu)}$ is
definable} if $\int_p^q f(s)X^{(\mu)}(ds)$ is convergent in probability in 
$\mathbb{R}^d$
as $p\downarrow a$
and $q\uparrow b$. The limit is written as $\int_{a+}^{b-} f(s)X^{(\mu)}(ds)$ and its
distribution is written as $\Phi_f(\mu)$.
\end{defn}

\begin{defn}\label{d4.2}
We say that {\it the essential improper integral on $(a,b)$ of $f$ with respect to 
$X^{(\mu)}$ is definable} if  there is 
a nonrandom $\mathbb{R}^d$-valued
function $g(p,q)$, $a<p<q<b$ such that $\int_p^q f(s)X^{(\mu)}(ds)-g(p,q)$ is
convergent in probability in $\mathbb{R}^d$ as $p\downarrow a$
and $q\uparrow b$. Notice that there is a freedom of choice of $g(p,q)$. Let
$\Phi_{f,\,\mathrm{es}}(\mu)$ denote the class of the
distributions of all such limits.
\end{defn}

\begin{defn}\label{d4.3}
We say that {\it the compensated improper integral on $(a,b)$ of $f$ with respect to
$X^{(\mu)}$ is definable} if  there is $\theta\in\mathbb{R}^d$
 such that  $\int_{a+}^{b-} f(s)X^{(\mu*\delta_{-\theta})}(ds)$ is definable.
Here $\delta_{-\theta}$ is the distribution concentrated at $-\theta$. 
As there may be a freedom of choice of $\theta$, let $\Phi_{f,\,\mathrm{c}}(\mu)$
denote the class of the distributions of all such limits.
\end{defn}

\begin{defn}\label{d4.4}
Let $X^{(\mu) \sharp}$ be an independent copy of $X^{(\mu)}$. 
We say that {\it the symmetrized improper integral on $(a,b)$ of $f$ with respect to
$X^{(\mu)}$ is definable} if\linebreak
$\int_{a+}^{b-} f(s)(X^{(\mu)}(ds)-X^{(\mu) \sharp}(ds))$
is definable. Let $\Phi_{f,\,\mathrm{sym}}(\mu)$ denote the distribution of the limit.
\end{defn}

Note that $\Phi_f(\mu)$ and $\Phi_{f,\,\mathrm{sym}}(\mu)$ are elements of
 $ID(\mathbb{R}^d)$,
while $\Phi_{f,\,\mathrm{es}}(\mu)$ and $\Phi_{f,\,\mathrm{c}}(\mu)$ are subsets of
$ID(\mathbb{R}^d)$.
{\it Thus we consider $\Phi_f$ and $\Phi_{f,\,\mathrm{sym}}$ as transformations 
of $\mu\in ID(\mathbb{R}^d)$ into $ID(\mathbb{R}^d)$, and $\Phi_{f,\,\mathrm{es}}$ and 
$\Phi_{f,\,\mathrm{c}}$ as transformations 
of $\mu\in ID(\mathbb{R}^d)$ with values being subsets of $ID(\mathbb{R}^d)$.}  

Sometimes we say that $\Phi_f(\mu)$ is definable if $\int_{a+}^{b-} f(s)X^{(\mu)}(ds)$
is definable. We say that $\Phi_{f,\,\mathrm{es}}(\mu)$ [resp.\ $\Phi_{f,\,\mathrm{c}}
(\mu)$,
$\Phi_{f,\,\mathrm{sym}}(\mu)$] is definable if the essential [resp.\ compensated, 
symmetrized] improper integral on $(a,b)$ of $f$ with respect to 
$X^{(\mu)}$ is definable.

\begin{thm}\label{t4.1}
The following three statements are equivalent.

{\rm(a)} $\Phi_f(\mu)$ is definable.

{\rm(b)} For each $z\in\mathbb{R}^d$, $\int_p^q C_{\mu}(f(s)z)ds$ is convergent
in $\mathbb{C}$ as $p\downarrow a$ and $q\uparrow b$.

{\rm(c)} The triplet $(A,\nu,\gamma)$ satisfies the following:
{\allowdisplaybreaks
\begin{gather}
\int_a^b f(s)^2 \mathrm{tr}\, A\,ds<\infty,\label{4.1}\\
\int_a^b ds\int_{\mathbb{R}^d} (|f(s)x|^2\land1)\,\nu(dx)<\infty,\label{4.2}\\
\text{$\gamma_p^q$ is convergent in $\mathbb{R}^d$ as $p\downarrow a$ and $q\uparrow b$.}
\label{4.3}
\end{gather}
}
\end{thm}

\begin{proof}
Similar to Proposition 5.5 of (2006a) and Propositions 2.2 and 2.6 
of (2006b). Here the equivalence of statement (c) 
is based on an analogue of Lemma 5.4 of (2006a).
\end{proof}

\begin{thm}\label{t4.2}
$\Phi_{f,\,\mathrm{es}}(\mu)$ is definable if and only if \eqref{4.1} and 
\eqref{4.2} hold.
\end{thm}

\begin{proof}
Similar to Proposition 5.6 of (2006a).
\end{proof}

\begin{thm}\label{t4.3}
$\Phi_{f,\,\mathrm{c}}(\mu)$ is definable if and only if \eqref{4.1}, \eqref{4.2},
and
\begin{equation}\label{4.4}
\begin{split}
&\int_p^q f(s)\left(\gamma^{\sharp}+\int_{\mathbb{R}^d} x\left(\frac{1}{1+|f(s)x|^2}
-\frac{1}{1+|x|^2}\right)\nu(dx)\right)ds\\
&\text{is convergent in $\mathbb{R}^d$ with some $\gamma^{\sharp}\in\mathbb{R}^d$ 
as $p\downarrow a$ and $q\uparrow b$.}
\end{split}
\end{equation}
\end{thm}

\begin{proof} This theorem follows from Definition \ref{d4.3} and 
Theorem \ref{t4.1}.
\end{proof}

\begin{cor}\label{c4.1}
Suppose that $\int_p^q f(s)ds$ is convergent as $p\downarrow a$ and $q\uparrow b$.
Then, $\Phi_{f,\,\mathrm{c}}(\mu)$ is definable if and only if $\Phi_f(\mu)$
is definable.
\end{cor}

This follows from Theorems \ref{t4.1} and \ref{t4.3}.

\begin{thm}\label{t4.4}
$\Phi_{f,\,\mathrm{sym}}(\mu)$ is definable if and only if \eqref{4.1} and 
\eqref{4.2} hold.
\end{thm}

\begin{proof} The law of $\int_p^q f(s)(X^{(\mu)}(ds)-X^{(\mu) \sharp}(ds))$
has triplet $(2A_p^q,(\nu_p^q)_{\mathrm{sym}},0)$, where
\begin{equation}\label{4.4a}
(\nu_p^q)_{\mathrm{sym}}\,(B)=\nu_p^q(B)+\nu_p^q(-B).
\end{equation}
Hence the condition for definability of $\Phi_{f,\,\mathrm{sym}}(\mu)$ is the same
as \eqref{4.1} and 
\eqref{4.2}.
\end{proof}

\begin{thm}\label{t4.5}
If\/ $\Phi_f(\mu)$ is definable, then $\Phi_f(\mu)$ has triplet 
$(A_{a}^{b},\nu_{a}^{b},\gamma_{a+}^{b-})$ given by
{\allowdisplaybreaks
\begin{gather}
A_{a}^{b}=\int_a^b f(s)^2 A ds,\label{4.5}\\
\nu_{a}^{b}(B)=\int_a^b ds \int_{\mathbb{R}^d} 1_B(f(s)x)\nu(dx)
\quad\text{for $B\in\mathcal B
(\mathbb{R}^d)$ with $0\not\in B$}\label{4.6}\\
\gamma_{a+}^{b-}=\lim_{p\downarrow a,\,q\uparrow b}
\int_p^q ds\left(f(s)\gamma+\int_{\mathbb{R}^d} f(s)x\left(\frac{1}{1+|f(s)x|^2}
-\frac{1}{1+|x|^2}\right)\nu(dx)\right).\label{4.7}
\end{gather}
}
\end{thm}

\begin{proof}
Similar to Proposition 5.5 of (2006a) and Proposition 2.6 
of (2006b).
\end{proof}

\begin{thm}\label{t4.6}
If\/ $\Phi_{f,\,\mathrm{es}}(\mu)$ is definable, then $\Phi_{f,\,\mathrm{es}}(\mu)$
is the class of all infinitely divisible distributions $\mu_{(\widetilde A,
\widetilde\nu,\widetilde\gamma)}$
on $\mathbb{R}^d$ such that $\widetilde A$ is $A_a^b$ of \eqref{4.5}, 
$\widetilde\nu$ is $\nu_a^b$ of 
\eqref{4.6}, and $\widetilde\gamma\in\mathbb{R}^d$.
\end{thm}

\begin{proof}
Obvious from Definition \ref{d4.2} and Theorem \ref{t4.5}.
\end{proof}

We write the limit of $\int_p^q f(s)ds$ as $p\downarrow a$ and $q\uparrow b$ as
$\int_{a+}^{b-} f(s)ds$.

\begin{thm}\label{t4.7}
Suppose that $\Phi_{f,\,\mathrm{c}}(\mu)$ is definable and that $f(s)$ is locally
integrable on $(a,b)$.

{\rm(i)}  If $\int_p^q f(s)ds$ converges to a nonzero real number
 as $p\downarrow a$ and $q\uparrow b$, then $\Phi_{f,\,\mathrm{c}}(\mu)$ is not 
a singleton,
 and $\Phi_{f,\,\mathrm{c}}(\mu)=\Phi_{f,\,\mathrm{es}}(\mu)$. 

{\rm(ii)} Assume that one of the following two conditions is satisfied:\\
\indent{\rm(a)} $\int_p^q f(s)ds$ converges to zero as $p\downarrow a$ and 
$q\uparrow b$,\\
\indent{\rm(b)} $\int_p^q f(s)ds$ is not convergent as $p\downarrow a$ and 
$q\uparrow b$.\\
Then $\Phi_{f,\,\mathrm{c}}(\mu)$ consists of a single distribution 
$\mu_{(\widetilde A,\widetilde\nu,\widetilde\gamma)}\in ID(\mathbb{R}^d)$, where 
$\widetilde A$ is $A_a^b$ of \eqref{4.5} and $\widetilde\nu$ is $\nu_a^b$ of \eqref{4.6}.
\end{thm}

\begin{proof} (i) Suppose that $\int_p^q f(s)ds$ is convergent 
 as $p\downarrow a$ and $q\uparrow b$ and $\int_{a+}^{b-} f(s)ds\neq0$. Since
$\Phi_{f,\,\mathrm{c}}(\mu)$ is definable, it follows from Theorem \ref{t4.3}
that \eqref{4.1}, \eqref{4.2}, and \eqref{4.4} hold. For any 
$\theta\in\mathbb{R}^d$,
\begin{equation*}
\int_p^q f(s)\left(\gamma-\theta+\int_{\mathbb{R}^d} x\left(\frac{1}{1+|f(s)x|^2}
-\frac{1}{1+|x|^2}\right)\nu(dx)\right)ds
\end{equation*}
tends to 
\begin{equation*}
\int_{a+}^{b-} f(s)\left(\gamma^{\sharp}+\int_{\mathbb{R}^d} x\left(\frac{1}{1+|f(s)x|^2}
-\frac{1}{1+|x|^2}\right)\nu(dx)\right)ds
+\int_{a+}^{b-} f(s)ds(\gamma-\theta-\gamma^{\sharp}),
\end{equation*}
where $\gamma^{\sharp}$ is that of \eqref{4.4}. Hence 
$\Phi_{f,\,\mathrm{c}}(\mu)=\Phi_{f,\,\mathrm{es}}(\mu)$. 

(ii) If (a) is satisfied, then it follows from \eqref{4.4} that 
\[
\int_p^q f(s)\left(\int_{\mathbb{R}^d} x\left(\frac{1}{1+|f(s)x|^2}
-\frac{1}{1+|x|^2}\right)\nu(dx)\right)ds
\]
is convergent as $p\downarrow a$ and $q\uparrow b$ and, for any $\theta\in\mathbb{R}^d$,
$\int_{a+}^{b-} f(s)X^{(\mu*\delta_{-\theta})} (ds)$ is definable and does not
depend on $\theta$. If condition (b) is satisfied, there is only one 
$\theta\in\mathbb{R}^d$ such that 
$\int_{a+}^{b-} f(s)X^{(\mu*\delta_{-\theta})} (ds)$ is definable; indeed, if
it is definable for $\theta$ and also for some $\theta'\neq\theta$ in place of
$\theta$, then
{\allowdisplaybreaks
\begin{align*}
&\int_p^q f(s)ds(\theta'-\theta)\\
&\qquad=\int_p^q f(s)\left(\gamma-\theta+\int_{\mathbb{R}^d} x\left(
\frac{1}{1+|f(s)x|^2}
-\frac{1}{1+|x|^2}\right)\nu(dx)\right)ds\\
&\qquad\quad -\int_p^q f(s)\left(\gamma-\theta'+\int_{\mathbb{R}^d} x\left(
\frac{1}{1+|f(s)x|^2}
-\frac{1}{1+|x|^2}\right)\nu(dx)\right)ds,
\end{align*}
which is convergent as} $p\downarrow a$ and $q\uparrow b$, a contradiction.
\end{proof}

\begin{thm}\label{t4.7b}
Suppose that $\Phi_{f,\,\mathrm{c}}(\mu)$ is definable and that 
$f(s)$ is locally integrable on $(a,b)$. Suppose, further, that
$\Phi_{f,\,\mathrm{c}}(\mu)$ is a singleton $\{\widetilde\mu\}$.
If $\int_{\mathbb{R}^d}|x|\widetilde\mu(dx)<\infty$, then $\int_{\mathbb{R}^d}
x\widetilde\mu(dx)=0$.
\end{thm}

\begin{proof}
We may assume that $f(s)$ is not identically zero.
Suppose $\int_{\mathbb{R}^d}|x|\widetilde\mu(dx)<\infty$. Then $\int_{|x|>1}|x|
\nu_a^b(dx)
<\infty$, and hence $\int_{|x|>1}|x|\nu_p^q(dx)<\infty$. 
We also have $\int_{|x|>1}|x|\nu(dx)<\infty$, since, for any $a>0$,
{\allowdisplaybreaks
\begin{align*}
&\int_{|x|>1}|x|\nu_p^q(dx)=\int_p^q ds\int_{|f(s)x|>1} |f(s)x|\nu(dx)\\
&\qquad\geqslant \int_{|f(s)|>a} |f(s)|ds \int_{|x|>1/a}|x|\nu(dx).
\end{align*}
Using $\theta$
such that $\int_{a+}^{b-} f(s)X^{(\mu*\delta_{-\theta})} (ds)$ is definable,
we have
\begin{align*}
&\int_p^q C_{\mu*\delta_{-\theta}}(f(s)z)ds\\
&\qquad=\int_p^q \left[-\frac12 \langle z,f(s)^2 Az\rangle
+\int_{\mathbb{R}^d}(e^{i\langle z,f(s)x\rangle}-1-i\langle z,f(s)x\rangle)\nu(dx)
\right]ds\\
&\qquad\quad+i\int_p^q f(s)ds\left[\int_{\mathbb{R}^d}\langle z,x\rangle 
(1-(1+|x|^2)^{-1})\nu(dx)+
\langle \gamma-\theta,z\rangle\right].
\end{align*}
As} $p\downarrow a$ and $q\uparrow b$, the left-hand side  and the first 
term of the right-hand side are convergent.
Hence the second term of the right-hand side is also convergent, 
but the limit must be zero, as condition (a) or (b) of Theorem \ref{t4.7} is 
satisfied.  Therefore
\[
C_{\widetilde\mu}(z)=-\frac12 \langle z,A_a^b z\rangle +\int_{\mathbb{R}^d}
(e^{i\langle z,x\rangle}-1
-i\langle z,x\rangle)\nu_a^b(dx),
\]
which shows that $\int_{\mathbb{R}^d}x\widetilde\mu(dx)=0$.
\end{proof}

Even if we assume that $f(s)$ is locally integrable on $(a,b)$, that
$\Phi_{f,\,\mathrm{c}}(\mu)$ is a singleton
 $\{\widetilde\mu\}$, and that $\int_{\mathbb{R}^d}|x|\mu(dx)$ is finite,
these assumptions do not imply finiteness of  $\int_{\mathbb{R}^d}
|x|\widetilde\mu(dx)$.
Examples for this fact are given in pp.\,36--37 of (2006c).

\begin{thm}\label{t4.7a}
If $\Phi_{f,\,\mathrm{sym}}(\mu)$ is definable, then $\Phi_{f,\,\mathrm{sym}}(\mu)$ has
triplet $(2A_{a+}^{b-},\linebreak (\nu_{a+}^{b-})_{\mathrm{sym}}, 0)$, where 
$A_{a+}^{b-}$
is given by \eqref{4.5} and
\begin{equation}\label{4.11a}
(\nu_{a+}^{b-})_{\mathrm{sym}}(B)=\nu_{a+}^{b-}(B)+\nu_{a+}^{b-}(-B)
\end{equation}
with $\nu_{a+}^{b-}$ given by \eqref{4.6}.
\end{thm}

\begin{proof}
This follows from Theorem \ref{t4.4} and its proof.
\end{proof}

The following result will be useful later.

\begin{thm}\label{t4.8}
Suppose that $\mu\in ID_{\mathrm{AB}}(\mathbb{R}^d)$. 
Then the following two statements {\rm(a)} and {\rm(b)} are equivalent.

{\rm(a)} $\Phi_f(\mu)$ is definable  and $\Phi_f(\mu)\in ID_{\mathrm{AB}}(\mathbb{R}^d)$.

{\rm(b)} The triplet $(0,\nu,\gamma^0)_0$ of $\mu$ satisfies
\begin{equation}\label{7.3}
\int_a^b ds\int_{\mathbb{R}^d} (|f(s)x|\land1)\,\nu(dx)<\infty
\end{equation}
and
\begin{equation}\label{7.4}
\int_p^q f(s)\gamma^0 ds\text{ is convergent as $p\downarrow a$ and $q\uparrow b$}.
\end{equation}

If statements {\rm(a)} and {\rm(b)} are true, then $\widetilde\mu=\Phi_f(\mu)$
has triplet $(0,\widetilde\nu,\widetilde\gamma^0)_0$, where $\widetilde\nu$ is 
$\nu_a^b$ of
\eqref{4.6} and
\begin{equation}\label{7.6}
\widetilde\gamma^0=\int_{a+}^{b-} f(s)\gamma^0 ds.
\end{equation}
\end{thm}

\begin{proof}
Assume (b). Let us show (a).  Condition
\eqref{7.4} means that either $\gamma^0=0$ or $\int_p^q f(s)ds$ 
is convergent as $p\downarrow a$ and $q\uparrow b$.
We make an argument similar to that in the proof that (b) implies (a) in
Theorem \ref{p7.2}.
Thus we have \eqref{4.2} and \eqref{4.3}, observing that
\[
\gamma_p^q=\int_p^q ds\left(f(s)\gamma^0+\int_{\mathbb{R}^d} \frac{f(s)x}{1+|f(s)x|^2}
\nu(dx)\right)
\]
with
\[
\int_a^b ds\int_{\mathbb{R}^d} \frac{|f(s)x|}{1+|f(s)x|^2}\nu(dx)<\infty.
\]
It follows
from Theorem \ref{t4.1} that $\Phi_f(\mu)$ is definable.
Let $(\widetilde A,\widetilde\nu,\widetilde\gamma)$ be the triplet of 
$\widetilde\mu=\Phi_f(\mu)$.
Then $\widetilde A=0$ and $\widetilde\nu$ is $\nu_a^b$ of
\eqref{4.6}.
Hence we have $\int_{\mathbb{R}^d} (|x|\land1)\widetilde\nu(dx)<\infty$ from \eqref{7.3},
that is, $\widetilde\mu\in ID_{\mathrm{AB}}$.
 We have
{\allowdisplaybreaks
\begin{equation}\label{7.7}
\begin{split}
C_{\widetilde\mu}(z)&=\lim_{p\downarrow a,\,q\uparrow b}C_{\int_p^q f(s)X^{(\mu)}(ds)}(z)
=\lim_{p\downarrow a,\,q\uparrow b}\int_p^q C_{\mu}(f(s)z)ds\\
&=\lim_{p\downarrow a,\,q\uparrow b}\int_p^q ds\left(\int_{\mathbb{R}^d}(e^{i\langle 
f(s)z,x\rangle}
-1)\,\nu(dx)+i\langle\gamma^0,f(s)z\rangle\right).
\end{split}
\end{equation}
Thus
\[
C_{\widetilde\mu}(z)=\int_a^b ds\int_{\mathbb{R}^d}(e^{i\langle z,f(s)x\rangle}
-1)\,\nu(dx)+i\left\langle \int_{a+}^{b-}f(s)\gamma^0 ds,z\right\rangle,
\]
since condition \eqref{7.3} allows
us to use the dominated convergence theorem.} Hence we obtain \eqref{7.6}.

Assume (a). Let us show (b).
Let $\widetilde\mu=\Phi_f(\mu)$ with triplet $(\widetilde A,\widetilde\nu,
\widetilde\gamma)$.
Since $\widetilde\nu$ is $\nu_a^b$ of \eqref{4.6}, we obtain \eqref{7.3} from
$\int_{\mathbb{R}^d} (|x|\land1)\widetilde\nu(dx)
<\infty$. We have \eqref{7.7} 
and
\[
\lim_{p\downarrow a,\,q\uparrow b}\int_p^q ds\int_{\mathbb{R}^d}(e^{i\langle f(s)z,
x\rangle}
-1)\,\nu(dx)
\]
exists in $\mathbb{R}^d$. Hence
$
\lim_{p\downarrow a,\,q\uparrow b}\int_p^q \langle f(s)\gamma^0,z\rangle ds
$
exists in $\mathbb{R}^d$. It follows that\linebreak
$\lim_{p\downarrow a,\,q\uparrow b}\int_p^q 
f(s)\gamma^0_j ds$ exists in $\mathbb{R}$ for $j=1,\ldots,d$. Hence we get \eqref{7.4}.
The last part of the theorem is obtained in the course of our 
discussion.
\end{proof}

\begin{rem}\label{r4.1}
We are considering $\Phi_f$, $\Phi_{f,\,\mathrm{c}}$, $\Phi_{f,\,\mathrm{es}}$,
and $\Phi_{f,\,\mathrm{sym}}$ as two-sided improper integrals (that is, 
$p\downarrow a$ and $q\uparrow b$).  We can reduce $\Phi_f$, 
$\Phi_{f,\,\mathrm{es}}$, and $\Phi_{f,\,\mathrm{sym}}$ to one-sided improper
integrals, but we cannot always reduce $\Phi_{f,\,\mathrm{c}}$
to one-sided improper integrals.

Fix $c\in(a,b)$. Let $\mu\in ID(\mathbb{R}^d)$ and $f\in\mathbf L_{(a,b)}(X^{(\mu)})$.
Then it is not hard to show the following (i)--(iii).

(i) $\Phi_f(\mu)$ is definable if and only if $\int_{a+}^c f(s)X^{(\mu)}(ds)$
 and $\int_c^{b-} f(s)X^{(\mu)}(ds)$ are definable, that is,
$\int_p^c f(s)X^{(\mu)}(ds)$ and
$\int_c^q f(s)X^{(\mu)}(ds)$ are convergent in probability as 
$p\downarrow a$ and $q\uparrow b$, respectively.

(ii) $\Phi_{f,\,\mathrm{es}}(\mu)$ is definable if and only if there are nonrandom
$\mathbb{R}^d$-valued functions $g(p)$, $p\in(a,c)$, and $h(q)$, $q\in(c,b)$,
such that $\int_p^c f(s)X^{(\mu)}(ds)-g(p)$ and
$\int_c^q f(s)X^{(\mu)}(ds)-h(q)$ are convergent in probability as 
$p\downarrow a$ and $q\uparrow b$, respectively.

(iii)  Let $X^{(\mu)\sharp}(\mu)$ be an independent copy of $X^{(\mu)}$.
$\Phi_{f,\,\mathrm{sym}}(\mu)$ is definable if and only if 
$\int_{a+}^c f(s)(X^{(\mu)}(ds)-X^{(\mu)\sharp}(ds))$ and
$\int_c^{b-} f(s)(X^{(\mu)}(ds)-X^{(\mu)\sharp}(ds))$ are definable.

(iv) Is it true that $\Phi_{f,\,\mathrm{c}}(\mu)$ 
is definable if and only if there are $\theta$
and $\theta'$ in $\mathbb{R}^d$ such that
$\int_{a+}^c f(s)X^{(\mu*\delta_{-\theta})}(ds)$ and
$\int_c^{b-} f(s)X^{(\mu*\delta_{-\theta'})}(ds)$ are definable? 
The answer is affirmative if $\int_p^c f(s)ds$ is convergent as $p\downarrow a$ or if 
$\int_c^q f(s)ds$ is convergent as $q\uparrow b$.
However, the answer is negative in general.  For a counter-example,
let $(a,b)=(0,\infty)$, $f(s)=s^{-1}$, and $\mu=\mu_{(0,\nu,\gamma)}$ such
that $\int_{|x|<\varepsilon}\nu(dx)=\int_{|x|>1/\varepsilon}\nu(dx)=0$ for some 
$\varepsilon\in(0,1)$ and $\int_{\mathbb{R}^d}x\nu(dx)\neq0$; use Example \ref{e5.1}
and Proposition \ref{pd.2} in the later sections; notice that then
$\Phi_{f,\,\mathrm{c}}(\mu)$ is not definable since $\int x(1+|x|^2)^{-1}\nu(dx)
\neq -\int x|x|^2(1+|x|^2)^{-1}\nu(dx)$.
\end{rem}

\section{Domains of $\Phi_f$, $\Phi_{f,\,\mathrm{es}}$,
$\Phi_{f,\,\mathrm{c}}$, and $\Phi_{f,\,\mathrm{sym}}$\\
and domain of absolute definability}

We continue to fix $-\infty\leqslant a<b\leqslant\infty$ and the dimension $d$.  
Let $f$ be an $\mathbb{R}$-valued measurable
function on $(a,b)$. Let 
$\mathfrak D (\Phi_f)$ [resp.\ $\mathfrak D (\Phi_{f,\,\mathrm{es}})$, 
$\mathfrak D (\Phi_{f,\,\mathrm{c}})$, 
$\mathfrak D (\Phi_{f,\,\mathrm{sym}})$]
denote the class of $\mu\in ID(\mathbb{R}^d)$ such that 
$f\in\mathbf L_{(a,b)}(X^{(\mu)})$
and $\Phi_f(\mu)$  [resp.\ $\Phi_{f,\,\mathrm{es}}(\mu)$, $\Phi_{f,\,\mathrm{c}}(\mu)$,
$\Phi_{f,\,\mathrm{sym}}(\mu)$] is definable.
Further we introduce the following notion.

\begin{defn}\label{d5.1}
Let $\mu\in ID(\mathbb{R}^d)$ and $f\in\mathbf L_{(a,b)}(X^{(\mu)})$.
We say that
the improper integral on $(a,b)$ of $f$
with respect to $X^{(\mu)}$, $\int_{a+}^{b-} f(s)X^{(\mu)}(ds)$,  
is {\it absolutely definable} if 
\begin{equation}\label{5.1}
\int_a^b |C_{\mu}(f(s)z)| ds<\infty\quad\text{for all $z\in\mathbb{R}^d$}.
\end{equation}
\end{defn}

If \eqref{5.1} holds, then $\int_{a+}^{b-} f(s)X^{(\mu)}(ds)$ is definable,
which follows from Theorem \ref{t4.1}.  Let $\mathfrak D^0 (\Phi_f)$ 
denote the class
of $\mu\in ID(\mathbb{R}^d)$ for which $f\in\mathbf L_{(a,b)}(X^{(\mu)})$ and
$\int_{a+}^{b-} f(s)X^{(\mu)}(ds)$ is 
absolutely definable.

\begin{thm}\label{t5.1}
Let $\mu=\mu_{(A,\nu,\gamma)}\in ID(\mathbb{R}^d)$. Then $\mu\in
\mathfrak D^0 (\Phi_f)$ if and only if
\eqref{4.1} and  \eqref{4.2} hold and
\begin{equation}\label{5.2}
\int_a^b
 \left|f(s)\left(\gamma+\int_{\mathbb{R}^d}x\left(\frac{1}{1+|f(s)x|^2}-
\frac{1}{1+|x|^2}
\right)\nu(dx) \right)\right| ds<\infty.
\end{equation}
\end{thm}

\begin{proof}
See (2006c), Proposition 2.3. Recall that \eqref{4.1}, \eqref{4.2}, and
\eqref{5.2} imply $f\in\linebreak\mathbf L_{(a,b)}(X^{(\mu)})$.
\end{proof}

Theorem \ref{t5.1} shows that $\mu\in
\mathfrak D^0 (\Phi_f)$ if and only if $f$ is $X^{(\mu)}$-integrable over 
$(a,b)$ in the sense of Rajput and Rosinski (1989); see Theorem 2.7 of their
paper. The author owes this remark to Jan Rosi\'nski.

\begin{thm}\label{c5.1}
The following relations are true:
\begin{equation}\label{5.3}
\mathfrak D^0 (\Phi_f)\subset \mathfrak D (\Phi_f)\subset 
\mathfrak D (\Phi_{f,\,\mathrm{c}})\subset \mathfrak D (\Phi_{f,\,\mathrm{es}})
=\mathfrak D (\Phi_{f,\,\mathrm{sym}}).
\end{equation}
\end{thm}

\begin{proof}
This is a consequence of the descriptions of these domains obtained from
Theorems \ref{t4.1}--\ref{t4.4} and \ref{t5.1}.
\end{proof}

\begin{rem}\label{r5.1}
If we restrict ourselves to symmetric infinitely divisible distributions,
then the five domains in \eqref{5.3} are identical.  That is, if $\mu$ is
symmetric and $\mu\in\mathfrak D (\Phi_{f,\,\mathrm{es}})$, then
$\mu\in\mathfrak D^0 (\Phi_f)$. Indeed, then the location parameter $\gamma$ 
of $\mu$ is zero and the integral in \eqref{5.2} is zero.
\end{rem}

For two functions $f$ and 
$g$, we write $f(s)\asymp g(s)$, 
$s\to\infty$, if there are positive constants $c_1$ and 
$c_2$ such that
$0<c_1 g(s)\leqslant f(s)\leqslant c_2 g(s)$ for all large $s$.

\begin{ex}\label{e5.1}
Let $f(s)$ be a measurable function on $(a,\infty)$ with $a$ finite
such that
$\int_a^t f(s)^2 ds<\infty$ for all $t\in(a,\infty)$ and
\[
f(s)\asymp s^{-1/\alpha}\quad\text{as }s\to\infty
\]
with $0<\alpha<2$.
Since we will have Remark \ref{r5.2}, the 
results of (2006b) say the following. Let $\mu=\mu_{(A,\nu,\gamma)}
\in ID(\mathbb{R}^d)$. Let $\gamma^1$ be mean of $\mu$ if it exists.

(i) Consider the condition
\begin{equation}\label{5.5}
\int_{|x|>1} |x|^{\alpha}\nu(dx)<\infty.
\end{equation}
Then, for $0<\alpha<2$,
\[
\mu\in\mathfrak D(\Phi_{f,\,\mathrm{es}})\quad\Leftrightarrow\quad
\eqref{5.5}.
\]

(ii) If $0<\alpha<1$, then
\begin{equation}\label{5.9a}
\mathfrak D^0(\Phi_f)= \mathfrak D(\Phi_f)= 
\mathfrak D(\Phi_{f,\,\mathrm{c}})= \mathfrak D(\Phi_{f,\,\mathrm{es}}).
\end{equation}

(iii) If $1<\alpha<2$, then
{\allowdisplaybreaks
\begin{gather}
\mathfrak D^0(\Phi_f)= \mathfrak D(\Phi_f)\subsetneqq 
\mathfrak D(\Phi_{f,\,\mathrm{c}})= \mathfrak D(\Phi_{f,\,\mathrm{es}}),\label{5.9b}\\
\mathfrak D(\Phi_f)\cap ID_0\subsetneqq 
\mathfrak D(\Phi_{f,\,\mathrm{c}})\cap ID_0,\label{5.9ba}
\end{gather}
and}
\[
\mu\in\mathfrak D(\Phi_f)\quad\Leftrightarrow\quad
\eqref{5.5}\text{ and }\gamma^1=0.
\]

(iv) Let $\alpha=1$ and suppose, in addition, that, for some 
$s_0>a\lor 0$ and $c>0$, 
\begin{equation*}
\int_{s_0}^{\infty}|f(s)-cs^{-1}|ds<\infty.
\end{equation*}
Consider the conditions
{\allowdisplaybreaks
\begin{gather}
\lim_{t\to\infty}\int_{s_0}^t s^{-1}ds\int_{|x|>s} x\nu(dx)
\text{ exists in $\mathbb{R}^d$,}\label{5.7}\\
\int_{s_0}^{\infty} s^{-1}ds\left|\int_{|x|>s} x\nu(dx)\right|
<\infty.\label{5.8}
\end{gather}
Then
\begin{equation}\label{5.9c}
\mathfrak D^0(\Phi_f)\cap ID_0\subsetneqq \mathfrak D(\Phi_f)\cap ID_0
\subsetneqq 
\mathfrak D(\Phi_{f,\,\mathrm{c}})\cap ID_0\subsetneqq 
\mathfrak D(\Phi_{f,\,\mathrm{es}})\cap ID_0
\end{equation}
and, letting \eqref{5.5} mean \eqref{5.5} with $\alpha=1$,
\begin{gather*}
\mu\in\mathfrak D(\Phi_{f,\,\mathrm{c}})\quad\Leftrightarrow\quad
\text{\eqref{5.5} and \eqref{5.7}},\\
\mu\in\mathfrak D(\Phi_f)\quad\Leftrightarrow\quad
\text{\eqref{5.5}, \eqref{5.7}, and $\gamma^1=0$},\\
\mu\in\mathfrak D^0(\Phi_f)\quad\Leftrightarrow\quad
\text{\eqref{5.5}, \eqref{5.8}, and $\gamma^1=0$}.
\end{gather*}}
\end{ex}

\begin{ex}\label{e5.2}
Let $f(s)$ be a measurable function on $(a,\infty)$ with $a$ finite
such that
$\int_a^t f(s)^2 ds<\infty$ for all $t\in(a,\infty)$ and
there are positive constants $\alpha$, $c_1$, and $c_2$ satisfying
\begin{equation*}
e^{-c_2 s^{\alpha}}\leqslant f(s)\leqslant e^{-c_1 s^{\alpha}}\quad
\text{for all large $s$.}
\end{equation*}
Then
\[
\mu=\mu_{(A,\nu,\gamma)}\in\mathfrak D(\Phi_{f,\,\mathrm{es}})
\quad\Leftrightarrow\quad
\int_{\mathbb{R}^d}(\log^+ |x|)^{1/\alpha}\nu(dx)<\infty,
\]
and \eqref{5.9a} holds.
See Theorem 5.15 of (2006a) and Proposition 4.3 of (2006c). 
Using notation $\Phi=\Phi_f$ for $f(s)=e^{-s}$ on $(0,\infty)$, we know 
that the range of
$\Phi$ is the class $L_0(\mathbb{R}^d)$ of all selfdecomposable distributions on 
$\mathbb{R}^d$
(see Rocha-Arteaga and Sato (2003) for references). Jurek (1983) shows
that, for $m=1,2,\ldots$,
 the subclass $L_m(\mathbb{R}^d)$ of $L_0(\mathbb{R}^d)$ is the
range of $\Phi_f$ for $f(s)=\exp(-((m+1)!\,s)^{1/(m+1)})$ on $(0,\infty)$
(see also Rocha-Arteaga and Sato (2003)).
\end{ex}

\begin{ex}\label{e5.3}
Let $f(s)$ be a measurable function on $(a,\infty)$ with $a$ finite
such that
$\int_a^t f(s)^2 ds<\infty$ for all $t\in(a,\infty)$ and
\[
f(s)\asymp e^{-e^s}\quad\text{as }s\to\infty.
\]
Then
\[
\mu=\mu_{(A,\nu,\gamma)}\in\mathfrak D(\Phi_{f,\,\mathrm{es}})
\quad\Leftrightarrow\quad
\int_{|x|>e}(\log\log |x|)\nu(dx)<\infty,
\]
and \eqref{5.9a} holds.

Proof is as follows. Let $h(s)=e^{-e^s}$.
We have $c_1 h(s)\leqslant f(s)\leqslant c_2 h(s)$ for $s\geqslant s_0$
with $c_1,c_2,\ldots$ positive constants and $s_0>2e$.
Clearly we have \eqref{4.1}. Let $s=h^{-1}(u)=\log\log(1/u)$
 be the inverse function of $u=h(s)$.
Let $I=\int_{s_0}^{\infty}ds\int_{\mathbb{R}^d}(|f(s)x|^2\land1)\,\nu(dx)$.
Then, 
{\allowdisplaybreaks
\begin{align*}
I&=\int_{s_0}^{\infty} ds\int_{|f(s)x|\leqslant1}|f(s)x|^2\nu(dx)
+\int_{s_0}^{\infty} ds\int_{|f(s)x|>1}\nu(dx)=I_1+I_2\quad
\text{(say),}\\ 
I_2&\leqslant \int_{|x|>c_2^{-1}/h(s_0)}\nu(dx)
\int_{s_0}^{h^{-1}(c_2^{-1}|x|^{-1})}ds
\leqslant \int_{|x|>c_3}(\log\log |x|)\nu(dx)+c_4,\\
I_1&\leqslant c_2^2\int_{|x|>c_1^{-1}/h(s_0)}|x|^2\nu(dx)
\int_{h^{-1}(c_1^{-1}|x|^{-1})}^{\infty}h(s)^2ds\\
&\quad+c_2^2\int_{|x|\leqslant c_1^{-1}/h(s_0)}|x|^2\nu(dx)
\int_{s_0}^{\infty}h(s)^2ds<\infty,
\end{align*}
since
\[
\lim_{u\downarrow0}\frac{1}{u^2}\int_{h^{-1}(u)}^{\infty} h(s)^2ds=
\lim_{u\downarrow0}\frac{ue^{-s}}{2u}=0.
\]
The estimate of} $I$ from below is similar.
Hence $I$ is finite if and only if\linebreak
 $\int_{|x|>e}(\log\log |x|)
\nu(dx)\allowbreak<\infty$.

To see \eqref{5.9a}, it is enough to show \eqref{5.2}
with 
$a,b$ replaced by $s_0, \infty$. Let 
\[
J=\int_{s_0}^{\infty} f(s)ds\int_{\mathbb{R}^d}
\frac{|x|^3 \nu(dx)}{(1+f(s)^2|x|^2)(1+|x|^2)}.
\]
Since $\int_{s_0}^{\infty}f(s)ds<\infty$ and $f(s)\to0$,
it is enough to show that $J<\infty$. Now,
\[
J\leqslant c_2c_1^{-1}\int_{\mathbb{R}^d}\frac{|x|^2\nu(dx)}{1+|x|^2}
\int_{s_0}^{\infty}\frac{c_1h(s)|x|ds}{1+c_1^2 h(s)^2 |x|^2}
<\infty,
\]
since, for any $r>0$,
{\allowdisplaybreaks
\begin{align*}
&\int_{s_0}^{\infty}\frac{rh(s)ds}{1+r^2 h(s)^2}=
\int_0^{h(s_0)}\frac{ru(-ds/du)du}{1+r^2u^2}
=\int_0^{h(s_0)}\frac{rdu}{(1+r^2u^2)(\log(1/u))}\\
&\qquad\leqslant c_5\int_0^{h(s_0)}\frac{rdu}{1+r^2u^2}
\leqslant c_5\int_0^{\infty}\frac{dv}{1+v^2}.
\end{align*}}
\end{ex}

\begin{ex}\label{e5.4}
Let $f(s)$ be a measurable function on $(a,\infty)$ with $a$ finite
such that
$\int_a^t f(s)^2 ds<\infty$ for all $t\in(a,\infty)$ and
\[
f(s)\asymp s^{-1}(\log s)^{-\beta}\quad\text{as }s\to\infty
\]
with $\beta\in\mathbb{R}$.

(i) Consider the condition
\begin{equation}\label{5.5a}
\int_{|x|>2} |x|(\log |x|)^{-\beta}\nu(dx)<\infty.
\end{equation}
Then, for $\beta\in\mathbb{R}$,
\[
\mu_{(A,\nu,\gamma)}\in\mathfrak D(\Phi_{f,\,\mathrm{es}})\quad\Leftrightarrow
\quad\eqref{5.5a}.
\]

(ii) If $0<\beta\leqslant1$, then
\begin{equation}\label{5.5b}
\mathfrak D(\Phi_f)\subsetneqq 
\mathfrak D(\Phi_{f,\,\mathrm{c}})\subsetneqq \mathfrak D(\Phi_{f,\,\mathrm{es}}).
\end{equation}

(iii) If $1<\beta\leqslant2$, then
\begin{equation}\label{5.5c}
\mathfrak D(\Phi_f)=
\mathfrak D(\Phi_{f,\,\mathrm{c}})\subsetneqq \mathfrak D(\Phi_{f,\,\mathrm{es}}).
\end{equation}

Proof of (i) is similar to the first half of
the proof in Example \ref{e5.3}.
This time $h(s)=s^{-1}(\log s)^{-\beta}$ and $s_0>a\lor1$ is
chosen so large that $h(s)$ is strictly decreasing for
$s\geqslant s_0$. We define $I$, $I_1$, and $I_2$ in the same way.
Then
{\allowdisplaybreaks
\begin{align*}
I_1&\leqslant c_6\int_{|x|>c_7}|x|(\log h^{-1}(c_1^{-1}|x|^{-1}))^{-\beta}
\nu(dx)+c_8\\
&\leqslant c_9\int_{|x|>c_7}|x|(\log |x|)^{-\beta}\nu(dx)+c_8
\intertext{since $\int_t^{\infty}h(s)^2 ds\sim t^{-1}
(\log t)^{-2\beta}=h(t)(\log t)^{-\beta}$ as $t\to\infty$ and
$\log h^{-1}(u)\sim\log(1/u)$ as $u\downarrow0$, and}
I_2&\leqslant c_{10}\int_{|x|>c_{11}}|x|(\log |x|)^{-\beta}\nu(dx)
\end{align*}
since} $h^{-1}(u)\sim u^{-1}(\log(1/u))^{-\beta}$ as $u\downarrow0$.
We can estimate $I$ from below similarly. Hence $I<\infty$
if and only if \eqref{5.5a} holds.

Prof of (ii).  Let $0<\beta\leqslant1$. 
We have $\int_{s_0}^{\infty}f(s)ds=\infty$ for
some $s_0$.  Hence, to show $\mathfrak D(\Phi_f)\subsetneqq 
\mathfrak D(\Phi_{f,\,\mathrm{c}})$, choose $\mu_{(A,\nu,\gamma)}$ with
$\nu$ symmetric and $\gamma\neq0$.  In order to show
$\mathfrak D(\Phi_{f,\,\mathrm{c}})\subsetneqq \mathfrak D(\Phi_{f,\,\mathrm{es}})$,
consider $\mu_{(A,\nu,\gamma)}$ with
$\nu$ concentrated on the $x_1$-axis.
For simplicity of notation let $d=1$. Assume that
$\int_{(-\infty,2]}\nu(dx)=0$, $\int_{(2,\infty)}x\nu(dx)=
\infty$, and $\int_{(2,\infty)}x(\log x)^{-\beta}\nu(dx)
<\infty$. Then
\[
\int_{(2,\infty)}x\left(\frac{1}{1+(f(s)x)^2}-\frac{1}{1+x^2}
\right)\nu(dx)\to\infty\quad\text{as }s\to\infty,
\]
from which it follows that $\mu\not\in \mathfrak D(\Phi_{f,\,\mathrm{c}})$.

Proof of (iii).  Let $1<\beta\leqslant2$. Then $\int_a^{\infty}f(s)ds$
is finite. Thus $\mathfrak D(\Phi_f)=\mathfrak D(\Phi_{f,\,\mathrm{c}})$ as
Corollary \ref{c4.1} says. To show 
$\mathfrak D(\Phi_{f,\,\mathrm{c}})\subsetneqq \mathfrak D(\Phi_{f,\,\mathrm{es}})$,
let $d=1$, $\int_{(-\infty,2]}\nu(dx)=0$, and $\nu(dx)=x^{-2}dx$
on $(2,\infty)$.  We have
\[
\int_{s_0}^q f(s)ds\int_2^{\infty}\left(\frac{1}{1+(f(s)x)^2}
-\frac{1}{1+x^2}\right) x^{-1}dx\to\infty\quad\text{as }
q\to\infty,
\]
because, using $\int_r^{\infty}x^{-1}(1+x^2)^{-1}dx\sim\log(1/r)$
as $r\downarrow0$, we have
{\allowdisplaybreaks
\begin{align*}
&\int_{s_0}^q f(s)ds\int_2^{\infty}\frac{x^{-1}dx}{1+(f(s)x)^2}
=\int_{s_0}^q f(s)ds\int_{2f(s)}^{\infty}\frac{x^{-1}dx}{1+x^2}\\
&\qquad\geqslant c_{11}\int_{s_0}^q h(s)\log\frac{1}{2c_2h(s)}ds\\
&\qquad=c_{11}\int_{s_0}^q s^{-1}(\log s)^{-\beta}\log
(2^{-1}c_2^{-1}s(\log s)^{\beta})ds\geqslant c_{12}\int_{s_0}^q s^{-1}
(\log s)^{1-\beta}
ds\to\infty.
\end{align*}
Hence} $\mu_{(A,\nu,\gamma)}\not\in\mathfrak D(\Phi_{f,\,\mathrm{c}})$.
\end{ex}

\begin{rem}\label{r5.3}
Let $f$ be a locally square-integrable function on $(a,b)$ with
$-\infty\leqslant a<b\leqslant\infty$.

(i) If $\mu_1$ and $\mu_2$ are in $\mathfrak D(\Phi_f)$, then 
$\mu_1*\mu_2\in \mathfrak D(\Phi_f)$.  That is, $\mathfrak D(\Phi_f)$ is closed
under convolution. Also, $\mathfrak D^0(\Phi_f)$, 
$\mathfrak D(\Phi_{f,\,\mathrm{c}})$, and $\mathfrak D(\Phi_{f,\,\mathrm{es}})$ are
closed under convolution.
Indeed, $f$ is in $\mathbf L_{(a,b)}(X)$ for all $X$, and the conditions
in Section 3 work, since convolution gives addition of triplets.

(ii) If $\mu_1$ and $\mu_2$ are in $ID(\mathbb{R}^d)$ and 
$\mu_1*\mu_2\in\mathfrak D(\Phi_{f,\,\mathrm{es}})$, then $\mu_1$ and $\mu_2$ are 
in $\mathfrak D(\Phi_{f,\,\mathrm{es}})$. Use Theorem \ref{t4.2}.

(iii) With some choice of $f$, there are $\mu_1$ and $\mu_2$ in $ID(\mathbb{R}^d)$
such that $\mu_1*\mu_2\in\mathfrak D(\Phi_f)$, $\mu_1\not\in\mathfrak D(\Phi_f)$,
and $\mu_2\not\in\mathfrak D(\Phi_f)$. Use Example \ref{e5.1} with 
$f(s)=s^{-1/\alpha}$, $1<\alpha<2$, on $(1,\infty)$.

(iv) If $\mu_2\in ID(\mathbb{R}^d)$ and if $\mu_1$ and $\mu_1*\mu_2$ are in
$\mathfrak D(\Phi_f)$, then $\mu_2\in\mathfrak D(\Phi_f)$. Use Theorem \ref{t4.1}.
\end{rem}

We give some comments on the relation with improper stochastic integrals
studied in (2006a,b,c).

\begin{rem}\label{r5.2}
Let $\mathbf L_{[0,\infty)}(X^{(\mu)})$ be the class of locally 
$X^{(\mu)}$-integrable functions on $[0,\infty)$ in Definition 2.16
of (2006a).  Then the following (a) and (b) are equivalent:

(a) $f\in\mathbf L_{[0,\infty)}(X^{(\mu)})$. 

(b) $f\in\mathbf L_{(0,\infty)}(X^{(\mu)})$ 
and, for some $t$ (equivalently, for any $t$) in
$(0,\infty)$,\linebreak
$\int_{0+}^t f(s)X^{(\mu)}(ds)$ is absolutely definable.\\
In this case,
\begin{equation}\label{5.4}
\int_0^t f(s)X^{(\mu)}(ds)=\int_{0+}^t f(s)X^{(\mu)}(ds),
\end{equation}
the left-hand side being defined in (2006a) since 
$f\in\mathbf L_{[0,\infty)}(X^{(\mu)})$. 

Proof is as follows. 
As Theorem 3.1 of (2006a) says, $f\in\mathbf L_{[0,\infty)}(X^{(\mu)})$
if and only if, for all $t\in(0,\infty)$,
{\allowdisplaybreaks
\begin{gather*}
\int_0^t f(s)^2 \mathrm{tr}\, A ds<\infty,\\
\int_0^t ds \int_{\mathbb{R}^d}(|f(s)x|^2\land 1)\,\nu(dx)<\infty,\\
\int_0^t \left|f(s)\gamma+\int_{\mathbb{R}^d} f(s)x\left(\frac{1}{1+|f(s)x|^2}
-\frac{1}{1+|x|^2}\right)\nu(dx)\right| ds<\infty.
\end{gather*}
Combine this  with Theorem \ref{t3.1} and the one-sided version of 
Theorem \ref{t5.1}. Then we see the equivalence
of (a) and (b). To prove \eqref{5.4}, it is enough to show that}
$\int_{1/n}^t f(s)X^{(\mu)}(ds)$ tends to $\int_0^t f(s)X^{(\mu)}(ds)$
in probability as $n\to\infty$. Let, for large $n$,
\[
F_n=\int_0^t f(s)X^{(\mu)}(ds)-\int_{1/n}^t f(s)X^{(\mu)}(ds)
=\int_0^{1/n} f(s) X^{(\mu)}(ds).
\]
Then the triplet $(A_n,\nu_n,\gamma_n)$ of the $\mathcal L(F_n)$ satisfies
$A_n\to0$, $\int_{R^d}(|x|^2\land1)\,\nu_n(dx)\to0$, and
$\gamma_n\to0$ as $n\to\infty$.  Thus $F_n\to0$ in probability.

In (2006a,b,c) $\int_0^{\infty-} f(s) X^{(\mu)}(ds)$
is said to be definable if $f\in\mathbf L_{[0,\infty)}(X^{(\mu)})$ and if 
$\int_0^t f(s)X^{(\mu)}(ds)$ is convergent in probability as $t\to\infty$.
It follows from the results above that the following (c) and (d) are
equivalent:

(c) $\int_0^{\infty-} f(s) X^{(\mu)}(ds)$ is definable.

(d) $f\in\mathbf L_{(0,\infty)}(X^{(\mu)})$, $\int_1^{\infty-} 
f(s) X^{(\mu)}(ds)$ is definable, and $\int_{0+}^1 f(s)X^{(\mu)}(ds)$ 
is absolutely definable.\\
In this case,
\[
\int_0^{\infty-} f(s)X^{(\mu)}(ds)=\int_{0+}^{\infty-} f(s)X^{(\mu)}(ds).
\]

The domain of those $\mu$ for which $\int_0^{\infty-} f(s) X^{(\mu)}(ds)$ is 
definable is denoted by $\mathfrak D[f,\mathbb{R}^d]$ in (2006a) and by
$\mathfrak D (\Phi_f)$ in (2006b,c).  The domains for
essential and compensated improper integrals for
$\int_0^{\infty-} f(s) X^{(\mu)}(ds)$ are respectively 
denoted by $\mathfrak D_{\mathrm{es}} [f,\mathbb{R}^d]$ and 
$\mathfrak D_{\mathrm{c}} [f,\mathbb{R}^d]$ 
in (2006a) and by $\mathfrak D_{\mathrm{e}} (\Phi_f)$ 
(or $\mathfrak D_{\mathrm{es}} (\Phi_f)$) and $\mathfrak D_{\mathrm{c}} (\Phi_f)$ in 
(2006b,c).
\end{rem}

\section{Duals of infinitely divisible distributions}
We introduce the concept of duals of purely non-Gaussian 
infinitely divisible distributions on $\mathbb{R}^d$. Utilizing this
concept, we relate some improper stochastic integrals on
$(0,b)$ with $0<b<\infty$ to those on $(a,\infty)$ with $a$ finite.

\begin{defn}\label{dd.1}
Let $\mu=\mu_{(0,\nu,\gamma)}\in ID_0(\mathbb{R}^d)$. We call the distribution
$\mu'\in ID_0(\mathbb{R}^d)$ the {\em dual} of $\mu$ if the triplet 
$(0,\nu',\gamma')$ of $\mu'$ satisfies
\begin{gather}
\nu'(B)=\int_{\mathbb{R}^d\setminus\{0\}} 1_B(\iota(x)) |x|^2 \nu(dx),
\quad B\in\mathcal B(\mathbb{R}^d),\label{d.1}\\
\gamma'=-\gamma, \label{d.2}
\end{gather}
where $\iota(x)=|x|^{-2}x$, the inversion of $x$, which maps 
$\mathbb{R}^d\setminus\{0\}$ onto itself.
\end{defn}

The simple form \eqref{d.2} of the relation of location parameters of
$\mu$ and $\mu'$ is due to the fact that we are using the centering
function $x(1+|x|^2)^{-1}$ in the L\'evy--Khintchine representation
in this paper.

The relation \eqref{d.1} implies
\begin{equation}\label{d.3}
\int_{\mathbb{R}^d}h(x)\nu'(dx)=\int_{\mathbb{R}^d\setminus\{0\}} h(\iota(x)) 
|x|^2 \nu(dx)
\end{equation}
for all nonnegative measurable function $h(x)$ and thus
\begin{equation}\label{d.4}
\int_{\mathbb{R}^d}(|x|^2\land1)\nu'(dx)=\int_{\mathbb{R}^d}(|x|^2\land1)\nu(dx).
\end{equation}

\begin{prop}\label{pd.1}
{\rm(i)} The dual of the dual of $\mu$ is $\mu$ itself. That is, 
$\mu''=\mu$ for all $\mu\in ID_0(\mathbb{R}^d)$.

{\rm(ii)}\quad $\int_{|x|\leqslant1} |x|^{2-\alpha}\nu'(dx)=\int_{|x|\geqslant1} 
|x|^{\alpha}\nu(dx)$ for $\alpha\in\mathbb{R}$.

{\rm(iii)} The dual $\mu'$ is of type {\rm A} if and only if $\mu$ has finite
second moment.

{\rm(iv)} The dual $\mu'$ is of type {\rm A} or {\rm B} if and only if $\mu$ 
has finite mean.

{\rm(v)} \;If $\mu'$ is of type {\rm A} or {\rm B}, then
\[
(\text{drift of }\mu')=-(\text{mean of }\mu).
\]

{\rm(vi)} Let $0<\alpha<2$. Then $\mu'$ is $(2-\alpha)$-stable if and only
if $\mu$ is $\alpha$-stable.

{\rm(vii)} Let $0<\alpha<2$. Then $\mu'$ is strictly $(2-\alpha)$-stable 
if and only if $\mu$ is strictly $\alpha$-stable.
\end{prop}

\begin{proof}
(i) Using \eqref{d.3}, we have
\begin{align*}
\nu''(B)&=\int_{\mathbb{R}^d\setminus\{0\}} 1_B(\iota(x)) |x|^2 \nu'(dx)\\
&=\int_{\mathbb{R}^d\setminus\{0\}} 1_B(\iota(\iota(x))) |\iota(x)|^2
|x|^2 \nu(dx)
=\nu(B)
\end{align*}
for all $B$. Further, $\gamma''=-\gamma'=\gamma$.

(ii) Using \eqref{d.3} again, we have
\[
\int_{|x|\leqslant1} |x|^{2-\alpha}\nu'(dx)
=\int_{|\iota(x)|\leqslant1} |\iota(x)|^{2-\alpha}|x|^2\nu(dx)
=\int_{|x|\geqslant1} |x|^{\alpha}\nu(dx).
\]

(iii) and (iv)\quad Use (ii) with $\alpha=2$ and $1$, respectively.

(v) Let $(\gamma^0)'$ and $\gamma^1$ be the drift of $\mu'$ and the mean
of $\mu$, respectively.  Then
\[
(\gamma^0)'=\gamma'-\int_{\mathbb{R}^d}\frac{x\nu'(dx)}{1+|x|^2},\quad
\gamma^1=\gamma+\int_{\mathbb{R}^d}\frac{x|x|^2\nu(dx)}{1+|x|^2}.
\]
Hence
\[
(\gamma^0)'=-\gamma-\int_{\mathbb{R}^d}\frac{\iota(x)}{1+|\iota(x)|^2}
|x|^2\nu(dx)=-\gamma-\int_{\mathbb{R}^d}\frac{x}{1+|x|^{-2}}\nu(dx)=-\gamma^1.
\]

(vi) Let $0<\alpha<2$.  Then a nontrivial distribution $\mu$ is 
$\alpha$-stable if and only if $A=0$ and
\[
\nu(B)=\int_S \lambda(d\xi)\int_0^{\infty} 1_B(r\xi) r^{-\alpha-1}dr,
\]
where $S$ is the unit sphere and $\lambda$ is a nonzero finite measure 
on $S$.  In this case,
\[
\int h(\iota(x))|x|^2 \nu(dx)=\int_S \lambda(d\xi)\int_0^{\infty}
h(r^{-1}\xi) r^2 r^{-\alpha-1}dr
\]
and thus
\[
\nu'(B)=\int_S \lambda(d\xi)\int_0^{\infty} 1_B(r^{-1}\xi) r^{1-\alpha}dr
=\int_S \lambda(d\xi)\int_0^{\infty} 1_B(u\xi) u^{-1-(2-\alpha)}du.
\]

(vii) If $0<\alpha<1$ and $\mu$ is nontrivial, then
 $\mu$ is strictly $\alpha$-stable if and only if
$\mu$ is $\alpha$-stable with additional condition $\gamma^0=0$ 
($0<\alpha<1$), $\gamma^1=0$ ($1<\alpha<2$), or $\int_S \xi\lambda(d\xi)=0$ 
($\alpha=1$).  Moreover, $\delta_{\gamma}$ with $\gamma\neq0$ is
strictly $1$-stable. See Theorem 14.7 of Sato (1999). 
Now use (v) and the fact that $\nu$ and $\nu'$ has the 
identical $\lambda$-measure  as seen in the proof of (vi).
\end{proof}

The following fact is in duality with the fact in Example \ref{e5.1}.

\begin{prop}\label{pd.2}
Let $f(s)$ be a measurable function on $(0,b)$ with $0<b<\infty$
such that
$\int_t^b f(s)^2 ds<\infty$ for all $t\in(0,b)$ and
\[
f(s)\asymp s^{-1/(2-\alpha)}\quad\text{as }s\downarrow0
\]
with $0<\alpha<2$.
Let $\mu=\mu_{(A,\nu,\gamma)}
\in ID(\mathbb{R}^d)$. Let $\gamma^0$ be drift of $\mu$ if it exists.

{\rm(i)} Consider the condition
\begin{equation}\label{d.5}
A=0\text{ and }\int_{|x|<1} |x|^{2-\alpha}\nu(dx)<\infty.
\end{equation}
Then, for $0<\alpha<2$, 
\[
\mu\in\mathfrak D(\Phi_{f,\,\mathrm{es}})\quad\Leftrightarrow\quad
\eqref{d.5}.
\]

{\rm(ii)} If $0<\alpha<1$, then
\begin{equation}\label{d.6a}
\mathfrak D^0(\Phi_f)= \mathfrak D(\Phi_f)= 
\mathfrak D(\Phi_{f,\,\mathrm{c}})= \mathfrak D(\Phi_{f,\,\mathrm{es}}).
\end{equation}

{\rm(iii)} If $1<\alpha<2$, then
\begin{equation}\label{d.6b}
\mathfrak D^0(\Phi_f)= \mathfrak D(\Phi_f)\subsetneqq 
\mathfrak D(\Phi_{f,\,\mathrm{c}})= \mathfrak D(\Phi_{f,\,\mathrm{es}})
\end{equation}
and
\[
\mu\in\mathfrak D(\Phi_f)\quad\Leftrightarrow\quad
\eqref{d.5}\text{ and }\gamma^0=0.
\]

{\rm(iv)} Let $\alpha=1$ and suppose, in addition, that, for some 
$s_0\in(0,b)$ and $c>0$, 
\begin{equation*}
\int_0^{s_0}|f(s)-cs^{-1}|ds<\infty.
\end{equation*}
Consider the conditions
{\allowdisplaybreaks
\begin{gather}
\lim_{t\downarrow0}\int_t^{s_0} s^{-1}ds\int_{|x|<s} x\nu(dx)
\text{ exists in $\mathbb{R}^d$,}\label{d.8}\\
\int_{0}^{s_0} s^{-1}ds\left|\int_{|x|<s} x\nu(dx)\right|
<\infty.\label{d.9}
\end{gather}
Then
\begin{equation}\label{d.6c}
\mathfrak D^0(\Phi_f)\subsetneqq \mathfrak D(\Phi_f)\subsetneqq 
\mathfrak D(\Phi_{f,\,\mathrm{c}})\subsetneqq 
\mathfrak D(\Phi_{f,\,\mathrm{es}})
\end{equation}
and, letting \eqref{d.5} mean \eqref{d.5} with $\alpha=1$,
\begin{gather*}
\mu\in\mathfrak D(\Phi_{f,\,\mathrm{c}})\quad\Leftrightarrow\quad
\text{\eqref{d.5} and \eqref{d.8}},\\
\mu\in\mathfrak D(\Phi_f)\quad\Leftrightarrow\quad
\text{\eqref{d.5}, \eqref{d.8}, and $\gamma^0=0$},\\
\mu\in\mathfrak D^0(\Phi_f)\quad\Leftrightarrow\quad
\text{\eqref{d.5}, \eqref{d.9}, and $\gamma^0=0$}.
\end{gather*}}
\end{prop}

\begin{proof}
From our assumption there are $c_1,c_2>0$ and $s_0\in(0,b)$
such that
\begin{equation}\label{d.11}
c_1 s^{-1/(2-\alpha)}\leqslant f(s)\leqslant c_2 s^{-1/(2-\alpha)}
\quad\text{for }s\in(0,s_0].
\end{equation}
Let $u_0=s_0^{-\alpha/(2-\alpha)}$ and let $g(u)=1/f(u^{-(2-\alpha)/\alpha})$
for $u\geqslant u_0$. Then $g(u)\asymp u^{-1/\alpha}$ as $u\to\infty$,
since 
\[
\frac{g(u)}{u^{-1/\alpha}}=\frac{u^{1/\alpha}}{f(u^{-(2-\alpha)/\alpha})}
=\frac{s^{-1/(2-\alpha)}}{f(s)},
\]
where $s=u^{-(2-\alpha)/\alpha}$. As this shows that
 $g(u)$ is locally bounded
on $[u_0,\infty)$, we have $\int_{u_0}^v g(u)^2 du<\infty$
for all $v\in(u_0,\infty)$.

(i)  If $\mu\in\mathfrak D(\Phi_{f,\,\mathrm{es}})$, then $A=0$, 
since $\int_0^{s_0} f(s)^2 ds\geqslant c_1\int_0^{s_0}
s^{-2/(2-\alpha)}ds=\infty$. We have
{\allowdisplaybreaks
\begin{align*}
&\mu=\mu_{(A,\nu,\gamma)}\in\mathfrak D(\Phi_{f,\,\mathrm{es}})
\quad\Leftrightarrow\quad A=0\text{ and }
\mu_{(0,\nu,\gamma)}\in\mathfrak D(\Phi_{f,\,\mathrm{es}})\\
&\qquad\overset{(*)}{\Leftrightarrow}\quad A=0\text{ and }
\mu_{(0,\nu',\gamma')}\in\mathfrak D(\Phi_{g,\,\mathrm{es}})\\
&\qquad\overset{(**)}{\Leftrightarrow}\quad A=0\text{ and }
\int_{|x|>1}|x|^{\alpha}\nu'(dx)<\infty
\quad\overset{(***)}{\Leftrightarrow}\quad
\eqref{d.5},
\end{align*}
where $\mu_{(0,\nu',\gamma')}$ is the dual of $\mu_{(0,\nu,\gamma)}$.
Write $\beta=(2-\alpha)/\alpha$. 
The equivalence $(*)$ comes from the equality that
\begin{align*}
&\int_0^{s_0} ds\int_{\mathbb{R}^d}(|f(s)x|^2\land1)\nu(dx)
=\beta\int_{u_0}^{\infty} u^{-2/\alpha}du
\int_{\mathbb{R}^d}\left(\left|\frac{x}{g(u)}
\right|^2\land1\right)\nu(dx)\\
&\qquad =\beta\int_{u_0}^{\infty} u^{-2/\alpha}du
\int_{\mathbb{R}^d}(|g(u)x|^{-2}\land1)|x|^2\nu'(dx)\\
&\qquad 
=\beta\int_{u_0}^{\infty} \frac{u^{-2/\alpha}}{g(u)^2}du
\int_{\mathbb{R}^d}(1\land|g(u)x|^2)\nu'(dx).
\end{align*}
The equivalence} $(**)$ comes from Example \ref{e5.1} and
$(***)$ from Proposition \ref{pd.1}.

(ii) We have
{\allowdisplaybreaks
\begin{align*}
&\int_0^{s_0} f(s)ds\left|\gamma+\int_{\mathbb{R}^d}x\left(
\frac{1}{1+|f(s)x|^2}-\frac{1}{1+|x|^2}\right)\nu(dx)\right|\\
&\qquad=\beta\int_{u_0}^{\infty} \frac{u^{-2/\alpha}}{g(u)}du
\left|\gamma+\int_{\mathbb{R}^d}x\left(
\frac{1}{1+|x/g(u)|^2}-\frac{1}{1+|x|^2}\right)\nu(dx)\right|\\
&\qquad=\beta\int_{u_0}^{\infty} \frac{u^{-2/\alpha}}{g(u)^2}g(u)du
\left|\gamma'+\int_{\mathbb{R}^d}x\left(
\frac{1}{1+|g(u)x|^2}-\frac{1}{1+|x|^2}\right)\nu'(dx)\right|.
\end{align*}
Hence, for $0<\alpha<1$,
\begin{align*}
&\mu\in\mathfrak D^0(\Phi_f)
\quad\Leftrightarrow\quad A=0\text{ and }
\mu_{(0,\nu,\gamma)}\in\mathfrak D^0(\Phi_f)\\
&\qquad\Leftrightarrow\quad A=0\text{ and }
\mu_{(0,\nu',\gamma')}\in\mathfrak D^0(\Phi_g)
\quad\Leftrightarrow\quad A=0\text{ and }
\mu_{(0,\nu',\gamma')}\in\mathfrak D(\Phi_{g,\,\mathrm{es}})\\
&\qquad\Leftrightarrow\quad A=0\text{ and }
\mu_{(0,\nu,\gamma)}\in\mathfrak D(\Phi_{f,\,\mathrm{es}})
\quad\Leftrightarrow\quad \mu\in\mathfrak D(\Phi_{f,\,\mathrm{es}}).
\end{align*}}
\indent
(iii) We have, for $v=t^{-\alpha/(2-\alpha)}$,
{\allowdisplaybreaks
\begin{align*}
&\int_t^{s_0} f(s)ds\left(\gamma+\int_{\mathbb{R}^d}x\left(
\frac{1}{1+|f(s)x|^2}-\frac{1}{1+|x|^2}\right)\nu(dx)\right)\\
&\qquad=\beta\int_{u_0}^v \frac{u^{-2/\alpha}}{g(u)}du
\left(\gamma+\int_{\mathbb{R}^d}x\left(
\frac{1}{1+|x/g(u)|^2}-\frac{1}{1+|x|^2}\right)\nu(dx)\right)\\
&\qquad=-\beta\int_{u_0}^v \frac{u^{-2/\alpha}}{g(u)^2}g(u)du
\left(\gamma'+\int_{\mathbb{R}^d}x\left(
\frac{1}{1+|g(u)x|^2}-\frac{1}{1+|x|^2}\right)\nu'(dx)\right).
\end{align*}
Thus, for $1<\alpha<2$, using Example \ref{e5.1},
\begin{align*}
&\mu\in\mathfrak D(\Phi_f)
\quad\Leftrightarrow\quad A=0\text{ and }
\mu_{(0,\nu,\gamma)}\in\mathfrak D(\Phi_f)\\
&\qquad\Leftrightarrow\quad A=0\text{ and }
\mu_{(0,\nu',\gamma')}\in\mathfrak D(\Phi_g)\\
&\qquad\Leftrightarrow\quad A=0, \int_{|x|>1}|x|^{\alpha}\nu'(dx)<\infty,
\text{ and (mean of $\mu_{(0,\nu',\gamma')})=(\gamma^1)'=0$}\\
&\qquad\Leftrightarrow\quad \text{\eqref{d.5} and (drift 
of $\mu_{(0,\nu,\gamma)})=\gamma^0=0$}.
\end{align*}
and, in a similar way,} $\mu\in\mathfrak D^0(\Phi_f)$ if and only if
\eqref{d.5} and $\gamma^0=0$. Now we see that
$\mu\in\mathfrak D(\Phi_{f,\,\mathrm{c}})$ if and only if \eqref{d.5} holds.
Hence $\mathfrak D(\Phi_f)=\mathfrak D^0(\Phi_f)$ and $\mathfrak D
(\Phi_{f,\,\mathrm{c}})
=\mathfrak D(\Phi_{f,\,\mathrm{es}})$. Since $\mathfrak D(\Phi_g)\cap ID_0
\subsetneqq\mathfrak D(\Phi_{g,\,\mathrm{es}})\cap ID_0$, we have 
$\mathfrak D(\Phi_f)\subsetneqq\mathfrak D(\Phi_{f,\,\mathrm{es}})$.

(iv) ($\alpha=1$) 
As in the proof of (ii),
{\allowdisplaybreaks
\begin{align*}
&\mu\in\mathfrak D^0(\Phi_f)
\quad\Leftrightarrow\quad A=0\text{ and }
\mu_{(0,\nu,\gamma)}\in\mathfrak D^0(\Phi_f)\\
&\qquad\Leftrightarrow\quad A=0\text{ and }
\mu_{(0,\nu',\gamma')}\in\mathfrak D^0(\Phi_g)\\
&\qquad\Leftrightarrow\quad A=0, \int_{|x|>1}|x|^{\alpha}\nu'(dx)<\infty,
\int_{u_0}^{\infty}u^{-1}du\left|\int_{|x|>u}x\nu'(dx)\right|<\infty,
(\gamma^1)'=0\\
&\qquad\Leftrightarrow\quad\text{\eqref{d.5}, \eqref{d.9}, and $\gamma^0=0$}.
\end{align*}
For the last equivalence, note that
\[
\int_{u_0}^{\infty}u^{-1}du\left|\int_{|x|>u}x\nu'(dx)\right|
=\int_0^{s_0} s^{-1}ds\left|\int_{|x|<s}x\nu(dx)\right|.
\]
Let
\[
I(t)=\int_t^{s_0} f(s)ds\left(\gamma+\int_{\mathbb{R}^d}x\left(
\frac{1}{1+|f(s)x|^2}-\frac{1}{1+|x|^2}\right)\nu(dx)\right).
\]
Let us write $c_3,c_4,\ldots$ for positive constants. We have
\begin{align*}
I(t)&=\int_t^{s_0} (f(s)-cs^{-1})ds\left(\gamma+\int_{\mathbb{R}^d}x\left(
\frac{1}{1+|f(s)x|^2}-\frac{1}{1+|x|^2}\right)\nu(dx)\right)\\
&\quad+\int_t^{s_0} cs^{-1}ds\int_{\mathbb{R}^d}x\left(
\frac{1}{1+|f(s)x|^2}-\frac{1}{1+|cs^{-1}x|^2}\right)\nu(dx)\\
&\quad+\int_t^{s_0} cs^{-1}ds\left(\gamma+\int_{\mathbb{R}^d}x\left(
\frac{1}{1+|cs^{-1}x|^2}-\frac{1}{1+|x|^2}\right)\nu(dx)\right)\\
&=I_1(t)+I_2(t)+I_3(t)\qquad\text{(say).}
\end{align*}
Under \eqref{d.5}, $I_1(t)$ and $I_2(t)$ 
are convergent as $t\downarrow0$, since
\begin{align*}
&\int_0^{s_0} |f(s)-cs^{-1}|ds\left|\gamma+\int_{\mathbb{R}^d}x\left(
\frac{1}{1+|f(s)x|^2}-\frac{1}{1+|x|^2}\right)\nu(dx)\right|\\
&\qquad\leqslant \int_0^{s_0} |f(s)-cs^{-1}|ds\left(|\gamma|
+\int_{\mathbb{R}^d}\frac{|x|^3+|x||f(s)x|^2}{(1+|f(s)x|^2)(1+|x|^2)}\nu(dx)
\right)\\
&\qquad\leqslant \int_0^{s_0} |f(s)-cs^{-1}|ds\left(|\gamma|
+2\int_{|x|<1}|x|\nu(dx)+c_3\int_{|x|\geqslant1}\nu(dx)\right)
\end{align*}
(note that $f(s)\geqslant c_4>0$) and
\begin{align*}
&\int_0^{s_0} cs^{-1}ds\left|\int_{\mathbb{R}^d}x\left(
\frac{1}{1+|f(s)x|^2}-\frac{1}{1+|cs^{-1}x|^2}\right)\nu(dx)\right|\\
&\qquad\leqslant \int_0^{s_0} cs^{-1}ds\int_{\mathbb{R}^d}\frac{|x|\,||cs^{-1}x|^2
-|f(s)x|^2|}{(1+|f(s)x|^2)(1+|cs^{-1}x|^2)}\nu(dx)\\
&\qquad\leqslant c_5\int_0^{s_0} |cs^{-1}-f(s)|ds\int_{\mathbb{R}^d}
\frac{|x|^3(cs^{-1})^2}{(1+|f(s)x|^2)(1+|cs^{-1}x|^2)}\nu(dx)\\
&\qquad\leqslant c_5\int_0^{s_0} |cs^{-1}-f(s)|ds\left(\int_{|x|<1}
|x|\nu(dx)+c_6\int_{|x|\geqslant1}\nu(dx)\right).
\end{align*}
Hence, under \eqref{d.5}, $I(t)$ is convergent if and only if 
$I_3(t)$ is convergent. We have, with $v=t^{-1}$,
\begin{align*}
I_3(t)&=\int_{u_0}^v cu^{-1}du\left(\gamma+\int_{\mathbb{R}^d}x\left(
\frac{1}{1+|cux|^2}-\frac{1}{1+|x|^2}\right)\nu(dx)\right)\\
&=-c^2\int_{u_0}^v c^{-1}u^{-1}du\left(\gamma'+\int_{\mathbb{R}^d}x\left(
\frac{1}{1+|c^{-1}u^{-1}x|^2}-\frac{1}{1+|x|^2}\right)\nu'(dx)\right).
\end{align*}
Let $h(u)=c^{-1}u^{-1}$ for $u\geqslant u_0=s_0^{-1}$. Then it follows that
\begin{align*}
&\mu\in\mathfrak D(\Phi_f)\quad\Leftrightarrow\quad A=0\text{ and }
\mu_{(0,\nu,\gamma)}\in\mathfrak D(\Phi_f)
\text{ $\Leftrightarrow$ }A=0\text{ and }
\mu_{(0,\nu',\gamma')}\in\mathfrak D(\Phi_h)\\
&\Leftrightarrow\quad A=0, \int_{|x|>1}|x|\nu'(dx)<\infty,
\lim_{v\to\infty}\int_{u_0}^v u^{-1}du\int_{|x|>u}x\nu'(dx)\text{ exists},
(\gamma^1)'=0\\
&\Leftrightarrow\quad\text{\eqref{d.5}, \eqref{d.8}, and $\gamma^0=0$}.
\end{align*}
It follows from this that} $\mu\in\mathfrak D(\Phi_{f,\,\mathrm{c}})$ 
if and only if \eqref{d.5} and \eqref{d.8} hold.
We get relation \eqref{d.6c}, using \eqref{5.9c} with $h$ in place
of $f$. Thus, in showing that $\mathfrak D^0(\Phi_f)\subsetneqq\mathfrak D(\Phi_f)$,
we make indirect use of the elaborate construction in (2006b) of
a distribution in the difference of the two classes.
\end{proof}

\begin{prop}\label{pd.3}
Let $f(s)$ be a measurable function on $(0,b)$ with $0<b<\infty$
such that
$\int_t^b f(s)^2 ds<\infty$ for all $t\in(0,b)$ and
\[
f(s)\asymp s^{-1/2}\quad\text{as }s\downarrow0.
\]
Then
\[
\mu_{(A,\nu,\gamma)}\in\mathfrak D(\Phi_{f,\,\mathrm{es}})\quad\Leftrightarrow\quad
A=0\text{ and }\int_{|x|<1} |x|^2\log(1/|x|)\nu(dx)<\infty,
\]
and \eqref{d.6a} holds.
\end{prop}

\begin{proof}
We have \eqref{d.11} with $\alpha=0$. Note that
$\int_0^{s_0} f(s)^2 ds=\infty$ and $\int_0^{s_0} f(s)ds<\infty$.
Let $g(u)=1/f(e^{-2u})$ for $u\geqslant u_0=s_0^{-1}$. Then 
$g(u)\asymp e^{-u}$ as $u\to\infty$,
since 
\[
\frac{g(u)}{e^{-u}}=\frac{e^{u}}{f(e^{-2u})}
=\frac{s^{-1/2}}{f(s)},
\]
where $s=e^{-2u}$. We have
{\allowdisplaybreaks
\begin{align*}
&\int_0^{s_0} ds\int_{\mathbb{R}^d}(|f(s)x|^2\land1)\nu(dx)
=\int_{u_0}^{\infty} 2e^{-2u}du
\int_{\mathbb{R}^d}\left(\left|\frac{x}{g(u)}
\right|^2\land1\right)\nu(dx)\\
&\qquad 
=\int_{u_0}^{\infty} \frac{2e^{-2u}}{g(u)^2}du
\int_{\mathbb{R}^d}(1\land|g(u)x|^2)\nu'(dx).
\end{align*}
Hence, using Example \ref{e5.2}, we have
\begin{align*}
&\mu\in\mathfrak D(\Phi_{f,\,\mathrm{es}})
\quad\Leftrightarrow\quad A=0,\,
\mu_{(0,\nu,\gamma)}\in\mathfrak D(\Phi_{f,\,\mathrm{es}})
\quad\Leftrightarrow\quad A=0,\,
\mu_{(0,\nu',\gamma')}\in\mathfrak D(\Phi_{g,\,\mathrm{es}})\\
&\qquad\Leftrightarrow\quad A=0,\,
\int_{|x|>1}\log |x|\nu'(dx)<\infty
\end{align*}
and the last integral equals $\int_{|x|<1} |x|^2\log(1/|x|)\nu(dx)$.
We have
\begin{align*}
&\int_0^{s_0} f(s)ds\left|\gamma+\int_{\mathbb{R}^d}x\left(
\frac{1}{1+|f(s)x|^2}-\frac{1}{1+|x|^2}\right)\nu(dx)\right|\\
&\quad\leqslant \int_0^{s_0} ds\int_{|x|\leqslant1}
\frac{|f(s)x|^3\nu(dx)}{(1+|f(s)x|^2)(1+|x|^2)}+c_3\\
&\quad\leqslant c_4\int_0^{s_0} ds\int_{|x|\leqslant1}
\frac{|s^{-1/2}x|^3\nu(dx)}{(1+c_5|s^{-1/2}x|^2)(1+|x|^2)}+c_3\\
&\quad\leqslant c_4\int_0^{s_0} ds\int_{|x|\leqslant1\land s^{1/2}}
|s^{-1/2}x|^3\nu(dx)+c_5\int_0^{s_0} ds\int_{s^{1/2}<|x|\leqslant1}
s^{-1/2}|x|\nu(dx)+c_3\\
&\quad= c_4\int_{|x|\leqslant1\land s_0^{1/2}}|x|^3\nu(dx)
\int_{|x|^2}^{s_0} 
s^{-3/2}ds+c_5\int_{|x|\leqslant1}|x|\nu(dx)
\int_0^{s_0\land|x|^2} s^{-1/2}ds+c_3,
\end{align*}
which is finite. 
Hence} $\mathfrak D^0(\Phi_f)=\mathfrak D(\Phi_{f,\,\mathrm{es}})$
and we obtain \eqref{d.6a}.
\end{proof}

\begin{prop}\label{pd.4}
Let $f(s)$ be a measurable function on $(0,b)$ with $0<b<\infty$
such that
$\int_t^b f(s)^2 ds<\infty$ for all $t\in(0,b)$ and
\[
f(s)\asymp s^{-1}(\log(1/s))^{-\beta}\quad\text{as }s\downarrow0
\]
with $\beta\in\mathbb{R}$. Then
\[
\mu_{(A,\nu,\gamma)}\in\mathfrak D(\Phi_{f,\,\mathrm{es}})\quad\Leftrightarrow\quad
A=0\text{ and }\int_{|x|<1/2} |x|(\log(1/|x|))^{-\beta}
\nu(dx)<\infty.
\]
If $0<\beta\leqslant1$, then \eqref{5.5b} holds.
If $1<\beta\leqslant2$, then \eqref{5.5c} holds.
\end{prop}

\begin{proof}
Choose $s_0$ large enough and let $u_0$ be such that $s_0=u_0^{-1}
(\log u_0)^{-2\beta}$. Let $g(u)=1/f(s)$ for $u\geqslant u_0$, where
$s=u^{-1}(\log u)^{-2\beta}$. Then we can prove that
$g(u)\asymp u^{-1}(\log u)^{-\beta}$ as $u\to\infty$,
\[
\int_0^{s_0} ds\int_{\mathbb{R}^d}(|f(s)x|^2\land1)\nu(dx)
=\int_{u_0}^{\infty} (-ds/du)g(u)^{-2}du
\int_{\mathbb{R}^d}(1\land|g(u)x|^2)\nu'(dx),
\]
and $(-ds/du)g(u)^{-2}\asymp 1$. Using Example \ref{e5.4} (i),
the rest of the proof is similar.

Proof of \eqref{5.5b} or \eqref{5.5c} for $0<\beta\leqslant1$ or
$1<\beta\leqslant2$, respectively, is done by the same method as the
proof of (ii) and (iii) of Example \ref{e5.4}.
\end{proof}

\section{Properties of $f$ and largeness of various domains}

Fix $-\infty\leqslant a<b\leqslant\infty$ and the dimension $d$. 
Let $f$ be an $\mathbb{R}$-valued measurable
function on $(a,b)$.
We use the subclasses $ID_0(\mathbb{R}^d)$ and $ID_{\mathrm{AB}}(\mathbb{R}^d)$
of $ID(\mathbb{R}^d)$ introduced in Section 1.

\begin{thm}\label{t6.1}
The following statements are equivalent.

{\rm(a)} \quad$ID(\mathbb{R}^d)=\mathfrak D^0 (\Phi_f)$.

{\rm(b)} \quad$ID(\mathbb{R}^d)=\mathfrak D (\Phi_{f,\,\mathrm{es}})$.

{\rm(c)} \quad$ID_0(\mathbb{R}^d)\subset\mathfrak D^0 (\Phi_f)$.

{\rm(d)} \quad$ID_0(\mathbb{R}^d)\subset\mathfrak D (\Phi_{f,\,\mathrm{es}})$.

{\rm(e)} \quad$\int_a^b 1_{\{f(s)\neq0\}} ds
<\infty$ and $\int_a^b f(s)^2 ds<\infty$.
\end{thm}

As we have the relation of the various domains in Theorem \ref{c5.1}, 
statement (b) or (d) with $\Phi_{f,\,\mathrm{es}}$ replaced by any one of 
$\Phi_f$, $\Phi_{f,\,\mathrm{c}}$, and $\Phi_{f,\,\mathrm{sym}}$ is also 
equivalent to statements (a)--(e).

\begin{proof}[Proof of Theorem \ref{t6.1}]
Clearly (a) $\Rightarrow$ (b) $\Rightarrow$ (d), and
(a) $\Rightarrow$ (c) $\Rightarrow$ (d). 
Therefore it is enough to show that (e) $\Rightarrow$ (a), 
and (d) $\Rightarrow$ (e).

Assume (e).  Then \eqref{4.1}, \eqref{4.2}, and \eqref{5.2} hold.
Indeed, \eqref{4.1}  is obvious, \eqref{4.2} follows as
\[
\int_a^b ds\int_{\mathbb{R}^d} (|f(s)x|^2\land1)\,\nu(dx)\leqslant \int_a^b
(f(s)^2+1)1_{\{f(s)\neq0\}} ds\int_{\mathbb{R}^d}(|x|^2\land1)\,\nu(dx)
<\infty,
\]
and \eqref{5.2}
follows as in the proof of Theorem \ref{t3.1} (i).
Hence we obtain (a) by virtue of Theorem
\ref{t5.1}.

Assume (d). Then, for every $\mu\in ID_0$, $f$ is in
$\mathbf L_{(a,b)}(X^{(\mu)})$ and $\Phi_{f,\,\mathrm{es}}(\mu)$ is
definable. Hence
\begin{equation}\label{6.1}
\int_a^b ds\int_{\mathbb{R}^d} (|f(s)x|^2\land1)\,\nu(dx)<\infty\quad
\text{for all }\nu\in\mathrm{Lvm}(ID(\mathbb{R}^d))
\end{equation}
by virtue of Theorem
\ref{t4.2}.  Let us show (e) in two steps.

{\it Step 1}.  Suppose that $\int_a^b f(s)^2 ds=\infty$. Let
\[
k(r)=\int_a^b f(s)^2 1_{\{|f(s)|\leqslant 1/r\}} ds\quad\text{for }
r>0.
\]
For every $\nu\in\mathrm{Lvm}(ID(\mathbb{R}^d))$ we have
{\allowdisplaybreaks
\begin{equation}\label{6.2}
\begin{split}
&\int_{|x|\leqslant1} |x|^2 k(|x|)\nu(dx)=\int_a^b ds\int_{|x|\leqslant1}
f(s)^2 1_{\{|f(s)|\leqslant 1/|x|\}} |x|^2 \nu(dx)\\
&\qquad \leqslant\int_a^b ds\int_{|x|\leqslant1}
(|f(s)x|^2\land1) \nu(dx)<\infty
\end{split}
\end{equation}
by \eqref{6.1}. Considering  Lebesgue measure on $\{|x|\leqslant1\}$ as
$\nu$, we see that $k(|x|)<\infty$ for almost every $x$ with 
$|x|\leqslant1$. Hence $k(r)<\infty$ for almost every $r$ in $(0,1]$.
Therefore $k(r)$ is finite for all $r>0$ and increases to $\infty$
as $r\downarrow0$. 
Choose $r_n$, $n=1,2,\ldots$, such that $1>r_n>r_{n+1}>0$ and
$k(r_n)\geqslant n$.  Let
$\rho=\sum_{n=1}^{\infty} n^{-2} \delta_{r_n}$ and let
\[
\nu(B)=\int_S \lambda(d\xi)\int_{(0,1]} 1_B(r\xi)\,r^{-2}\rho(dr)
\quad\text{for }B\in\mathcal B(\mathbb{R}^d),
\]
where $\lambda$ is a finite nonzero measure on the unit sphere
$S=\{\xi\in\mathbb{R}^d\colon |\xi|=1\}$. 
Then $\nu\in\mathrm{Lvm}(ID(\mathbb{R}^d))$, since
\[
\int_{|x|\leqslant1} |x|^2\nu(dx)=\int_S \lambda(d\xi)\int_{(0,1]}\rho(dr)
=\lambda(S)\sum_{n=1}^{\infty} n^{-2}<\infty.
\]
But
\begin{align*}
&\int_{|x|\leqslant1} |x|^2 k(|x|)\nu(dx)=\int_S \lambda(d\xi)\int_{(0,1]}
k(r)\rho(dr)\\
&\qquad=\lambda(S)\sum_{n=1}^{\infty}k(r_n)\rho(\{r_n\})
\geqslant\lambda(S)\sum_{n=1}^{\infty} n^{-1}=\infty,
\end{align*}
which contradicts \eqref{6.2}.  Therefore, $\int_a^b f(s)^2 ds<\infty$.}

{\it Step 2}. Suppose that $\int_a^b 1_{\{f(s)\neq0\}} ds=\infty$. Let
\[
h(r)=\int_a^b  1_{\{|f(s)|> 1/r\}} ds\quad\text{for }
r>0.
\]
We have $h(r)\leqslant r^2\int_a^b f(s)^2 ds<\infty$.
For every $\nu\in\mathrm{Lvm}(ID(\mathbb{R}^d))$ we have
{\allowdisplaybreaks
\begin{equation}\label{6.3}
\begin{split}
&\int_{|x|>1} h(|x|)\nu(dx)=\int_a^b ds \int_{|x|>1}
1_{\{|f(s)|> 1/|x|\}} \nu(dx)\\
&\qquad \leqslant\int_a^b ds \int_{|x|>1}
(|f(s)x|^2\land1) \nu(dx)<\infty
\end{split}
\end{equation}
by \eqref{6.1}. As $r\uparrow\infty$, $h(r)$ increases to $\infty$.
Choose $r_n$, $n=1,2,\ldots$, such that $1<r_n<r_{n+1}$ and
$h(r_n)\geqslant n$.  Let
$\rho=\sum_{n=1}^{\infty} n^{-2} \delta_{r_n}$ and let
\[
\nu(B)=\int_S \lambda(d\xi)\int_{(1,\infty)} 1_B(r\xi)\,\rho(dr)
\quad\text{for }B\in\mathcal B(\mathbb{R}^d),
\]
where $\lambda$ is a finite nonzero measure on $S$.
Then $\nu\in\mathrm{Lvm}(ID(\mathbb{R}^d))$ and
\begin{align*}
&\int_{|x|>1} h(|x|)\nu(dx)=\int_S \lambda(d\xi)\int_{(1,\infty)}
h(r)\rho(dr)\\
&\qquad=\lambda(S)\sum_{n=1}^{\infty}h(r_n)\rho(\{r_n\})
\geqslant\lambda(S)\sum_{n=1}^{\infty} n^{-1}=\infty,
\end{align*}
which contradicts \eqref{6.3}.  Therefore,} $\int_a^b 1_{\{f(s)\neq0\}} ds
<\infty$.

Steps 1 and 2 combined imply (e).
\end{proof}

\begin{thm}\label{t6.2}
The following statements are equivalent.

{\rm(a)}\quad $ID_{\mathrm{AB}}(\mathbb{R}^d)\subset\mathfrak D ^0(\Phi_f)$ 
and $\Phi_f(ID_{\mathrm{AB}}(\mathbb{R}^d))\subset ID_{\mathrm{AB}}(\mathbb{R}^d)$.

{\rm(b)}\quad $ID_{\mathrm{AB}}(\mathbb{R}^d)\subset\mathfrak D (\Phi_{f,\mathrm{es}})$
 and $\Phi_{f,\mathrm{es}}(ID_{\mathrm{AB}}(\mathbb{R}^d))
\subset ID_{\mathrm{AB}}(\mathbb{R}^d)$.

{\rm(c)}\quad\,$\int_a^b 1_{\{f(s)\neq0\}} ds
<\infty$ and $\int_a^b |f(s)|ds<\infty$.
\end{thm}

Statement (b) with $\Phi_{f,\mathrm{es}}$ replaced by any one of
$\Phi_f$, $\Phi_{f,\mathrm{c}}$, and $\Phi_{f,\mathrm{sym}}$ is equivalent to
statements (a)--(c).

\begin{proof}[Proof of Theorem \ref{t6.2}] 
Asssume (c). Let $\mu\in ID_{\mathrm{AB}}$. Then
$f\in\mathbf{L}_{(a,b)}(X^{(\mu)})$ from Theorem \ref{p7.2}. As
 \eqref{7.4} is clear and \eqref{7.3} follows from
\[
\int_a^b ds\int_{\mathbb{R}^d} (|f(s)x|\land1)\,\nu(dx)
\leqslant\int_a^b (|f(s)|+1)1_{\{f(s)\neq0\}}ds\int_{\mathbb{R}^d} (|x|\land1)\,\nu(dx),
\]
Theorem \ref{t4.8} is applicable and (a) of Theorem \ref{t4.8} holds 
for all $\mu\in ID_{\mathrm{AB}}$.
We have \eqref{5.2}, recalling the proof of Theorem \ref{p7.2}.
Thus, using Theorem \ref{t5.1}, we see that
 (a) is true. 

Clearly (a) implies (b).

Assume (b). Then
\[
\int_a^b ds\int_{\mathbb{R}^d} (|f(s)x|\land1)\,\nu(dx)<\infty
\quad\text{for all }\nu\in\mathrm{Lvm}(ID_{\mathrm{AB}}(\mathbb{R}^d)).
\]
Let us show (c) in two steps. The argument is similar to the
corresponding part of the proof of Theorem \ref{t6.1}, but we give a
complete proof.

{\it Step 1}.  Suppose that $\int_a^b |f(s)| ds=\infty$. Let
\[
k(r)=\int_a^b |f(s)| 1_{\{|f(s)|\leqslant 1/r\}} ds\quad\text{for }
r>0.
\]
For every $\nu\in\mathrm{Lvm}(ID_{\mathrm{AB}})$ we have
{\allowdisplaybreaks
\begin{align*}
&\int_{|x|\leqslant1} |x| k(|x|)\nu(dx)=\int_a^b ds\int_{|x|\leqslant1}
|f(s)x| 1_{\{|f(s)|\leqslant 1/|x|\}} \nu(dx)\\
&\qquad \leqslant\int_a^b ds\int_{|x|\leqslant1}
(|f(s)x|\land1) \nu(dx)<\infty.
\end{align*}
It follows that} $k(r)$ is finite 
for all $r>0$ and increases to $\infty$
as $r\downarrow0$. 
Choosing $r_n$, $n=1,2,\ldots$, such that $1>r_n>r_{n+1}>0$ and
$k(r_n)\geqslant n$, let
$\rho=\sum_{n=1}^{\infty} n^{-2} \delta_{r_n}$ and
\[
\nu(B)=\int_S \lambda(d\xi)\int_{(0,1]} 1_B(r\xi)\,r^{-1}\rho(dr),
\]
where $\lambda$ is a finite nonzero measure on the unit sphere $S$.
Then $\nu\in\mathrm{Lvm}(ID_{\mathrm{AB}})$, since
\[
\int_{|x|\leqslant1} |x|\nu(dx)=\int_S \lambda(d\xi)\int_{(0,1]}\rho(dr)
=\lambda(S)\sum_{n=1}^{\infty} n^{-2}<\infty.
\]
But
\[
\int_{|x|\leqslant1} |x| k(|x|)\nu(dx)=
\lambda(S)\sum_{n=1}^{\infty}k(r_n)\rho(\{r_n\})
\geqslant\lambda(S)\sum_{n=1}^{\infty} n^{-1}=\infty,
\]
a contradiction. Hence $\int_a^b |f(s)| ds<\infty$.

{\it Step 2}. Suppose that $\int_a^b 1_{\{f(s)\neq0\}}ds=\infty$. Let
\[
h(r)=\int_a^b  1_{\{|f(s)|> 1/r\}} ds\quad\text{for }
r>0.
\]
Then $h(r)\leqslant r\int_a^b |f(s)|ds<\infty$.
As $r\uparrow\infty$, $h(r)$ increases to $\infty$.
For every $\nu\in\mathrm{Lvm}(ID_{\mathrm{AB}})$ 
{\allowdisplaybreaks
\begin{align*}
&\int_{|x|>1} h(|x|)\nu(dx)=\int_a^b ds \int_{|x|>1}
1_{\{|f(s)|> 1/|x|\}} \nu(dx)\\
&\qquad \leqslant\int_a^b ds \int_{|x|>1}
(|f(s)x|\land1) \nu(dx)<\infty.
\end{align*}
Choosing} $r_n$, $n=1,2,\ldots$, such that $1<r_n<r_{n+1}$ and
$h(r_n)\geqslant n$, let
$\rho=\sum_{n=1}^{\infty} n^{-2} \delta_{r_n}$ and 
\[
\nu(B)=\int_S \lambda(d\xi)\int_{(1,\infty)} 1_B(r\xi)\,\rho(dr),
\]
with a finite nonzero measure $\lambda$ on $S$.
Then $\nu\in\mathrm{Lvm}(ID_{\mathrm{AB}})$ and
\[
\int_{|x|>1} h(|x|)\nu(dx)=\lambda(S) \sum_{n=1}^{\infty}
h(r_n)\rho(\{r_n\})
\geqslant\lambda(S)\sum_{n=1}^{\infty} n^{-1}=\infty,
\]
a contradiction.  Therefore, $\int_a^b 1_{\{f(s)\neq0\}}ds<\infty$.
\end{proof}

\begin{ex}\label{e6.5}
Let $(a,b)=(0,b)$ with $b$ finite.  Suppose that 
$\int_t^b f(s)^2ds<\infty$ for all $t\in(0,b)$ and that
$f(s)\asymp s^{-\beta}$ as $s\downarrow0$.  If $0<\beta<1/2$, then
$f$ satisfies (e) of Theorem \ref{t6.1}. If $0<\beta<1$, then
$f$ satisfies (c) of Theorem \ref{t6.2}.
\end{ex}

In relation to the two theorems above, it is interesting to consider
the condition that $ID_{\mathrm{AB}}\subset\mathfrak D 
(\Phi_{f,\mathrm{es}})$ and 
the condition that $ID_{\mathrm{AB}}\subset\mathfrak D^0 
(\Phi_f)$.  We have the following two theorems.

\begin{thm}\label{t6.3}
The following statements are equivalent.

{\rm(a)}\quad $ID_{\mathrm{AB}}(\mathbb{R}^d)\subset\mathfrak D 
(\Phi_{f,\mathrm{es}})$.

{\rm(b)}\quad\,The function $f(s)$ is locally integrable on $(a,b)$, 
$\int_a^b 1_{\{f(s)\neq0\}}ds<\infty$,
{\allowdisplaybreaks
\begin{align}
\int_a^b f(s)^2 1_{\{|f(s)|\leqslant 1/r\}} ds&=O(1/r)\quad\text{as }r\downarrow0,
\label{6.4}\\
\intertext{and}
\int_a^b  1_{\{|f(s)|> 1/r\}} ds&=O(r)\quad\text{as }r\downarrow0.\label{6.5}
\end{align}}
\end{thm}

\begin{proof}
Let
\[
k(r)=\int_a^b f(s)^2 1_{\{|f(s)|\leqslant 1/r\}} ds\quad\text{and}\quad
h(r)=\int_a^b  1_{\{|f(s)|> 1/r\}} ds.
\]

Assume (b). By virtue of Theorem \ref{p7.2}, $f\in\mathbf{L}_{(a,b)}
(X^{(\mu)})$ for any $\mu\in ID_{\mathrm{AB}}$. Thus, it follows from
Theorem \ref{t4.2} that
statement (a) is true if
\begin{equation}\label{6.6}
\int_a^b ds\int_{\mathbb{R}^d} (|f(s)x|^2\land1)\,\nu(dx)<\infty\quad
\text{for all }
\nu\in\mathrm{Lvm}(ID_{\mathrm{AB}}).
\end{equation}
Let $\nu\in\mathrm{Lvm}(ID_{\mathrm{AB}})$. Then
{\allowdisplaybreaks
\begin{align*}
&\int_a^b ds\int_{\mathbb{R}^d} (|f(s)x|^2\land1)\,\nu(dx)\\
&\quad\leqslant
\int_a^b ds\int_{|x|\leqslant1} (|f(s)x|^2\land1)\,\nu(dx)
+\int_a^b1_{\{f(s)\neq0\}}ds \int_{|x|>1} \nu(dx)
=I_1+I_2,
\end{align*}
say. Clearly $I_2$ is finite. Using \eqref{6.4}, \eqref{6.5}, and
$\int_a^b 1_{\{f(s)\neq0\}}ds<\infty$, we have
\begin{align*}
I_1&=\int_{|x|\leqslant1}|x|^2\nu(dx)\int_a^b f(s)^2
1_{\{|f(s)|\leqslant 1/|x|\}} ds
+\int_{|x|\leqslant1}\nu(dx)\int_a^b 1_{\{|f(s)|> 1/|x|\}} ds\\
&=\int_{|x|\leqslant1}|x|^2 k(|x|)\nu(dx)
+\int_{|x|\leqslant1}h(|x|)\nu(dx)
\leqslant c\int_{|x|\leqslant1} |x|\nu(dx)<\infty,
\end{align*}
where} $c$ is a constant. Hence \eqref{6.6} holds.  Hence (a) is true.

Conversely, assume (a).
Then, for every $\mu\in ID_{\mathrm{AB}}$ with triplet 
$(0,\nu,\gamma^0)_0$, $f\in\mathbf{L}_{(a,b)}(X^{(\mu)})$, and hence
\[
\int_p^q |f(s)|\left| \gamma^0+\int_{\mathbb{R}^d}\frac{x\nu(dx)}{1+|f(s)x|^2}
\right|ds<\infty
\]
for all $p,q$ with $a<p<q<b$ (recall Theorem \ref{t3.1}).  
Considering the case $\nu=0$, we see that $\int_p^q|f(s)|ds<\infty$,
that is, $f(s)$ is locally integrable on $(a,b)$.  Further
(a) implies \eqref{6.6}.  Thus,
for every $\nu\in\mathrm{Lvm}(ID_{\mathrm{AB}})$, 
{\allowdisplaybreaks
\begin{align*}
&\int_{|x|\leqslant1} |x|^2 k(|x|)\,\nu(dx)=\int_a^b f(s)^2 ds
\int_{|x|\leqslant1} |x|^2 1_{\{|f(s)|\leqslant 1/|x|\}} \nu(dx)\\
&\qquad \leqslant \int_a^b ds\int_{\mathbb{R}^d} (|f(s)x|^2\land1)\,\nu(dx)<\infty\\
\intertext{and}
&\int_{\mathbb{R}^d} h(|x|)\,\nu(dx)=\int_a^b ds
\int_{\mathbb{R}^d} 1_{\{|f(s)|> 1/|x|\}} \nu(dx)\\
&\qquad \leqslant \int_a^b ds\int_{\mathbb{R}^d} (|f(s)x|^2\land1)\,\nu(dx)<\infty.
\end{align*}
Using an appropriate} $\nu$, we see that $k(r)$ and $h(r)$ are finite
almost everywhere.

{\it Step 1}. Suppose that $\int_a^b 1_{\{f(s)\neq0\}}ds
=\infty$.  Then we have a contradiction
exactly in the same way as in Step 2 in the proof of Theorem \ref{t6.2}.
Hence $\int_a^b 1_{\{f(s)\neq0\}}ds<\infty$.

{\it Step 2}. Suppose that $\limsup_{r\downarrow0} r^{-1}h(r)=\infty$.
Choose a sequence $r_n\leqslant1$ decreasing to $0$ such that ${r_n}^{-1}
h(r_n)\geqslant n$.  Let $\rho=\sum_{n=1}^{\infty} n^{-2}\delta_{r_n}$ and
\[
\nu(B)=\int_S \lambda(d\xi)\int_{(0,1]} 1_B(r\xi) r^{-1}\rho(dr)
\]
 with
a finite nonzero measure $\lambda$ on $S$.  Then 
$\nu\in\mathrm{Lvm}(ID_{\mathrm{AB}})$ but
\[
\int_{|x|\leqslant1} h(|x|)\,\nu(dx)=\lambda(S)\sum_{n=1}^{\infty} h(r_n){r_n}^{-1}
n^{-2}\geqslant \lambda(S)\sum_{n=1}^{\infty} n^{-1}=\infty,
\]
which is a contradiction.  Hence we obtain \eqref{6.5}.

{\it Step 3}. Suppose that $\limsup_{r\downarrow0} r\,k(r)=\infty$.
Choose a sequence $r_n\leqslant1$ decreasing to $0$ such that $r_n
k(r_n)\geqslant n$.  Define $\rho$ and $\nu$ by the same formulas
 as in Step 2.  Then 
$\nu\in\mathrm{Lvm}(ID_{\mathrm{AB}})$ but
\[
\int_{|x|\leqslant1} |x|^2 k(|x|)\,\nu(dx)=\lambda(S)\sum_{n=1}^{\infty}
r_n k(r_n)n^{-2}\geqslant \lambda(S)\sum_{n=1}^{\infty} n^{-1}=\infty,
\]
a contradiction.  Thus \eqref{6.4} is true.
\end{proof}

\begin{thm}\label{t6.4}
The following statements are equivalent.

{\rm(a)}\quad $ID_{\mathrm{AB}}(\mathbb{R}^d)\subset\mathfrak D ^0(\Phi_f)$.

{\rm(b)}\quad\,$\int_a^b 1_{\{f(s)\neq0\}}ds<\infty$ and $\int_a^b |f(s)|ds<\infty$.
\end{thm}

\begin{proof}
Assume (a). Let $\nu\in\mathrm{Lvm}(ID_{\mathrm{AB}})$. Then, for
any $\gamma^1$ and $\gamma^2$ in $\mathbb{R}^d$, $\mu_{(0,\nu,\gamma^1)}$ and
$\mu_{(0,\nu,\gamma^1)}$ are in $\mathfrak D ^0(\Phi_f)$.
Recall Theorem \ref{t5.1}. We see that \eqref{5.2} is true for
$\gamma=\gamma^1$ and $\gamma=\gamma^2$. Hence $\int_a^b |f(s)(\gamma^1-\gamma^2)|ds
<\infty$.
It follows that $\int_a^b |f(s)|ds<\infty$. 
We can prove that $\int_a^b 1_{\{f(s)\neq0\}}ds
<\infty$ in the same way as Step 2 in the proof of 
Theorem \ref{t6.1}.

Conversely, assume (b).  Then, 
$ID_{\mathrm{AB}}\subset\mathfrak D ^0(\Phi_f)$ and
$\Phi_f(ID_{\mathrm{AB}})\subset ID_{\mathrm{AB}}$ by virtue of
Theorem \ref{t6.2}.  Thus, a fortiori, (a) holds.
\end{proof}

\begin{rem}\label{r6.1}
Consider the following conditions:
{\allowdisplaybreaks
\begin{gather}
ID=\mathfrak D ^0(\Phi_f),\label{6.7a}\\
ID=\mathfrak D (\Phi_{f,\mathrm{es}}),\label{6.7}\\
ID_{\mathrm{AB}}\subset\mathfrak D ^0(\Phi_f)\quad\text{and}\quad
\Phi_f(ID_{\mathrm{AB}})\subset ID_{\mathrm{AB}},\label{6.8a}\\
ID_{\mathrm{AB}}\subset\mathfrak D (\Phi_{f,\mathrm{es}})\quad\text{and}\quad
\Phi_{f,\mathrm{es}}(ID_{\mathrm{AB}})\subset ID_{\mathrm{AB}},\label{6.8}\\
ID_{\mathrm{AB}}\subset \mathfrak D ^0(\Phi_f),\label{6.10}\\
ID_{\mathrm{AB}}\subset \mathfrak D (\Phi_{f,\mathrm{es}}).\label{6.9}
\end{gather}
Then, it follows from the theorems in this section that}
\[
\text{\eqref{6.7a} $\Leftrightarrow$ \eqref{6.7} 
$\Rightarrow$ \eqref{6.8a} $\Leftrightarrow$ \eqref{6.8} 
$\Leftrightarrow$ \eqref{6.10} $\Rightarrow$ \eqref{6.9}.}
\]
Further, we can show that
condition \eqref{6.7} is strictly stronger than condition \eqref{6.8}
and that condition \eqref{6.8} is strictly stronger than condition 
\eqref{6.9}.
They are proved by the use of the analytical expressions of the conditions.
Indeed, it is obvious that \eqref{6.7} is strictly stronger than \eqref{6.8},
since there is $f(s)$ on a finite open interval $(a,b)$
 such that $\int_a^b |f(s)|ds
<\infty$ but $\int_a^b f(s)^2ds=\infty$.  
To show that \eqref{6.8} is strictly stronger than \eqref{6.9}, consider 
the function $f(s)$ in Example \ref{e6.1} or \ref{e6.2} below.
\end{rem}

\begin{ex}\label{e6.1}
Let $f(s)$ be as in Proposition \ref{pd.2} (iv). Thus $(a,b)=(0,b)$ 
with $b$ finite and $f(s)\asymp s^{-1}$ as $s\downarrow0$.  Then, 
$f(s)$ satisfies \eqref{6.4} and \eqref{6.5} and
$\int_0^1 f(s)ds=\infty$. Hence this $f(s)$ satisfies \eqref{6.9} but
does not satisfy \eqref{6.8}.  We have shown 
$ID_{\mathrm{AB}}=\mathfrak D(\Phi_{f,\mathrm{es}})$
in Proposition \ref{pd.2}. We have
{\allowdisplaybreaks
\begin{equation}\label{6.12}
\begin{split}
&\{\mu\in ID_{\mathrm{AB}}\colon \Phi_{f,\mathrm{es}}(\mu)\subset
ID_{\mathrm{AB}}\}\\
&\qquad=\left\{\mu=\mu_{(0,\nu,\gamma)}\in ID_{\mathrm{AB}}\colon
\int_{|x|\leqslant1/2} |x|\log(1/|x|)\,\nu(dx)<\infty\right\}.
\end{split}
\end{equation}
This again shows that} 
$\Phi_{f,\mathrm{es}}(ID_{\mathrm{AB}})\not\subset ID_{\mathrm{AB}}$.

Let us prove \eqref{6.12}. Let $\mu\in ID_{\mathrm{AB}}$. 
Let $\widetilde\nu$ denote the L\'evy measure of $\Phi_{f,\mathrm{es}}(\mu)$.
Then $\widetilde\nu(B)=\int_0^b ds\int_{\mathbb{R}^d} 1_B(f(s)x)\,\nu(dx)$.
Hence
{\allowdisplaybreaks
\begin{align*}
&\int_{|x|\leqslant1} |x|\widetilde\nu(dx)\leqslant c_2\int_0^{s_0} ds
\int_{\mathbb{R}^d}
 |s^{-1}x| 1_{\{c_1|s^{-1}x|\leqslant1\}}\nu(dx)+c_3\\
&\quad =c_2\int_{|x|\leqslant c_1^{-1}s_0} |x|\nu(dx)\int_{c_1|x|}^{s_0}
s^{-1}ds+c_3 \leqslant c_2\int_{|x|\leqslant c_1^{-1}s_0} |x|\log(1/|x|)\,\nu(dx)
+c_4,
\end{align*}
and the converse estimate is similar.}
\end{ex}

\begin{ex}\label{e6.2}
Let $(a,b)=(0,b)$ with $b$ finite and $f(s)\asymp s^{-1}(\log(1/s))^{-1}$
as $s\downarrow0$. We can prove that
{\allowdisplaybreaks
\begin{equation}\label{6.13}
\begin{split}
\mathfrak D(\Phi_{f,\mathrm{es}})=\{\mu&=\mu_{(A,\nu,\gamma)}\in ID
\colon A=0\text{ and}\\
&\int_{|x|\leqslant1/2} |x|(\log(1/|x|))^{-1}\nu(dx)
<\infty\}
\end{split}
\end{equation}
and
\begin{equation}\label{6.14}
\begin{split}
&\{\mu\in ID_{\mathrm{AB}}\colon \Phi_{f,\mathrm{es}}(\mu)\subset
ID_{\mathrm{AB}}\}\\
&\qquad=\left\{\mu=\mu_{(0,\nu,\gamma)}\in ID_{\mathrm{AB}}\colon
\int_{|x|\leqslant1/(2e)} |x|\log\log(1/|x|)\,\nu(dx)<\infty\right\}.
\end{split}
\end{equation}
It follows from} \eqref{6.13} that
\begin{equation}\label{6.15}
ID_{\mathrm{AB}}\subsetneqq \mathfrak D(\Phi_{f,\mathrm{es}}).
\end{equation}
The assertion \eqref{6.14} implies that $\Phi_{f,\mathrm{es}}
(ID_{\mathrm{AB}})\not\subset
ID_{\mathrm{AB}}$, but this fact follows also from
$\int_0^{1/2} f(s) ds=\infty$.
Thus, like Example \ref{e6.1}, this example satisfies 
\eqref{6.9} and
does not satisfy \eqref{6.8}. 
However, property \eqref{6.15} differs from property 
$ID_{\mathrm{AB}}=\mathfrak D(\Phi_{f,\mathrm{es}})$
of Example \ref{e6.1}.

Let us prove \eqref{6.13}. 
Let $\mu=\mu_{(A,\nu,\gamma)}$. Then $\mu\in\mathfrak D(\Phi_{f,\mathrm{es}})$
if and only if $A=0$ and \eqref{4.2} holds.  We have
\[
c_1 s^{-1}(\log(1/s))^{-1}\leqslant f(s)\leqslant c_2 s^{-1}(\log(1/s))^{-1}
\]
for $0<s<s_0$ with $c_1,c_2>0$ and $s_0\in(0,b\land(1/2))$.
Now
{\allowdisplaybreaks
\begin{align*}
&\int_0^{s_0} ds\int_{|f(s)x|>1} \nu(dx)=\int_{|x|>1}\nu(dx)\int_0^{s_0} 
1_{\{|f(s)x|>1\}}ds\\
&\qquad+\int_{|x|\leqslant1}\nu(dx)\int_0^{s_0} 1_{\{|f(s)x|>1\}}ds
=I_1+I_2\quad\text{(say)},\\
&\int_0^{s_0} ds\int_{|f(s)x|\leqslant1} |f(s)x|^2\nu(dx)=\int_{|x|>1}\nu(dx)
\int_0^{s_0} |f(s)x|^2
1_{\{|f(s)x|\leqslant1\}}ds\\
&\qquad+\int_{|x|\leqslant1}|x|^2\nu(dx)\int_0^{s_0} f(s)^2
1_{\{|f(s)x|\leqslant1\}}ds=J_1+J_2\quad\text{(say)}.
\end{align*}
Both $I_1$ and $J_1$ are bounded by $s_0\int_{|x|>1}\nu(dx)$.
We have
\begin{gather*}
I_2\leqslant\int_{|x|\leqslant1}\nu(dx)\int_0^{s_0} 1_{\{s\log (1/s)<c_2|x|\}}ds,\\
J_2\leqslant c_2\int_{|x|\leqslant1}|x|^2\nu(dx)\int_0^{s_0} s^{-2}
(\log (1/s))^{-2}
1_{\{s\log (1/s)\geqslant c_1|x|\}}ds,
\end{gather*}
and similar estimates from below. Then $I_2+J_2<\infty$ if and only if
\[
\int_{|x|\leqslant1/2} |x|(\log(1/|x|))^{-1}\nu(dx)
<\infty,
\]
since, letting $u=s\log(1/s)$, we have $du/ds=\log(1/s)-1\sim
\log(1/u)$ as $s\downarrow0$ (equivalently, as $u\downarrow0$), and since,
as $r\downarrow0$,
\[
\int_0^{s_0} 1_{\{s\log (1/s)<r\}}ds\sim \int_0^r\left(\frac{du}{ds}
\right)^{-1} du\sim \int_0^r (\log(1/u))^{-1} du\sim r(\log(1/r))^{-1}
\]
and
\begin{align*}
&\int_0^{s_0} s^{-2}(\log (1/s))^{-2} 1_{\{s\log (1/s)\geqslant r\}}ds\sim 
\int_r^{\varepsilon} u^{-2}\left(\frac{du}{ds}
\right)^{-1} du\\
&\qquad\sim \int_r^{\varepsilon}u^{-2} (\log(1/u))^{-1} du\sim r^{-1}
(\log(1/r))^{-1}
\end{align*}
with a small number $\varepsilon>0$.}
 
Proof of \eqref{6.14} is as follows.
Let $\mu\in ID_{\mathrm{AB}}$. Then
$\mu\in \mathfrak D(\Phi_{f,\mathrm{es}})$ from \eqref{6.13}.
For the L\'evy measure $\widetilde\nu$ of any distribution in 
$\Phi_{f,\mathrm{es}}(\mu)$, we have
{\allowdisplaybreaks
\begin{align*}
&\int_{|x|\leqslant1} |x|\widetilde\nu(dx)\leqslant \int_{\mathbb{R}^d} c_2|x|
\nu(dx)\int_0^{s_0}
s^{-1}(\log(1/s))^{-1} 1_{\{s\log(1/s)\geqslant c_1|x|\}} ds+c_3\\
&\qquad=I_1+I_2+c_3,
\end{align*}
where $I_1$ and $I_2$ are the repeated integrals with the integration 
over $\mathbb{R}^d$ replaced by that over $\{|x|>1\}$ and $\{|x|\leqslant1\}$,
respectively. Then
\[
I_1\leqslant c_1^{-1}c_2\int_{|x|>1}\nu(dx)\int_0^{s_0} ds <\infty
\]
and, with $u=s\log(1/s)$ and some $\varepsilon>0$,
\begin{align*}
&\int_0^{s_0} s^{-1}(\log(1/s))^{-1} 1_{\{s\log(1/s)\geqslant r\}}ds
\sim \int_r^{\varepsilon} u^{-1}\left(\frac{du}{ds}\right)^{-1} du\\
&\qquad \sim \int_r^{\varepsilon} u^{-1}(\log(1/u))^{-1}du\sim\log\log(1/r)
\end{align*}
as $r\downarrow0$.  Hence,} $I_2<\infty$ if and only if 
\[
\int_{|x|\leqslant1/(2e)} |x|\log\log(1/|x|)\,\nu(dx)<\infty.
\]
The estimate of $\int_{|x|\leqslant1} |x|\widetilde\nu(dx)$ from below is 
similar.
\end{ex}

\begin{ex}\label{e6.3}
Let $g_{\alpha}(u)=\int_u^{\infty} v^{-{\alpha}-1} e^{-v}dv$ for
$u\in(0,\infty)$ with $\alpha\in\mathbb{R}$. Let $b_{\alpha}=g_{\alpha}(0+)$.
Thus $b_{\alpha}=\Gamma(-{\alpha})$ for $\alpha<0$ and $b_{\alpha}=\infty$ for 
$\alpha\geqslant0$. Consider $(0,b_{\alpha})$ as $(a,b)$.
Let $u=f_{\alpha}(s)$, $s\in(0,b_{\alpha})$, be 
the inverse function of $s=g_{\alpha}(u)$, $u\in(0,\infty)$. Then
$f_{\alpha}(s)$ strictly decreases from $\infty$ to $0$ as $s$ goes 
from $0$ to $b_{\alpha}$.  We have
\begin{equation*}
f_{\alpha}(s)\sim \log(1/s)\quad\text{as $s\downarrow0$ for 
$\alpha\in\mathbb{R}$},
\end{equation*}
as Proposition 1.1 of (2006b) says. It follows that, for 
$f=f_{\alpha}$ with $\alpha<0$, \eqref{6.7a}--\eqref{6.9} hold.
For $\alpha\geqslant0$, Proposition 1.1 of (2006b) shows that, as $s\uparrow\infty$,
{\allowdisplaybreaks
\begin{gather*}
f_0(s)\sim c e^{-s}\quad\text{with a constant $c>0$},\\
f_{\alpha}(s)\sim (\alpha s)^{-1/\alpha}\quad\text{for $\alpha>0$},\\
f_1(s)=s^{-1}-s^{-2}\log s+o(s^{-2}\log s).
\end{gather*}
The five domains
in Theorem \ref{c5.1} for} $f=f_{\alpha}$ on $(0,\infty)$ with
$\alpha\geqslant0$ are described in (2006b), Theorems 2.4 and 2.8.
The case $\alpha=1$ is a special case of Example \ref{e5.1}.
We have $f_{-1}(s)=\log(1/s)$ for $s\in(0,b_{-1})=(0,1)$.
This case ($\alpha=-1$) was first introduced 
in Barndorff-Nielsen and Thorbj\o rnsen (2002) with the notation 
$\Upsilon$ for $\Phi_{f_{-1}}$ and explored by Barndorff-Nielsen, Maejima,
and Sato (2006) in connection with the Goldie--Steutel--Bondesson
class $B(\mathbb{R}^d)$ and the Thorin class $T(\mathbb{R}^d)$.
\end{ex}

\begin{ex}\label{e6.4}
Let $g_{\beta,\alpha}(u)=(\Gamma(\alpha-\beta))^{-1}\int_u^1 (1-v)^{\alpha-\beta-1}
 v^{-{\alpha}-1} dv$ for
$u\in(0,1)$ with $-\infty<\beta<\alpha<\infty$. Let $b_{\beta.\alpha}
=g_{\beta,\alpha}(0+)$, which equals $\Gamma(-{\alpha})/\Gamma(-\beta)$ for 
$\alpha<0$ and $\infty$ for 
$\alpha\geqslant0$. Let $u=f_{\beta,\alpha}(s)$, $s\in(0,b_{\beta,\alpha})$, be 
the inverse function of $s=g_{\beta,\alpha}(u)$, $u\in(0,1)$. Then
$f_{\beta,\alpha}(s)$ strictly decreases from $1$ to $0$ as $s$ goes 
from $0$ to $b_{\beta,\alpha}$.  Now Proposition 1.1 of (2006b) says that, 
as $s\uparrow\infty$,
{\allowdisplaybreaks
\begin{gather*}
f_{\beta,0}(s)\sim c_{\beta} e^{-\Gamma(-\beta)s}\quad\text{with a constant 
$c_{\beta}>0$ for $\beta<0$},\\
f_{\beta,\alpha}(s)\sim (\alpha\Gamma(\alpha-\beta) s)^{-1/\alpha}\quad
\text{for $\alpha>0$ and $\beta<\alpha$},\\
f_{\beta,1}(s)=(\Gamma(1-\beta))^{-1}s^{-1}+\beta(\Gamma(1-\beta))^{-2}s^{-2}\log s
+o(s^{-2}\log s)\quad\text{for $\beta<1$}.
\end{gather*}
From these behaviors Theorems 2.4 and 2.8 of (2006b) show that
the five domains in Theorem \ref{c5.1} for} $f=f_{\beta,\alpha}$ on 
$(0,\infty)$ with $\alpha\geqslant0$ do not depend on $\beta$ and
are the same as those for $f=f_{\alpha}$ on $(0,\infty)$.
We have $f_{-1,0}(s)=e^{-s}$ and thus $\Phi_{f_{-1,0}}$ equals $\Phi$
of Example \ref{e5.2}.
The family $\{\Phi_{f_{\beta,\alpha}}\}$ has a close connection with the 
family $\{\Phi_{f_{\alpha}}\}$ in Example \ref{e6.3}. Namely,
Theorem 3.1 of (2006b) proves that
\[
\Phi_{f_{\alpha}}=\Phi_{f_{\beta}}\Phi_{f_{\beta,\alpha}}=\Phi_{f_{\beta,\alpha}}
\Phi_{f_{\beta}}
\quad\text{ for $-\infty<\beta<\alpha<\infty$},
\]
including the equality of the domains of both sides.
A special case of this equality with $\alpha=0$ and $\beta=-1$ is given in
Barndorff-Nielsen, Maejima, and Sato (2006).
\end{ex}

At the end of this section, let us consider the case where 
$\mathfrak D (\Phi_{f,\,\mathrm{es}})$ is very small.

\begin{thm}\label{t6.5}
$\mathfrak D (\Phi_{f,\,\mathrm{es}})$ equals the class
$\{\delta_{\gamma}\colon \gamma\in\mathbb{R}^d\}$ if and only if
\begin{equation}\label{6.16}
\text{$f(s)$ is locally integrable on $(a,b)$
and $\int_a^b (f(s)^2\land1) ds=\infty$.}
\end{equation}
\end{thm}

Notice that $\int_a^b (f(s)^2\land1) ds=\infty$ implies
$\int_a^b f(s)^2ds=\infty$ and $b-a=\infty$, but that the converse is
not true. For example, consider $(a,b)=(0,\infty)$ and
$f(s)=\sum_{n=1}^{\infty} n1_{[n,n+n^{-2})}(s)$.

\begin{proof}[Proof of Theorem \ref{t6.5}]  We use Theorems \ref{t3.1}
and \ref{t4.2}.

The \lq\lq only if" part: The function
 $f(s)$ is locally integrable on $(a,b)$, since 
$f\in\mathbf L_{(a,b)}(X^{(\mu)})$ for any $\mu=\delta_{\gamma}$ with 
$\gamma\in\mathbb{R}^d$.  Further $\int_a^b (f(s)^2\land1) ds=\infty$,
because otherwise $\mu=\mu_{(0,\nu,0)}$ with $\nu=\delta_{x_0}$, 
$|x_0|=1$, belongs to $\mathfrak D (\Phi_{f,\,\mathrm{es}})$.
Hence \eqref{6.16} holds.

The \lq\lq if" part: Let $\mu=\mu_{(A,\nu,\gamma)}\in\mathfrak D 
(\Phi_{f,\,\mathrm{es}})$.  Then $A=0$, since $\int_a^b f(s)^2ds=\infty$.
If $\nu\neq0$, then we have, using $0<c\leqslant1$ such that $\nu(\{|x|>c\})
>0$, 
{\allowdisplaybreaks
\begin{align*}
\infty&>\int_a^b ds\int_{\mathbb{R}^d}(|f(s)x|^2\land1)\,\nu(dx)
\geqslant\int_a^b ds\int_{|x|>c}((f(s)^2c^2)\land1)\,\nu(dx)\\
&\geqslant c^2\int_{|x|>c}\nu(dx)\int_a^b (f(s)^2\land1) ds,
\end{align*}
which contradicts \eqref{6.16}. Hence  $\mu=\delta_{\gamma}$. Conversely, any} 
$\mu=\delta_{\gamma}$ is in $\mathfrak D (\Phi_{f,\,\mathrm{es}})$.
\end{proof}

\begin{rem}\label{r6.2}
(i) $\mathfrak D (\Phi_{f,\,\mathrm{c}})=\{\delta_{\gamma}\colon 
\gamma\in\mathbb{R}^d\}$ if and only
if \eqref{6.16} holds.  

(ii) $\mathfrak D^0 (\Phi_f)=\{\delta_{\gamma}\colon \gamma\in\mathbb{R}^d\}$ 
if and only if 
\eqref{6.16} holds and $\int_a^b |f(s)|ds<\infty$. 

(iii) $\mathfrak D (\Phi_f)=\{\delta_{\gamma}\colon \gamma\in\mathbb{R}^d\}$ 
if and only if
\eqref{6.16} holds and $\int_{a+}^{b-} f(s)ds$ exists in $\mathbb{R}$.

(iv) $\mathfrak D^0(\Phi_f)=\{\delta_0\}$ if and only if 
$\int_a^b (f(s)^2\land1) ds=\infty$ and $\int_a^b |f(s)|ds=\infty$.

(v)  Assume that $f(s)$ is locally integrable on $(a,b)$.  Then,
$\mathfrak D(\Phi_f)=\{\delta_0\}$ if and only if 
$\int_a^b (f(s)^2\land1) ds=\infty$ and $\int_p^q f(s)ds$ is not
convergent in $\mathbb{R}$ as $p\downarrow a$ and $q\uparrow b$.

These facts are proved similarly to Theorem \ref{t6.5}. Use $\nu=
\delta_{x_0}+\delta_{-x_0}$ with $|x_0|=1$ instead of $\nu=\delta_{x_0}$.
\end{rem}

\begin{ex}\label{e6.6}
Let $(a,b)=(a,\infty)$ with $a$ finite.  Suppose that 
$\int_a^t f(s)^2ds<\infty$ for all $t\in(a,\infty)$ and that
$f(s)\asymp s^{-1/\alpha}$ as $s\to\infty$.  If $\alpha\geqslant2$, then
\[
\mathfrak D^0(\Phi_f)= \mathfrak D(\Phi_f)=\{\delta_0\}\subsetneqq 
\mathfrak D(\Phi_{f,\,\mathrm{c}})= \mathfrak D(\Phi_{f,\,\mathrm{es}})
=\{\delta_{\gamma}\colon \gamma\in\mathbb{R}^d\}.
\]
\end{ex}

\section{The $\tau$-measure  of function $f$}

Let us define the $\tau$-measure of $f$.  
Functions $f(s)$ and $f_j(s)$ in the following are $\mathbb{R}$-valued. 
We fix the dimension $d$ in $ID(\mathbb{R}^d)$.

\begin{defn}\label{dt.1}
Let $f(s)$ be a measurable function $f(s)$ on $(a,b)$ with $-\infty\leqslant a
<b\leqslant\infty$. Define a measure $\tau$ on $\mathbb{R}$ as
\begin{equation}\label{8.7}
\tau(B)=\int_a^b 1_{\{f(s)\in B\}}ds,\qquad B\in\mathcal B_{\mathbb{R}}.
\end{equation}
We call $\tau$ the {\it $\tau$-measure} of function $f$.
\end{defn}

It follows from \eqref{8.7} that
\begin{equation}\label{8.7a}
\int_{\mathbb{R}} h(u)\tau(du)=\int_a^b h(f(s))ds
\end{equation}
for all nonnegative measurable functions $h$ on $\mathbb{R}$.

\medskip
We discuss two questions. The first is whether the $\tau$-measure $\tau$ 
of $f$ determines the domain $\mathfrak D(\Phi_f)$
and its variants. The second is under what conditions a given measure
$\tau$ is the $\tau$-measure of some $f$.

\begin{thm}\label{tt.1}
Suppose that $f_1(s)$ and $f_2(s)$ are measurable functions on $(a_1,b_1)$
and $(a_2,b_2)$, respectively, with identical $\tau$-measure.  Then,
{\allowdisplaybreaks
\begin{gather}
\mathfrak D^0(\Phi_{f_1})=\mathfrak D^0(\Phi_{f_2}),\label{t.1}\\
\Phi_{f_1}(\mu)=\Phi_{f_2}(\mu)\quad\text{for all }\mu\in
\mathfrak D^0(\Phi_{f_1})=\mathfrak D^0(\Phi_{f_2}),\label{t.2}\\
\mathfrak D(\Phi_{f_1,\mathrm{es}})=\mathfrak D(\Phi_{f_2,\mathrm{es}}),\label{t.3}\\
\Phi_{f_1,\mathrm{es}}(\mu)=\Phi_{f_2,\mathrm{es}}(\mu)\quad\text{for all }\mu\in
\mathfrak D(\Phi_{f_1,\mathrm{es}})=\mathfrak D(\Phi_{f_2,\mathrm{es}}).\label{t.4}
\end{gather}}
\end{thm}

\begin{proof}
 It follows from \eqref{8.7a} that, for any $\mu\in ID(\mathbb{R}^d)$,
\[
\int_{a_j}^{b_j} |C_{\mu}(f_j(s)z)|ds=\int_{\mathbb{R}} |C_{\mu}(uz)|\,\tau(du)\quad
\text{for }j=1,2,\;z\in\mathbb{R}^d.
\]
Hence we obtain \eqref{t.1} from Definition \ref{d5.1}.

Proof of \eqref{t.2}.  Let $\mu\in
\mathfrak D^0(\Phi_{f_1})=\mathfrak D^0(\Phi_{f_2})$. We have $\int_{\mathbb{R}} 
|C_{\mu}(uz)|\,\tau(du)<\infty$. Hence $C_{\mu}(f_j(s)z)$ is integrable
on $(a_j,b_j)$ and it follows from \eqref{8.7a} that
\[
\int_{a_j}^{b_j} C_{\mu}(f_j(s)z)ds=\int_{\mathbb{R}} C_{\mu}(uz)\,\tau(du).
\]
Denote $\widetilde\mu_j=\Phi_{f_j}(\mu)$. Then
\[
C_{\widetilde\mu_j}(z)=\lim_{p\downarrow a_j,q\uparrow b_j}\int_p^q C_{\mu}(f_j(s)z)ds
=\int_{a_j}^{b_j} C_{\mu}(f_j(s)z)ds.
\]
Thus $\widetilde\mu_1=\widetilde\mu_2$.

Proof of \eqref{t.3}. Notice that
\begin{gather*}
\int_{a_j}^{b_j} f_j(s)^2\,\mathrm{tr}\, A\,ds=\int_{\mathbb{R}}
u^2\,\mathrm{tr}\, A\,\tau(du),\\
\int_{a_j}^{b_j} ds\int_{\mathbb{R}^d} (|f_j(s)x|^2\land1)\,\nu(dx)
=\int_{\mathbb{R}}\tau(du)\int_{\mathbb{R}^d} (|ux|^2\land1)\,\nu(dx).
\end{gather*}
Use Theorem \ref{t4.2}.

Proof of \eqref{t.4}. Notice that
\begin{gather*}
\int_{a_j}^{b_j} f_j(s)^2\,A\,ds=\int_{\mathbb{R}}u^2\,A\,\tau(du),\\
\int_{a_j}^{b_j} ds\int_{\mathbb{R}^d} 1_B(f_j(s)x)\,\nu(dx)
=\int_{\mathbb{R}}\tau(du)\int_{\mathbb{R}^d} 1_B(ux)\,\nu(dx)
\end{gather*}
for $B\in\mathcal B(\mathbb{R}^d)$ with $0\not\in B$, and use Theorem \ref{t4.6}.
\end{proof}

\begin{prop}\label{pt.1}
{\rm(i)} There are measurable functions $f_1(s)$ and $f_2(s)$ on $(0,\infty)$
with identical $\tau$-measure, 
such that $\mathfrak D(\Phi_{f_1})\neq  \mathfrak D(\Phi_{f_2})$.

{\rm(ii)} There are measurable functions $f_1(s)$ and $f_2(s)$ on $(0,\infty)$
with identical $\tau$-measure, 
such that $\Phi_{f_1}(\mu)\neq \Phi_{f_2}(\mu)$ for some $\mu\in
\mathfrak D(\Phi_{f_1})\cap \mathfrak D(\Phi_{f_2})$.
\end{prop}

\begin{proof}
We prove (i) and (ii) together. In the following $f_3$ serves as $f_2$ in (ii).
Let $f_1(s)=s^{-1}\sin s$ for $s\in(0,\infty)$. Let
$c_n=\int_{n\pi}^{(n+1)\pi}s^{-1}\sin s\,ds$. Then
\[
\int_0^{\infty-} f_1(s)ds=\sum_{n=0}^{\infty} c_n=\frac{\pi}{2}.
\]
Let $\sum_{m=0}^{\infty} d_m$ be a rearrangement of $\sum_{n=0}^{\infty}c_n$
(that is for each $m$ there is a unique $n$ such that $d_m=c_n$, and
for each $n$ there is a unique $m$ such that $c_n=d_m$), such that
$\sum_{m=0}^{\infty} d_m=\infty$. 
Let $\sum_{m=0}^{\infty} e_m$ be a rearrangement of $\sum_{n=0}^{\infty}c_n$
such that
$\sum_{m=0}^{\infty} e_m=0$. Those rearrangements exist, since 
$\sum_{n=0}^{\infty}c_n$ is convergent but not absolutely. Define
$f_2(m\pi+s)=f_1(n\pi+s)$ for $0<s\leqslant\pi$ if $d_m=c_n$. Similarly define
$f_3(m\pi+s)=f_1(n\pi+s)$ for $0<s\leqslant\pi$ if $e_m=c_n$. 
Then $f_1$, $f_2$, and $f_3$ have an identical $\tau$-measure $\tau$.  
We have
\[
\int_0^{\infty-} f_2(s)ds=\sum_{m=0}^{\infty} d_m=\infty \quad
\text{and}\quad \int_0^{\infty-} f_3(s)ds=\sum_{m=0}^{\infty} e_m=0.
\]
Notice that $\int_0^{\infty} f_j(s)^2 ds=\int_{\mathbb{R}} u^2\tau(du)<\infty$
for $j=1,2,3$. Now consider 
$\mu=\mu_{(A,\nu,\gamma)}$ with $\nu$ symmetric and $\gamma\neq0$.
Use Theorems \ref{t4.1} and \ref{t4.5}. Then we see that $\mu$ 
belongs to $\mathfrak D(\Phi_{f_1})$ and $\mathfrak D(\Phi_{f_3})$, but not to 
$\mathfrak D(\Phi_{f_2})$; $\Phi_{f_1}(\mu)$ and  $\Phi_{f_3}(\mu)$ have a common
Gaussian part and a common L\'evy measure, but the location parameter of
$\Phi_{f_1}(\mu)$ equals $(\pi/2)\gamma$ and that of $\Phi_{f_3}(\mu)$
equals $0$.
\end{proof}

\begin{prop}\label{pt.2}
Suppose that $f_1(s)$ and $f_2(s)$ are measurable functions on $(a_1,b_1)$
and $(a_2,b_2)$, respectively, with identical $\tau$-measure $\tau$ and that
$\tau(\mathbb{R}\setminus\{0\})<\infty$ and $\int_{\mathbb{R}}u^2\tau(du)<\infty$. 
Then,
\begin{equation}\label{t.5}
\mathfrak D^0(\Phi_{f_1})=\mathfrak D^0(\Phi_{f_2})=ID(\mathbb{R}^d)
\end{equation}
and
\begin{equation}\label{t.6}
\Phi_{f_1}(\mu)=\Phi_{f_2}(\mu)\quad\text{for all }\mu\in ID(\mathbb{R}^d).
\end{equation}
\end{prop}

\begin{proof}
Since
\[
\int_{a_j}^{b_j} 1_{\{f_j(s)\neq0\}}ds=\tau(\mathbb{R}\setminus\{0\})\quad
\text{and}\quad \int_{a_j}^{b_j} f_j(s)^2 ds=\int_{\mathbb{R}}u^2\tau(du),
\]
we can apply Theorem \ref{t6.1} to show \eqref{t.5}.
We obtain \eqref{t.6} from \eqref{t.2} of Theorem \ref{tt.1}.
\end{proof}

\begin{ex}\label{et.1}
Let $\tau$ be a measure on $\mathbb{R}$ with $\tau(\mathbb{R})=b<\infty$ and let 
$g(u)=\tau((u,\infty))$. Assume that $g(u)$ is continuous and strictly decreasing
from $b$ to $0$ as $u$ moves from $-\infty$ to $\infty$.   Let $u=f(s)$, 
$s\in(0,b)$, be the
inverse function of $s=g(u)$.  Then $f(s)$ is
continuous and strictly decreasing from $\infty$ to $-\infty$ as $s$ moves
from $0$ to $b$. The measure $\tau$ is recovered as the $\tau$-measure of
$f$, since
\[
\tau((u_1,u_2])=g(u_1)-g(u_2)=\int_0^b 1_{[g(u_2),g(u_1))}(s)ds
=\int_0^b 1_{(u_1,u_2]}(f(s))ds
\]
for $u_1<u_2$.

For example, let $\tau$ be standard Gaussian distribution on $\mathbb{R}$. Then
$u=f(s)$, $s\in(0,1)$, is the inverse function of 
$s=g(u)=(2\pi)^{-1/2}\int_u^{\infty} e^{-v^2/2}dv$, $u\in\mathbb{R}$, and 
Proposition \ref{pt.2} applies. These $\tau$ and $f$ are considered by
Aoyama and Maejima (2007).  They show that the range 
$\{ \Phi_f(\mu)\colon \mu\in ID(\mathbb{R}^d)\}$ 
is the class of multivariate type $G$ distributions introduced by
Maejima and Rosi\'nski (2002).
\end{ex}

\begin{ex}\label{et.2}
Let $\tau$ be a measure on $(0,\infty)$ with total mass $b\leqslant\infty$ such that 
$g(u)=\tau((u,\infty))$, $u>0$, is finite, continuous, and strictly decreasing
from $b$ to $0$ as $u$ moves from $0$ to $\infty$.   Let $u=f(s)$, 
$s\in(0,b)$, be the
inverse function of $s=g(u)$.  Then $f(s)$ is
continuous and strictly decreasing from $\infty$ to $0$ as $s$ moves
from $0$ to $b$. The measure $\tau$ is the $\tau$-measure of
$f$. The pairs $g_{\alpha}(u)$ and $f_{\alpha}(s)$ with $\alpha\in\mathbb{R}$ in
Example \ref{e6.3} are special cases; in particular, if $\alpha<0$, then
the $\tau$-measure is $\Gamma$-distribution (with parameters $-\alpha$, $1$) 
multiplied by $\Gamma(-\alpha)$ and Proposition \ref{pt.2} applies. 
For another example, if $\tau$ is Mittag-Leffler distribution with
parameter $\alpha\in(0,1)$ (see Example 24.12 of Sato (1999)), then it satisfies
the condition above and Proposition \ref{pt.2} again applies ($\tau$ has
finite moments of all orders as is shown in p.\,74 of 
Barndorff-Nielsen and Thorbj\o rnsen (2006b)),
the corresponding $\Phi_f$ 
is studied by Barndorff-Nielsen and Thorbj\o rnsen (2006a) 
with notation $\Upsilon^{\alpha}$.
\end{ex}

We introduce some conditions on $f$ and $\tau$.

\begin{defn}\label{dt.2}
We say that a function $f(s)$ on $(a,b)$ 
satisfies {\it Condition} $(A)$ if $f(s)$ is
decreasing, left-continuous, not constant, and
\begin{equation}\label{t.7}
\inf_{t\in(a,b)} f(t)<f(s)<\sup_{t\in(a,b)} f(t)\quad\text{for all }
s\in(a,b).
\end{equation}
\end{defn}

The left-continuity in Condition $(A)$ is inessential in the 
following sense. If $f(s)$ satisfies Condition $(A)$ except the
left-continuity requirement, then the left-continuous modification
$f^-(s)$ defined by $f^-(s)=f(s-)$ satisfies Condition $(A)$ and,
for all $\mu\in ID(\mathbb{R}^d)$,
\[
\int_p^q f(s)X^{(\mu)} (ds)=\int_p^q f^-(s)X^{(\mu)} (ds)
\quad\text{whenever }a<p<q<b,
\]
because $f(s)=f^-(s)$ except for at most countably many $s$.

\begin{defn}\label{d8.3}
We say that a measure $\tau$ on $\mathbb{R}$ satisfies {\it Condition} $(B)$ 
if $\tau$ is not identically zero and if, for $a'=\inf\mathrm{Supp}(\tau)$ 
and $b'=\sup\mathrm{Supp}(\tau)$, $\tau$ has the following properties:
{\allowdisplaybreaks
\begin{gather}
a'<b',\label{t.8}\\
\tau((p,q))<\infty\quad\text{whenever }a'<p<q<b',\label{t.9}\\
\text{either $a'=-\infty$, or $a'>-\infty$ and $\tau(\{a'\})=0$},
\label{t.10}\\
\text{either $b'=\infty$, or $b'<\infty$ and $\tau(\{b'\})=0$},
\label{t.11}
\end{gather}}
\end{defn}

\begin{thm}\label{tt.2}
{\rm(i)} Let $-\infty\leqslant a<b\leqslant\infty$. If $f(s)$ is a 
function on $(a,b)$ satisfying Condition $(A)$, then the 
$\tau$-measure $\tau$ of $f$ satisfies Condition $(B)$.

{\rm(ii)} If $\tau$ is a measure on $\mathbb{R}$ satisfying Condition $(B)$,
then there are an interval $(a,b)$ with $-\infty\leqslant a<b\leqslant\infty$
and a function $f(s)$ on $(a,b)$ satisfying Condition $(A)$ 
such that $\tau$ is the $\tau$-measure of $f$.
\end{thm}

Assertion (i) of this theorem is straightforward from the definitions of
Conditions $(A)$ and $(B)$. In order to show (ii), we prepare a lemma,
which is an extension of Lemma 7.1 of Meyer (1962).

\begin{lem}\label{lt.1}
Let $-\infty\leqslant A'<B'\leqslant-\infty$ and let $G(u)$ be an $\mathbb{R}$-valued,
increasing, right-continuous function on $(A',B')$ which is not a
constant function. Let
\begin{equation}\label{t.12}
A=\inf_{u\in(A',B')} G(u)\quad\text{and}\quad 
B=\sup_{u\in(A',B')} G(u).
\end{equation}
Define
\begin{equation}\label{t.13}
F(s)=\inf\{u\in(A',B')\colon G(u)>s\}\quad\text{for }s\in(A,B).
\end{equation}
Then $F(s)$ takes values in $(A',B')$ and is increasing and 
right-continuous and
\begin{equation}\label{t.14}
G(u)=\inf\{s\in(A,B)\colon F(s)>u\}\text{ $($with $\inf\emptyset=B)$\quad for 
$u\in(A',B')$}.
\end{equation}
Moreover, for any nonnegative measurable function $h(u)$ on $(A',B')$,
\begin{equation}\label{t.15}
\int_{(A',B')} h(u)dG(u)=\int_{(A,B)} h(F(s))ds.
\end{equation}
\end{lem}

\begin{proof}
It is clear that $F(s)$ is $(A',B')$-valued and increasing. For
$s\in(A,B)$ we have $F(s)\leqslant F(s+)$. If $F(s)< F(s+)$,
then there is $r$ such that $F(s)<r<F(t)$ for all $t\in(s,B)$ and thus
$G(r)>s$ and $G(r)\leqslant t$ for all $t\in(s,B)$, which is impossible.
Hence $F(s)$ is right-continuous.  Extend $F(s)$ to $\widetilde F(s)$
on $[A,B]$ by
\[
\widetilde F(s)=\inf\{u\in(A',B')\colon G(u)>s\}\text{ $($with $\inf\emptyset=B')$}.
\]
Then $\widetilde F(s)$ is $[A',B']$-valued, increasing, and right-continuous.
For any $u\in(A',B')$, 
\[
\widetilde F(G(u))=\inf\{v\in(A',B')\colon G(v)>G(u)\}\geqslant u
\]
and hence $\widetilde F(G(u+\varepsilon))\geqslant u+\varepsilon$ for all small 
$\varepsilon>0$, from which
it follows that $\inf\{s\in[A,B]\colon \widetilde F(s)>u\}$ is 
$\leqslant G(u+\varepsilon)$, hence $\leqslant G(u)$.  We now have
\begin{equation}\label{t.16}
G(u)=\inf\{s\in[A,B]\colon \widetilde F(s)>u\}\text{\quad for 
$u\in(A',B')$}
\end{equation}
because, if not, there is $r$ such that $G(u)>r>
\inf\{s\in[A,B]\colon \widetilde F(s)>u\}$ and thus $\widetilde F(r)>u$ showing that
$G(u)\leqslant r$, which is absurd.  The right-hand side of
\eqref{t.14} equals the right-hand side of
\eqref{t.16} for all $u\in(A',B')$ because, if not, then for some
$u\in(A',B')$ and $t\in(A,B)$
\[
\inf\{s\in(A,B)\colon F(s)>u\} >t>\inf\{s\in[A,B]\colon \widetilde F(s)>u\},
\]
which implies $F(t)\leqslant u$ and $F(t)>u$, a contradiction.
 Therefore \eqref{t.14} is true.  For any $v\in(A',B')$,
\[
\int_{(A',B')} 1_{(A',v]}(u) dG(u)=\int_{(A',v]}dG(u)=G(v)-G(A'+)
=G(v)-A
\]
and
{\allowdisplaybreaks
\begin{align*}
&\int_{(A,B)} 1_{(A',v]}(F(s))ds=\int_{(A,B)}1_{\{F(s)\leqslant v\}} ds\\
&\qquad=\inf\{s\in(A,B)\colon F(s)>v\}-A=G(v)-A
\end{align*}
(note that if} $F(s)>v$ for all $s\in(A,B)$, then $A=G(v)>-\infty$
by \eqref{t.14}). Hence \eqref{t.15} holds for $h=1_{(A',v]}$ with
$v\in(A',B')$.  Therefore \eqref{t.15} holds for general $h$.
(In this lemma the roles of $G(u)$ and $F(s)$ are not symmetric, as we do
not necessarily have $A'=\inf_{s\in(A,B)} F(s)$ and
$B'=\sup_{s\in(A,B)} F(s)$.)
\end{proof}

\begin{proof}[Proof of Theorem \ref{tt.2} {\rm(ii)}]
Suppose that $\tau$ satisfies Condition $(B)$. Fix a point $c\in(a',b')$
and define $A'=a'$, $B'=b'$, and
\[
G(u)=\begin{cases} \tau((c,u]) & \qquad\text{for }c<u<B',\\
0 & \qquad\text{for }u=c,\\
-\tau((u,c]) & \qquad\text{for }A'<u<c.
\end{cases}
\]
Then $G(u)$ is finite, increasing, right-continuous, and not constant.
Now we apply Lemma \ref{lt.1}. Let $A$, $B$, and $F(s)$ on $s\in(A,B)$
be defined by \eqref{t.12} and \eqref{t.13}. Then $A=-\tau((A',c])$,
$B=\tau((c,B'))$, and $F(s)$ is $(A',B')$-valued, increasing, and
right-continuous.  Let $a=-B$, $b=-A$, and $f(s)=F(-s)$ for 
$s\in(a,b)$.  Then $f(s)$ is $(A',B')$-valued, decreasing, and
left-continuous. We claim that $f(s)$ satisfies Condition $(A)$.
If $f(s)$ is constant, then, for some $u_0\in(A',B')$, $F(s)=u_0$
for all $s\in(A,B)$, and hence (use \eqref{t.13}) 
$G(u_0)\geqslant B=\tau((c,B'))$, which contradicts Condition $(B)$.
Thus $f(s)$ is not constant.  Let us show \eqref{t.7}, that is,
\begin{equation}\label{t.17}
\inf_{t\in(A,B)} F(t)< F(s) < \sup_{t\in(A,B)} F(t)\quad\text{for }
s\in(A,B).
\end{equation}
If $A'<u<B'$, then $A<G(u)<B$ by Condition $(B)$. We have
$G(A'+)=A$ and $G(B'-)=B$.  Hence, using \eqref{t.13}, we obtain
$F(A+)=A'$ and $F(B-)=B'$.  Thus \eqref{t.17} follows, since
$A'<F(s)<B'$ for $s\in(A,B)$. Hence $f(s)$ satisfies Condition 
$(A)$.  For any nonnegative measurable $h$,
{\allowdisplaybreaks
\begin{align*}
&\int_{(a,b)} h(f(s))ds=\int_{(-B,-A)}h(F(-s))ds=\int_{(A,B)}h(F(s))ds\\
&\qquad =\int_{(A',B')}h(u)dG(u)=\int_{(a',b')} h(u)\tau(du)=\int_{\mathbb{R}}
h(u)\tau(du)
\end{align*} 
by virtue of \eqref{t.15} and Condition} $(B)$.  Therefore $\tau$ is
the $\tau$-measure of $f$.
\end{proof}

\begin{ex}\label{et.3}
The pairs of $f$ and $\tau$ in Examples \ref{et.1} and \ref{et.2} satisfy
Conditions $(A)$ and $(B)$. If $\tau$ is the probability measure with distribution
function equal to Cantor function, then $\tau$ satisfies Condition $(B)$ and
the function $f$ associated in Theorem \ref{tt.2} with Condition $(A)$
increases only with jumps.
\end{ex}

\section{Transformations of infinitely divisible distributions\\
on proper cone}

A subset $K$ of $\mathbb{R}^d$ is called a cone if it is a nonempty closed convex set
such that (1) $x\in K$ and $\alpha\geqslant0$ imply $\alpha x\in K$ and 
(2) $K\not=\{0\}$.
If, moreover, $x\in K$ implies $-x\not\in K$, then $K$ is called a proper cone.
  For example, $K$ is a proper cone in $\mathbb{R}$ if and only if $K$ is either
$[0,\infty)$ or $(-\infty,0]$.
In $\mathbb{R}^2$, $K$ is a proper cone which is nondegenerate (that is, not
contained in any one-dimensional linear subspace), if and only if there 
are linearly independent $x^{(1)}$ and $x^{(2)}$ such that $K=\{
\alpha_1x^{(1)}+\alpha_2x^{(2)}\colon \alpha_1\geqslant0\text{ and }\alpha_2
\geqslant0\}$.
In $\mathbb{R}^3$, there are many proper cones such as triangular cones and
circular cones.  If $K$ is the set of $(x_j)_{1\leqslant j\leqslant 3}$ such that
\[
\begin{pmatrix}
x_1 & x_3\\
x_3 & x_2
\end{pmatrix}
\]
is nonnegative-definite, then $K$ is linearly isomorphic to a 
circular cone in $\mathbb{R}^3$; see Pedersen and Sato (2003).

In this section let $K$ be a proper cone in $\mathbb{R}^d$.
Let $ID(K)$ denote the class of infinitely divisible
distributions supported on $K$ (that is, $\mathrm{Supp}(\mu)\subset K$).

\begin{prop}\label{p7.1}
Let $\mu=\mu_{(A,\mu,\gamma)}$ be in $ID(\mathbb{R}^d)$. Then $\mu\in ID(K)$ if
and only if $\mu$ is of type A or B, $\mathrm{Supp}(\nu)\subset K$, 
and the drift $\gamma^0$ is in $K$. 
\end{prop}

This proposition  is given in Skorohod (1986) and also in E\,22.11 
of Sato (1999).

Let $\mathrm{Lvm}(ID(K))$ denote the class of L\'evy measures of infinitely 
divisible distributions supported on $K$. That is, $\mathrm{Lvm}(ID(K))$ is
the class of measures $\nu$ on $\mathbb{R}^d$
satisfying $\nu(\{0\})=0$, $\mathrm{Supp}(\nu)\subset K$,  and 
$\int_{K}(|x|\land1)\nu(dx)<\infty$. 

Now let $-\infty\leqslant a<b\leqslant\infty$.
Let $f$ be a nonnegative measurable function on $(a,b)$. The
following propositions are the counterparts of Theorem \ref{p7.2}
and Theorems \ref{t4.8}, \ref{t6.2}, \ref{t6.3}, and \ref{t6.4}.

\begin{prop}\label{p7.2a}
Let $\mu\in ID(K)$.
Let $f(s)$ be a nonnegative, $\mathbb{R}$-valued measurable function on $(a,b)$.
Then the following two statements are equivalent.

{\rm(a)} $f\in\mathbf L_{(a,b)}(X^{(\mu)})$ and
\begin{equation}\label{7.2a}
\mathcal L\left(\int_p^q f(s)X^{(\mu)}(ds)\right) \in ID(K)\quad
\text{for all $p,q$ with $a<p<q<b$}.
\end{equation}

{\rm(b)} The L\'evy measure $\nu$ and the drift $\gamma^0$ of $\mu$ satisfy
\eqref{7.1b} and \eqref{7.1c}.
\end{prop}

\begin{proof} Assume (a). Then we have (b), using Theorem \ref{p7.2}.

Conversely assume (b).  Theorem \ref{p7.2} says that 
$f\in\mathbf L_{(a,b)}(X^{(\mu)})$ and that \eqref{7.1a} is true. Let 
$(A,\nu,\gamma^0)_0$ and 
$(A_p^q,\nu_p^q, (\gamma^0)_p^q)_0$ be the triplets of $\mu$ and 
$\int_p^q f(s)X^{(\mu)}(ds)$, respectively. Then 
$A_p^q=0$ from \eqref{7.1a} and, using Proposition \ref{p7.1},
we obtain $\mathrm{Supp}(\nu)\subset K$ from \eqref{3.9}
and $(\gamma^0)_p^q\in K$ from \eqref{7.1d}
since $f$ is nonnegative and $K$ is a cone.
Hence \eqref{7.2a} is true. 
\end{proof}

\begin{prop}\label{p7.3}
Let $\mu\in ID(K)$.
Let $f\in\mathbf L_{(a,b)}(X^{(\mu)})$ and $f\geqslant0$.  Then
$\Phi_f(\mu)$ is definable and $\Phi_f(\mu)\in ID(K)$ if and only if 
statement {\rm (b)} of Theorem \ref{t4.8} is true.
\end{prop}

\begin{proof} Use Theorem \ref{t4.8}.
The \lq\lq only if" part is because $ID(K)\subset ID_{\mathrm{AB}}$.
For the \lq\lq if" part, use Proposition \ref{p7.1} and assumption $f\geqslant0$.
\end{proof}

\begin{prop}\label{p7.4}
Let $f(s)$ be a nonnegative, $\mathbb{R}$-valued measurable function on $(a,b)$. 
Then the following statements are equivalent.

{\rm(a)}\quad $ID(K)\subset\mathfrak D ^0(\Phi_f)$ 
and $\Phi_f(ID(K))\subset ID(K)$.

{\rm(b)}\quad $ID(K)\subset\mathfrak D (\Phi_{f,\mathrm{es}})$
 and, for any $\mu\in ID(K)$,  $\Phi_{f,\mathrm{es}}(\mu)\cap ID(K)\neq \emptyset$.

{\rm(c)}\quad\,$\int_a^b 1_{\{f(s)\neq0\}} ds
<\infty$ and $\int_a^b f(s)ds<\infty$.
\end{prop}

\begin{proof}
(c) $\Rightarrow$ (a): Using Theorem \ref{t6.2}, we have
$ID(K)\subset ID_{\mathrm{AB}}(\mathbb{R}^d)\subset \mathfrak D ^0(\Phi_f)$ 
and $\Phi_f(ID(K))\subset\Phi_f(ID_{\mathrm{AB}}(\mathbb{R}^d))\subset 
ID_{\mathrm{AB}}(\mathbb{R}^d)$.
For any $\mu\in ID(K)$, $\widetilde\mu=\Phi_f(\mu)$ is not only in $ID_{\mathrm{AB}}$
but also in $ID(K)$ because $f\geqslant0$ (use \eqref{4.6} and \eqref{7.6}).

(a) $\Rightarrow$ (b): Obvious since $\mathfrak D ^0(\Phi_f)\subset \mathfrak D 
(\Phi_{f,\mathrm{es}})$.

(b) $\Rightarrow$ (c): We have, from Theorem \ref{t4.6},
\begin{equation}\label{7.5a}
\int_a^b ds\int_K (|f(s)x|\land1)\,\nu(dx)<\infty\quad\text{for all $\nu
\in\mathrm{Lvm}(ID(K))$}.
\end{equation}
Hence the proof is similar to that of the corresponding part of Theorem 
\ref{t6.2} (replace the unit sphere $S$ by $K\cap S$).
\end{proof}

\begin{prop}\label{p7.5}
Let $f(s)$ be a nonnegative, $\mathbb{R}$-valued measurable function on $(a,b)$. 
Then $ID(K)\subset\mathfrak D (\Phi_{f,\mathrm{es}})$ if and only if statement {\rm(b)}
of Theorem \ref{t6.3} is true.
\end{prop}

\begin{proof}
The \lq\lq if" part: Obvious from Theorem \ref{t6.3}.

The \lq\lq only if" part:  It follows that
\begin{equation}\label{7.6a}
\int_a^b ds\int_K (|f(s)x|^2\land1)\,\nu(dx)<\infty\quad\text{for all }
\nu\in\mathrm{Lvm}(ID(K)).
\end{equation}
Hence we can make a discussion similar to the corresponding part of the 
proof of Theorem \ref{t6.3}.
\end{proof}

\begin{prop}\label{p7.6}
Let $f(s)$ be a nonnegative, $\mathbb{R}$-valued measurable function on $(a,b)$. 
Then $ID(K)\subset\mathfrak D^0 (\Phi_f)$ if and only if statement {\rm(b)}
of Theorem \ref{t6.4} is true.
\end{prop}

\begin{proof}
The \lq\lq if" part follows from Theorem \ref{t6.4}.
The \lq\lq only if" part is
similar to the corresponding part of the 
proof of Theorem \ref{t6.4}.
\end{proof}

\section{Transformations of L\'evy measures}

Fix $-\infty\leqslant a<b\leqslant\infty$ and the dimension $d$.
Let $f(s)$ be an $\mathbb{R}$-valued measurable 
function on $(a,b)$.

\begin{defn}\label{d8.1}
For a measure $\nu$ on $\mathbb{R}^d$, let $\nu^{\sharp}$ be the measure given by
\begin{equation}\label{8.1}
\nu^{\sharp}(B)=\int_a^b ds\int_{\mathbb{R}^d} 1_B(f(s)x)\nu(dx),\qquad
B\in\mathcal B (\mathbb{R}^d).
\end{equation}
If $\nu\in\mathrm{Lvm}(ID(\mathbb{R}^d))$ and if 
$[\nu^{\sharp}]_{\mathbb{R}^d\setminus\{0\}}$,
the restriction of $\nu^{\sharp}$ to $\mathbb{R}^d\setminus\{0\}$, is 
in \linebreak 
$\mathrm{Lvm}(ID(\mathbb{R}^d))$, then define
\begin{equation}\label{8.2}
\Psi_f(\nu)=[\nu^{\sharp}]_{\mathbb{R}^d\setminus\{0\}}.
\end{equation}
The domain $\mathfrak D(\Psi_f)$ of $\Psi_f$ is defined as
\begin{equation}\label{8.3}
\mathfrak D(\Psi_f)=\{\nu\in\mathrm{Lvm}(ID(\mathbb{R}^d))\colon 
[\nu^{\sharp}]_{\mathbb{R}^d\setminus\{0\}}\in\mathrm{Lvm}(ID(\mathbb{R}^d))\}.
\end{equation}
\end{defn}

Since 
\begin{equation}\label{8.3a}
\int_{\mathbb{R}^d\setminus\{0\}} h(x)\nu^{\sharp}(dx)=\int_a^b ds
\int_{\mathbb{R}^d}h(f(s)x)1_{\mathbb{R}^d\setminus\{0\}}(f(s)x)\,\nu(dx)
\end{equation}
for all nonnegative measurable functions $h$ on $\mathbb{R}^d$,
a measure $\nu$ on $\mathbb{R}^d$ belongs to $\mathfrak D(\Psi_f)$ if and only 
if $\nu(\{0\})=0$, $\int_{\mathbb{R}^d}(|x|^2\land1)\,\nu(dx)<\infty$,
and
\begin{equation}\label{8.4}
\int_a^b ds\int_{\mathbb{R}^d}(|f(s)x|^2\land1)\,\nu(dx)<\infty.
\end{equation}

The following theorem shows a close relationship between the
transformation $\Psi_f$ and the essential improper integrals.

\begin{thm}\label{t8.1}
Assume  that $f$ is locally square-integrable on $(a,b)$. Let
$\nu$ and $\widetilde\nu$ be in $\mathrm{Lvm}(ID(\mathbb{R}^d))$. Then the 
following statements are equivalent.

{\rm(a)}\quad$\nu\in\mathfrak D(\Psi_f)$ and $\widetilde\nu=\Psi_f(\nu)$.

{\rm(b)}\quad If $\mu\in ID_0(\mathbb{R}^d)$ has L\'evy measure
$\nu$, then $\mu$ belongs to $\mathfrak D(\Phi_{f,\mathrm{es}})$ and any 
distribution in
$\Phi_{f,\mathrm{es}}(\mu)$ has L\'evy measure $\widetilde\nu$.
\end{thm}

\begin{proof}
Use Remark \ref{r3.2} and Theorems \ref{t4.2} and \ref{t4.6}.
The assumption of local square-integrability of $f$ on $(a,b)$
is needed because the theorems in Section 3 presupposes that
$f\in\mathbf L_{(a,b)}(X^{(\mu)})$.
\end{proof}

We give some results similar to those in Sections 6 and 8.

\begin{prop}\label{p8.1}
The following statements are equivalent.

{\rm(a)}\quad$\mathrm{Lvm}(ID(\mathbb{R}^d))=\mathfrak D(\Psi_f)$.

{\rm(b)}\quad$\int_a^b 1_{\{f(s)\neq0\}} ds
<\infty$ and $\int_a^b f(s)^2 ds<\infty$.
\end{prop}

\begin{proof}
Statement (a) is equivalent to saying that any 
$\nu\in\mathrm{Lvm}(ID(\mathbb{R}^d))$
satisfies \eqref{8.4}. If (b) is true, then 
(a) follows from 
Theorem \ref{t6.1}. Conversely, if (a) is true, then 
we obtain (b) exactly in the same way as in Steps 1 and 2 
in the proof of 
Theorem \ref{t6.1} (here we cannot directly apply this theorem,
as we do not assume $f\in\mathbf L_{(a,b)}(X^{(\mu)})$ and we 
cannot use Theorem \ref{t4.2} in order to say that (a) implies
$ID_0\subset\mathfrak D(\Phi_{f,\mathrm{es}})$).
\end{proof}

\begin{prop}\label{p8.2}
The following statements are equivalent.

{\rm(a)}\quad$\mathrm{Lvm}(ID_{\mathrm{AB}}(\mathbb{R}^d))\subset\mathfrak D(\Psi_f)$
and $\Psi_f(\mathrm{Lvm}(ID_{\mathrm{AB}}(\mathbb{R}^d)))\subset
\mathrm{Lvm}(ID_{\mathrm{AB}}(\mathbb{R}^d))$.

{\rm(b)}\quad$\int_a^b 1_{\{f(s)\neq0\}} ds
<\infty$ and $\int_a^b |f(s)| ds<\infty$.
\end{prop}

\begin{proof}
Theorem \ref{t6.2} tells that statement (b) is equivalent to condition (b) of
Theorem \ref{t6.2}.  Hence, if statement (b) is true, then
statement (a) follows. Conversely, if statement (a) is true, then 
we obtain statement (b) as Steps 1 and 2 
in the proof of Theorem \ref{t6.2} work.
\end{proof}

\begin{prop}\label{p8.3}
The following statements are equivalent.

{\rm(a)}\quad$\mathrm{Lvm}(ID_{\mathrm{AB}}(\mathbb{R}^d))\subset\mathfrak D(\Psi_f)$.

{\rm(b)}\quad$\int_a^b 1_{\{f(s)\neq0\}} ds
<\infty$ and \eqref{6.4} and \eqref{6.5} hold.
\end{prop}

\begin{proof}
If (b) is true, then (a) follows, as the proof of \eqref{6.6} in the first
half of the proof of Theorem \ref{t6.3} works.
If (a) is true, then (b) is true as in Steps 1, 2, and 3 of the second half
of the proof of Theorem \ref{t6.3}.
\end{proof}

\begin{prop}\label{p8.4}
The following statements are equivalent.

{\rm(a)}\quad$\mathfrak D(\Psi_f)$ consists only of the zero measure.

{\rm(b)}\quad$\int_a^b (f(s)^2\land1) ds=\infty$.
\end{prop}

\begin{proof}
Almost the same proof as that of Theorem \ref{t6.5} works.
\end{proof}

\begin{prop}\label{p8.5}
Let $K$ be a proper cone in $\mathbb{R}^d$. Assume that $f(s)$ is nonnegative.
Then
\begin{equation}\label{8.4a}
\mathrm{Lvm}(ID(K))\subset\mathfrak D(\Psi_f)\quad{and}\quad
\Psi_f(\mathrm{Lvm}(ID(K)))\subset
\mathrm{Lvm}(ID(K))
\end{equation}
 if and only if statement {\rm(c)} of Proposition
\ref{p7.4} is true.
\end{prop}

\begin{proof}
If (c) of Proposition
\ref{p7.4} is true, then (b) of Proposition
\ref{p7.4} is true and \eqref{8.4a} holds.
Conversely, if \eqref{8.4a} holds, then (c) of Proposition
\ref{p7.4} is true by an argument similar to the corresponding part 
of the proof of Theorem \ref{t6.2}.
\end{proof}

\begin{prop}\label{p8.6}
Let $K$ be a proper cone in $\mathbb{R}^d$. Assume that $f(s)\geqslant0$.
Then
\begin{equation}\label{8.4b}
\mathrm{Lvm}(ID(K))\subset\mathfrak D(\Psi_f)
\end{equation}
if and only if statement {\rm(b)} of Proposition \ref{p8.3} is true.
\end{prop}

\begin{proof}
The inclusion \eqref{8.4b} is equivalent to \eqref{7.6a}. Hence the
proof is similar to that of Proposition \ref{p8.3}.
\end{proof}

Using the $\tau$-measure $\tau$ of $f$ introduced in Section 7, 
we can express the transformation $\Psi_f$ as in the following 
proposition. If we restrict our attention to measures $\tau$ on
$(0,\infty)$, the transformation $\Psi_f$ in this form is
identical with the transformation discussed by Barndorff-Nielsen
and P\'erez-Abreu (2005) under the name of Upsilon transformation
$\Upsilon_{\tau}$; it is called a generalized Upsilon transformation in
Barndorff-Nielsen and Thorbj\o rnsen (2006b). Their studies are
being made in Barndorff-Nielsen and Maejima (2007) and 
Barndorff-Nielsen, Rosi\'nski, and Thorbj\o rnsen (2007).

\begin{prop}\label{p8.7}
Let  $\tau$ be the $\tau$-measure of  $f$. Then,
$\nu\in\mathfrak D(\Psi_f)$ if and only if $\nu$ is in 
$\mathrm{Lvm}(ID(\mathbb{R}^d))$
and satisfies
\begin{equation}\label{8.8}
\int_{\mathbb{R}}\tau(du)\int_{\mathbb{R}^d} (|ux|^2\land1)\,\nu(dx)<\infty.
\end{equation}
If $\nu\in\mathfrak D(\Psi_f)$, then, for 
$B\in\mathcal B_{\mathbb{R}^d\setminus\{0\}}$,
\begin{equation}\label{8.9}
(\Psi_f(\nu))(B)=\int_{\mathbb{R}} \tau(du)\int_{\mathbb{R}^d} 1_B(ux)\,\nu(dx)
=\int_{\mathbb{R}\setminus\{0\}}\nu\left(\frac{1}{u}B\right)\tau(du).
\end{equation}
\end{prop}

\begin{proof}
Immediate from \eqref{8.7a}, \eqref{8.4}, and Definition \ref{d8.1}.
\end{proof}

\end{document}